\documentclass[11pt,reqno]{amsart}

\usepackage{bbm}
\usepackage{bm}
\usepackage{tikz,pgfplots}
\usepgfplotslibrary{groupplots}
\usetikzlibrary{pgfplots.groupplots}
\usetikzlibrary{decorations.pathreplacing,calligraphy}
\pgfplotsset{compat=1.18}
\usepackage{amsthm}
\usepackage{amsmath,amssymb,amsfonts}
\usepackage{graphicx}
\usepackage{subfigure}
\usepackage{graphics}
\usepackage{cancel,accents}
\usepackage{bigints}
\usepackage{hyperref}
\usepackage{cleveref}
\usepackage[left=2.55cm, right=2.55cm, bottom=3cm]{geometry}
\newtheorem{theorem}{Theorem}[section]
\newtheorem{lemma}[theorem]{Lemma}
\newtheorem{defn}[theorem]{Definition}
\newtheorem{assum}[theorem]{Assumption}

\newtheorem{cor}[theorem]{Corollary}

\newtheorem{conj}[theorem]{Conjecture}

\title{A convergent front tracking scheme}

\author{Manas Bhatnagar and Robin Young}
\thanks{Department of Mathematics, University of
Massachusetts Boston, manas.bhatnagar@umb.edu}
\thanks {Department of Mathematics and Statistics, University of
Massachusetts Amherst, young@math.umass.edu}

\def\mc{\mathcal}
\def\com#1{\quad\text{#1}\quad}
\def\qcom#1{\qquad\text{#1}\qquad}
\def\wint{\star\!\!\!\int}
\def\B{\mathbb}
\def\e{\epsilon}
\def\ee#1{{\e^{e_#1}}}
\def\g{\gamma}
\def\z{\zeta}
\def\la{\lambda}
\def\al{\alpha}

\def\T{\intercal}
\def\wt{\widetilde}
\def\wh{\widehat}
\def\whh#1{\wh{\vphantom{\rule{14pt}{7pt}}\smash{\wh{#1}}}}

\def\ol{\overline}
\def\ul{\underline}
\def\dd{\partial}

\def\fs#1{{\cancel #1}}

\def\dot{\accentset{\mbox{\Large\bfseries .}}}
\def\({\begin{pmatrix}}
\def\){\end{pmatrix}}

\def\I#1#2#3{I_#1(p_#2,p_#3)}

\def\<{\langle}
\def\>{\rangle}
\newcommand{\RN}[1]{\textup{\uppercase\expandafter{\romannumeral#1}}}

\date{}
\numberwithin{equation}{section}

\definecolor{bakShock}{RGB}{0,0,255}
\definecolor{bakSimple}{RGB}{100,140,255}
\definecolor{composite}{RGB}{165,42,42}
\definecolor{multiRf}{RGB}{176,176,176}
\definecolor{contact}{RGB}{191,147,100}
\definecolor{extra}{RGB}{44,160,44}
\definecolor{fwdShock}{RGB}{255,0,0}
\definecolor{fwdSimple}{RGB}{255,140,100}
\definecolor{grey127}{RGB}{127,127,127}
\definecolor{direct}{RGB}{148,103,189}

\begin{document}

\begin{abstract}
We present a modified Front Tracking (mFT) scheme for hyperbolic
systems of conservation laws in one space dimension, in which we allow
arbitrarily large nonlinear waves.  We build the scheme by introducing
and solving a ``generalized Riemann Problem'', which yields exact
solutions for finite times.  This allows us to treat the states
adjacent to all waves exactly, and approximate compressive
simple waves in addition to rarefactions, contacts and shocks.  In
particular, we require exact expression of the various wave curves and
avoid the use of Taylor expansions.  After construction of the scheme,
under reasonable assumptions, we show that the mFT approximations
converge to a weak* solution of the system.  This essentially reduces
existence of solutions with large amplitude data to obtaining uniform
bounds on the total variation of the approximations.

We then apply the scheme to the Euler equations of gas dynamics, for
which we exactly solve the generalized Riemann Problem and define the
scheme for both 3x3 and 2x2 systems, and prove the equivalence of
Eulerian and Lagrangian frames.  For the $p$-system, modeling
isentropic gas dynamics in a Lagrangian frame, we show that there is
no finite accumulation of interaction times.  This means that the last
remaining obstacle to global existence of large data, large amplitude
solutions is the construction of a decreasing Glimm potential.
\end{abstract}

\maketitle


\section{Introduction}

We consider $N\times N$ systems of hyperbolic conservation
laws in one space variable, namely
\begin{equation}
  \label{eqsys}
  q(u)_t + f(u)_x = 0,\qquad (x,t)\in\mathbb{R}\times (0,T),
\end{equation}
up to fixed time $T>0$ (which may be infinite), with Cauchy data
\begin{equation}
  \label{u0}
  u(\cdot,0) = u^o(\cdot)\in L^\infty(\mathbb R).
\end{equation}
Here
$u=u(x,t)\in\mc U\subset\mathbb{R}^N$ is the \emph{state}, and $q$
and $f:\mc U\to\mathbb{R}^N$ are $C^2$ functions giving the (vectors
of) \emph{conserved quantities}, and corresponding \emph{fluxes},
respectively.  As is well known, analysis of hyperbolic systems is
complicated by the formation of nonlinear \emph{shock waves}, which
are discontinuities in the function $u(\cdot,t)$, and their subsequent
propagation~\cite{CF, Lax, Lax64}.

There is now a well established and largely complete theory of small
$BV$ solutions to \eqref{eqsys} going back to Glimm's foundational
existence theorem, including existence, uniqueness and well-posedness
results~\cite{G,S,Yth,Bre,Bressan, BCM10,CKV22}.  For some systems,
including the protoypical case of gas dynamics, the theory extends to
solutions having large total variation but small
oscillation~\cite{GL,Ba70, Nishida, NS, TY, ty2}, while recent results
include the existence of periodic solutions~\cite{TYperSnd}.

Many modern approaches are based on Front Tracking (FT)
approximations, which generate piecewise constant approximations and
in turn provide a simplified but effective setting for the nonlinear
analysis~\cite{Diperna,BJ,Holden,Bre,Dafermos}.  However, current FT
theory relies on a small $BV$ assumption and is unsatisfactory for
physically realistic systems, because shock waves are inherently very
strong nonlinear phenomena.  Ideally, mathematical models should allow
for large amplitude waves and solutions, which are ruled out by the
essential limitation $\|u(\cdot,t)\|_{\infty}<\epsilon$ made in most FT
approaches to \eqref{eqsys}.

Our goal in this paper is to define a modified FT (mFT) scheme for
general systems \eqref{eqsys} which is appropriate for waves and
solutions of arbitrarily large finite amplitude and locally finite
variation.  We also show that, under an \emph{a priori} assumption of
a finite $BV$ bound, together with some mild technical and checkable
assumptions on the structure of the system, the mFT approximations
converge to a weak* solution of \eqref{eqsys}.  In the forthcoming
paper~\cite{BY2}, we will apply this mFT scheme to prove the global
existence of large amplitude, large $BV$ weak* solutions to the
$p$-system of isentropic gas dynamics by obtaining the assumed global
$BV$ bound of the approximate solutions.

In developing and implementing the mFT scheme, we strictly apply the
following guiding principles:
\begin{itemize}
\item For all nonlinear waves, adjacent states defining \emph{all
    waves} must be treated \emph{exactly};
\item To correctly resolve interactions of strong waves,
  \emph{compressions} must be treated \emph{explicitly};
\item For finite times, the \emph{number of waves} in the
  approximation must remain \emph{finite};
\item For general systems, the scheme should be as \emph{flexible} as
  possible.
\end{itemize}
In particular, adherence to these principles allows us to avoid the
use of non-physical waves, which are necessary for most legacy
implementations of FT for systems with $N\ge3$ equations.

We assume that system \eqref{eqsys} is \emph{strictly hyperbolic}, so
that the generalized eigenvalue problem generated by the Jacobian
matrices of $q$ and $f$, namely
\[
  Df(u)\,r_i(u) = \lambda_i(u)\,Dq(u)\,r_i(u),
\]
has a full set of real eigenvalues
\[
  \lambda_1(u) < \lambda_2(u) < \ldots < \lambda_N(u), \qquad u\in\mc U.
\]
This implies that the corresponding eigenvectors,
$r_1(u), \ldots, r_N(u)$, are linearly independent for each
$u\in\mc U$.  In addition, we assume that each field is either
\emph{genuinely nonlinear (GNL)} or \emph{linearly degenerate},
\[
  r_k(u)\cdot\nabla\la_k(u) > 0, \com{or}
  r_k(u)\cdot\nabla\la_k(u) = 0,
\]
for all $u\in\mc U$, respectively.

Generally, physically relevant systems are also \emph{symmetrizable},
which means that for some choice of state variable $u\in\mc U$, both
$Dq$ and $Df$ are symmetric, and the system has a natural convex
entropy~\cite{FL,Laxbook,Serre}.  If this holds, $Dq$ has a Cholesky
factorization,
\[
  Dq(u) = \sqrt{Dq(u)}^\T\sqrt{Dq(u)},
\]
and the generalized eigenvectors $\{r_k\}$ satisfying \eqref{ev}
form an orthonormal frame relative to $Dq$,
\begin{equation}
  \label{ortho}
  r_j^\T(u)\,Dq(u)\,r_k(u) = \delta_{jk}
  = \big(\sqrt{Dq(u)}\,r_j(u)\big)^\T\big(\sqrt{Dq(u)}\,r_k(u)\big),  
\end{equation}
uniformly for $u\in\mc U$.  Existence of a convex entropy provides a
nonlocal condition for choosing admissible shocks~\cite{Laxbook,
  Wendroff, FL, TPLadm, LR03}, although for this paper we use a natural
extension of Lax's characteristic condition~\cite{Lax}.   Convex
entropies also form the basis of the method of compensated
compactness, which yields $L^\infty$ solutions of some
systems~\cite{DiPcc, LPT, Chen}.  However, these come with little
quantitative information and we prefer the local wave-based approach
which yields the detailed local structure of solutions~\cite{Dip}.

The \emph{Riemann problem} is the basis of most analyses of solutions
of conservation laws in one space dimension, being both experimentally
realizable in one dimension and respecting the scaling invariance of
\eqref{eqsys}.  Solutions to the Riemann problem also describe the
asymptotic structure of solutions with compactly supported data for
large times, and so can in some sense be regarded as ``universal''.
However, using Riemann solutions solely as a basis for analysis of
interactions is risky because they cannot adequately capture the fine
structure of solutions at ``intermediate'' times, at which the
variation of the solution may grow, say
\[
  \|u(\cdot,t)\|_{BV} >
  C\,\max\big\{\|u^o\|_{BV},\;\|u(\cdot,\infty)\|_{BV}\big\},
  \qquad t = O(1).
\]

As a way of understanding interactions of large waves and coping with
this intermediate growth, we introduce a \emph{generalized Riemann
  problem} (gRP), which admits compressions as well as rarefactions
and shocks, and captures some of the intermediate growth of nonlinear
interacting solutions.  The gRP data is realizable \emph{a
  posteriori}, and we do not consider the solution globally, but only
for some finite time until secondary interaction effects are
manifested.  This allows us to treat states exactly and incorporate
compressions, at least for short times.  More importantly, use of gRP
solutions naturally introduces the \emph{time} that nonlinear
interactions take to complete, effectively ``slowing down'' multiple
interactions and allowing us to bound the number of waves in our
approximate solutions.

Because we are treating general systems \eqref{eqsys}, we
\emph{assume} that the gRP is globally and uniquely solvable, although
this can be readily checked for well known systems such isentropic
and full gas dynamics and the nonlinear elastic string~\cite{S,Ystr}.  
We also \emph{assume} that $BV$ bounds are known; of course obtaining
such $BV$ bounds is well known to be the main obstacle in proofs of
existence and regularity of solutions.  Under some other technical
assumptions on the structure of the system, we then prove that the mFT
scheme we are defining converges to a weak* solution of system
\eqref{eqsys}.

Use of weak* solutions, introduced in \cite{MY1,MY2}, provides an
effective method for defining a \emph{residual} for approximate
solutions, and which are easily calculated for piecewise constant
approximations.  The residual is defined as a measure and can be
manipulated using the corresponding calculus.  Also, the residual can
be defined for any bounded function $u(x,t)$, and use of the norm on
the space of measures allows us to describe how far any bounded
function is from being a solution.

We build the piecewise constant mFT scheme by discretizing the gRP,
subject to the following constraints.  First, we insist that the
states adjacent to any wave must be \emph{exact}, so there is no
choice in assigning states and we never solve approximate Riemann
problems.  However, to distinguish between shocks and simple waves, we
assign to each wave a non-negative \emph{``virtual width''}.  This width
chooses whether or not a shock or simple wave should be used in
solving the gRP; a wave is a shock if and only if it is both
compressive and has zero width, and it is simple whenever its width is
positive.  Linearly degenerate waves are always given zero width.

Because states are assumed exact and gRPs are solved exactly, we have
no flexibility in assigning wave strengths or adjacent states, but we
can control the scheme by appropriately choosing the speed of waves
and their virtual widths.  We are working with general systems, so we
introduce several modifications of widths and speeds in order to keep
the residual small and the number of waves finite.

Strong rarefactions cannot be sufficiently well approximated by a
small number jump discontinuities, and so we introduce a method of
splitting strong rarefactions into several smaller waves, called
\emph{``multi-rarefactions''}, which keeps the residual under control.
However, we must also respect upper and lower bounds for the strengths
of split rarefactions, in order to avoid too many waves.  A more
immediate problem for general systems of size $N\ge 3$ is to control
the number of waves: there are several known examples in which waves
can accumulate so that infinitely many waves are generated in finite
time.  We control this problem by introducing \emph{``composite
  waves''}, in which weak waves, known as passengers, can
``piggyback'' on stronger waves, so that the total number of waves in
the mFT approximation remains controlled \emph{unless} some sort of
blowup occurs.

Definition of the scheme leads to a sequence of ``interaction times''
\[
  0 = t_0 < t_1 < t_2 < \dots < t_i < \dots,
\]
and between each pair of interaction times, waves are represented by
jump discontinuities with fixed wavespeed.  We represent these by a
\emph{``wave sequence''}
\[
  \Gamma^t = \big\{ \g_j\;:\; j = 1,\dots,n^t\big\}, \qquad
  t\in (t_i,t_{i+1}),
\]
where each wave $\g_j$ carries its own information, such as strength
$\z_j$, family, adjacent states, speed, width, etc.
Using this notation, the mFT scheme can be regarded as a natural
extension of the ``Method of Reorderings'', developed by the second
author in~\cite{Yth, TY}.

We present a method of initializing the scheme, which generates the
wave sequence $\Gamma^{0+}$ from $BV$ initial data \eqref{u0}, as well
as a reconstruction method, which extrapolates a given wave sequence
$\Gamma^t$ to generate a $BV$ approximation of the solution at time
$t$.  This reconstructed solution is not piecewise constant, but
provides profiles to the simple waves, consistent with the solution of
the gRP.

Our first theorem shows that the scheme can be defined inductively for
all interaction times $t_i$, and the corresponding approximation
remains in $BV$.  Moreover the piecewise constant approximation is in
the correct class for the residual to be well-defined.    

\begin{theorem}[Theorem \ref{thm:hatU}]
  For each $\e>0$ and $i\in\B N$, the wave sequence
  $\Gamma^{t_i+}$ has bounded variation, that is
  \[
    \mc V\big(\Gamma^{t_i+}\big) := \sum|\z_j|<\infty.
  \]
  Moreover, for each $i$, the corresponding piecewise constant
  approximation
  \[
    U^\e \in W^{1,\infty}_*\big(0,t_{i+1};BV_{loc},M_{loc}\big),
  \]
  lives in the correct function space $\mc Z$ given by \eqref{space},
  so that
  \[
    f(U^\e)\in L^\infty_*(0,t_{i+1};BV_{loc}) \qcom{and}
    q(U^\e)\in W^{1,\infty}_*(0,t_{i+1};BV_{loc},M_{loc}).
  \]
\end{theorem}

While we state the theorem for $BV$ solutions, our results also
hold in the wider class of $H^{-b}$ solutions, $b>1/2$.  We expect
that this wider class will extend to more space dimensions, for which
the class $BV$ is unstable to evolution by \eqref{eqsys}~\cite{Rauch}.

Although the scheme can be defined for all interaction times $t_i$,
these interaction times may accumulate at some finite time $t_\infty$,
possibly due to some form of blowup or simply because the number of
waves becomes infinite due to self-similarity.  If this happens, we
say we have an \emph{accumulation of interaction times}.  We introduce
a checkable condition, which we call the ``$\al$-ray condition'', which
is weaker than the requirement that the approximation $U^\e(x,t)$ be
$BV$ along rays, and which effectively characterizes the conditions
under which the scheme cannot be continued.

\begin{theorem}[Theorem \ref{thm:noaccum}]
  For fixed $\e>0$, suppose that the mFT approximation $U^\e$
  satisfies the $\al$-ray condition for all $(x,t)$, with
  $t\le t_\infty$.  Then there is no accumulation of interaction times
  and the scheme can be continued for all times, that is
  $t_\infty=\infty$.
\end{theorem}

Finally, we turn to convergence of the scheme.  For this we
\emph{assume} uniform $BV$ bounds on the approximation for time
$t\le T$.  We also introduce a technical condition that the ``weight
of heavy passengers'' $W_H$ be small: this can be interpreted as an
assumption that the different wave families effectively separate so
that comparably sized waves of different families do not continuously
interact with each other.  Under these conditions, we show that the
mFT scheme for general systems converges to a weak* solution as
$\e\to0^+$.

\begin{theorem}[Theorem \ref{thm:conv}]
  Suppose the mFT scheme is defined for $t\le T$ with the uniform
  variation bound
  \[
    \mc V(\Gamma^t) < V, \com{for} 0\le t\le T.
  \]
  Suppose also that $W_H(\Gamma,\e) = o(1)$ as $\e\to 0^+$, uniformly
  for $t\in[0,T]$.  Then some subsequence of the corresponding
  approximations $U^\e$ converge as $\e\to0^+$, and the limit is a
  weak* solution of the system \eqref{eqsys}.
\end{theorem}

Since our scheme is defined for large amplitude waves, these theorems
show that in order to prove a large amplitude $BV$ existence theorem
for a general nonlinear hyperbolic system, it suffices to obtain
uniform $BV$ bounds.  In particular, the authors expect that checking
the $\al$-ray and heavy passenger conditions will be essential steps
in deriving $BV$ bounds for such systems.

The most important and well-known system of conservation laws is the
inviscid Euler equations of compressible gas dynamics.  As an
application, we give a detailed description of the mFT for that
system.  Because the system consists of three families and the middle
(contact) family is linearly degenerate, we do not expect self-similar
patterns to develop, so we choose \emph{not to use} composite waves in
the scheme.  We derive exact expressions for the wave curves and for
the residual of each wave.  We then show that the mFT schemes and
their convergence properties are consistent between the spatial Euler
frame and the material Lagrangian frame, so it is sufficient to
consider only the Lagrangian frame, in which the wavespeeds are
simplified.  Finally, we restrict to the $p$-system of isentropic gas
dynamics, and we show that if a (nonlocal) Glimm potential $\mc P$ for
the total variation can be found, then the scheme is defined for all
times.

\begin{theorem}[Theorem \ref{thm:noacc}]
  Fix $\e>0$ and for some $T>0$, suppose that a potential $\mc P$ can
  be found for the $p$-system, such that
  \[
    \Delta\,\mc P\le 0 \qcom{and} \Delta(\mc V+\mc P) \le 0,
  \]
  for interactions occuring up to time $t<T$.  Then the scheme can be
  defined beyond time $T$; that is, there is some $k\in\B N$ such that
  the $k$-th interaction time satisfies $t_k>T$.
\end{theorem}

This theorem rules out wave patterns in which the variation becomes
unbounded, such as those in examples of \cite{BJ1} and \cite{BCZ18}.
Together with our earlier theorems, the theorem implies that if such a
potential can be found for large but finite amplitudes and variations,
then existence of $BV$ weak* solutions is implied for times
$t\in[0,T]$.  In the forthcoming paper \cite{BY2}, the authors
construct such a Glimm potential for a large class of initial data,
provided the vacuum is avoided for $t< T$.  We remark that this
potential is nonlocal, so is not inconsistent with~\cite{ChenJens}.
As a corollary, assuming the existence of the potential, we obtain
existence of solutions and convergence of the scheme to a weak*
solution \emph{with a rate} $\e^{2e_s}$ as $\e\to0^+$, uniformly for
$t\le T$.

The paper is laid out as follows: in Section~\ref{sec:w*}, we recall
the notion of weak* solution to \eqref{eqsys} and describe the
residual of piecewise constant approximations.  In
Section~\ref{sec:rp} we recall the Riemann problem and define and
solve the generalized Riemann problem (gRP).  Then in
Section~\ref{sec:mft} we define the modified Front Tracking (mFT)
approximations.  In Section~\ref{sec:ir} we describe the
initialization of the scheme and reconstruction of solutions using the
gRP.  In Section~\ref{sec:continue}, we show that the scheme can be
defined indefinitely, and in Section~\ref{sec:conv} we prove
convergence of a subsequence of the mFT approximations to a weak*
solution, assuming \emph{a priori} $BV$ bounds.  Finally in
Section~\ref{sec:gd} we specialize to the case of the Euler equations
of compressible gas dynamics, and in Section~\ref{sec:psys} we treat
the isentropic case, and show that the example of~\cite{BCZ18} cannot
be realized.

\section{Weak* Solutions}
\label{sec:w*}

We begin by recalling the notion of weak* solution, introduced in
~\cite{MY1,MY2}.  This is a modification (practically slight, but
philosophically significant) of the well-known notion of weak
solution~\cite{Lax,S,Dafermos}.  Recall that $u=u(x,t)$ is a weak
solution to the Cauchy problem \eqref{eqsys}, \eqref{u0} if it is a
bounded distributional solution; that is,
$u(x,t)\in L^\infty\big(\mathbb R\times[0,T)\big)$ is a weak solution
if
\begin{equation}
  \label{weak}
  \int_0^T\!\!\! \int_{\mathbb{R}}
  q\big(u(\fs x,\fs t)\big)\, \psi_t(\fs x,\fs t) +
  f\big(u(\fs x,\fs t)\big)\, \psi_x(\fs x,\fs t)\; d\fs x\;d\fs t +
  \int_\mathbb{R} u^o(\fs x)\, \psi(\fs x,0)\; d\fs x = 0,
\end{equation}
for all compactly supported smooth test functions
$\psi\in C_c^\infty ([0,T)\times \mathbb{R})$.

In contrast, weak* solutions are defined based on the observation that
the action \eqref{weak} of $u$ on test functions is \emph{linear}, so
that solutions should be correctly interpreted as living in the dual
of the Banach space of test functions.  With this interpretation, the
PDE system \eqref{eqsys} can be re-interpreted as an ODE in the Banach
space $X^*$, which is the dual of the space of test functions $X$,
namely
\begin{equation}
  \label{w*ode}
  q(u)' + \mathbb Df(u) = 0, \qquad u(0) = u^o,
\end{equation}
in which $\mathbb D$ is the spatial distributional derivative, and
$'=\frac d{dt}$ is the Gelfand-weak derivative of a Banach space
valued function $[0,T)\to X^*$.  This ODE can be integrated directly
to get the equivalent condition
\begin{equation}
  \label{w*int}
  q\big(u(\cdot,t)\big) - q\big(u^o\big)
  + \wint_0^t\B Df\big(u(\cdot,\fs t)\big)\;d\fs t = 0,
\end{equation}
in $X^*$, for all $t\in [0,T)$.  Here $\wint$ is the \emph{Gelfand
  integral}, defined by the condition
\[
  \Big\langle\wint_E g(\fs t)\;d\fs t,\varphi\Big\rangle
  = \int_E\big\langle g(\fs t),\varphi\big\rangle\;d\fs t \com{for all}
  \varphi\in X,
\]
where $E\subset [0,T)$ is Borel.  Note that the integral on the right
is just a Lebesgue integral, and this definition immediately yields an
interchange of limits (or Fubini theorem).  Thus \eqref{w*int} can be
written as
\begin{equation}
  \label{w*phi}
  \left\langle q\big(u(\cdot,t)\big),\varphi\right\rangle -
  \left\langle q\big(u^o\big),\varphi\right\rangle -
  \int_0^t\left\langle f\big(u(\cdot,\fs
    t)\big),\frac{d\varphi}{dx}\right\rangle\;d\fs t = 0,
\end{equation}
for all $t\in [0,T)$ and test functions $\varphi \in X$, where we have
integrated by parts inside the integral.  For this definition to make
sense, note that $q(u)$ must take values in $X^*$ and $f(u)$ takes
values in ${X^\dag}^*$, where $X^\dag$ are derivatives of functions in
$X$.

In \cite{MY1}, it is shown that when restricting to locally bounded
$BV_{loc}$ solutions, weak and weak* solutions are equivalent.  Due to
their flexibility in choice of test functions, and the ability to use
calculus in the space of measures, in this paper we consider only
weak* solutions.

More precisely, we define $BV$ weak* solutions of \eqref{eqsys} as
follows.  Note that for fixed compact set
$K = K_{\mc U}\subset \mc U$, $BV(K)$ is the dual of a space, which we
write as $BV(K) = \big(C^{-1}(K)\big)^*$, although care must be taken
in defining the separable space $C^{-1}(K)$; see~\cite{AmbFusPal}.
This notation is consistent with the fact that derivatives of $BV$
functions are Radon measures, which are in turn the dual of continuous
functions.  Since the compact set $K$ is arbitrary, and all our
approximations are constant outside a compact set, we drop explicit
mention of $K$ and write
\[
  BV_{loc} = (C^{-1}_0)^* \com{and} M_{loc} = (C^0_0)^*.
\]
Here the functions and corresponding measures are implicitly taken to
have values in $\mathbb R^N$ with the Euclidean norm; this is no
restriction because the target space is finite dimensional.
Throughout the paper, we use the convention that $BV$ functions are
the average of their one-sided limits,
\begin{equation}
  \label{bvm}
  u(x) = \frac{u(x+)+u(x-)}2, \com{for all} x.
\end{equation}

Recalling that $f$ and $q$ are $C^2$, we require that $u\in BV_{loc}$,
so that both $q\circ u$ and $f\circ u$ are also in
$BV_{loc}\subset M_{loc}$.  Next, since $\mathbb Df\in M_{loc}$, we
check \eqref{w*int} in $M_{loc}$, so we apply continuous test
functions.  As in \cite{MY1,MY2}, we let
\begin{equation}
  \label{dualSp}
  \begin{gathered}
    L^p_*(0,T;BV_{loc})= L^p_*\big(0,T;(C^{-1}_0)^*\big) \com{and}\\
    W^{1,p}_*(0,T;BV_{loc},M_{loc}) =
    W^{1,p}_*\big(0,T;(C^{-1}_0)^*,(C^0_0)^*\big)
  \end{gathered}
\end{equation}
denote the Gelfand $p$-integrable and absolutely continuous functions,
respectively.  Here, functions in $W^{1,p}_*(0,T;{X^\dag}^*,X^*)$ take
their values in ${X^\dag}^*(=BV_{loc})$ but are Gelfand differentiable in
the larger space $X^*(=M_{loc})$.  Note that $C^0_0$ is separable as
required, and it suffices to take $p=\infty$, because we are assuming
global $BV$ bounds.

\begin{defn}
  \label{def:w*}
  The function
  \begin{equation}
  \label{space}
    u\in \mc Z:= f^{-1}\big(L^\infty_*(0,T;BV_{loc})\big) \cap
    q^{-1}\big(W^{1,\infty}_*(0,T;BV_{loc},M_{loc})\big)
  \end{equation}
  is a BV-weak* solution to \eqref{eqsys} if \eqref{w*ode} holds in
  $M_{loc}$ for a.e.~$t\in[0,T)$, or equivalently, if \eqref{w*int}
  holds in $M_{loc}$ for every $t\in[0,T)$, that is \eqref{w*phi}
  holds for all continuous test functions $\varphi$.
\end{defn}

Here $u$ can be taken to be a norm-measurable representative, and
since the integral is absolutely continuous, \eqref{w*int} holds at
all times; see~\cite{MY1,MY2} for details.  Note that $\mc Z$ is
\emph{not} a Banach space, but as the intersection of preimages of
Banach spaces, is a complete metric space with a naturally endowed
metric.

We can similarly define weak* solutions in Sobolev spaces $H^{-b}$,
$b\in\B R$.  In this case we require that \eqref{w*int} holds in
$H^{-b}=(H^b)^*$, or equivalently that \eqref{w*phi} holds for all
$\varphi\in H^b$.  Making sense of this requires that
$f\circ u\in H^{1-b}=(H^{b-1})^*$ is Gelfand integrable, and that
$q\circ u$ lives in $H^{1-b}$ and is absolutely continuous in
$H^{-b}$.  Thus a Gelfand $p$-integrable $H^{-b}$ weak* solution is a
function
\begin{equation}
  \label{Hbspace}
  u \in \mc Z^{-b,p} := f^{-1}\big(L^p_*(0,T;H^{1-b})\big) \cap
  q^{-1}\big(W^{1,p}_*(0,T;H^{1-b},H^{-b})\big)
\end{equation}
for which \eqref{w*int} holds in $H^{-b}$.  

\subsection{Residual of Front Tracking Approximations}

A $BV$ weak* solution satisfies \eqref{w*ode} or \eqref{w*int} as an
equality in $M_{loc}$.  Although this can be checked by considering
\eqref{w*phi} for any test function $\varphi\in C^0_0$, it is easier
to simply check \eqref{w*ode} in the space $M_{loc}$ of measures.
That is, because $f\circ u$ is also $BV_{loc}$, the weak derivative
$\B Df(u)$ can be directly regarded as a measure.  Indeed, for any
function $u\in \mc Z$, where $\mc Z$ is the metric space given by
\eqref{space}, the right hand side of \eqref{w*ode} makes sense, and
so we can define a residual on $\mc Z$.

\begin{defn}
  \label{def:resid}
  For any $u\in \mc Z$, the \emph{residual} of $u$ is defined to be the
  time-dependent measure
  \begin{equation}
    \label{resid}
    \mc R(u)(t) := q\big(u(\cdot,t)\big)'
    + \B Df\big(u(\cdot,t)\big)\in M_{loc}.
  \end{equation}
  Similarly, if $u \in \mc Z^{-b,p}$, we define the residual
  $\mc R(u)\in L^p_*(0,T;H^{-b})$ in the same way.
\end{defn}

This ability to define the residual is a powerful distinction between
weak and weak* solutions: indeed, the residual can in principle be
calculated for any candidate $u\in \mc Z$, and it is clear that $u$ is
a weak* solution if and only if the residual $\mc R(u)$ vanishes for
almost every $t\in[0,T)$.  Moreover, in this context we can use the
calculus of measures to simplify and understand the residual because
we have a well understood topology to deal with.  In particular, in
this paper we are developing an approximation scheme for which the
residuals are bounded by the discretization parameter, and so it is
unsurprising that some subsequence of these will converge to a weak*
solution.

Front Tracking approximations are piecewise constant functions that
are built to approximate solutions of \eqref{eqsys}, and the class of
piecewise constant functions is particularly well suited to the
calculation of residuals as measures.  In particular, in a FT
approximation, each discontinuity, which represents a wave or part
thereof, propagates with a given, fixed speed for almost all times.

Thus suppose we are given some $U\in \mc Z$ which is piecewise constant
and absolutely continuous for $t$ in some open interval $(t_a,t_b)$.
That is, say $U(\cdot,t)$ has $m$ jumps, each located at
$\xi_i(t)$.  For $t\in(t_a,t_b)$, we can represent $U$ as
\[
  U(x,t) =
  \begin{cases}
    U_0(t), &x<\xi_1(t),\\
    U_i(t), &\xi_i(t)<x<\xi_{i+1}(t),\\
    U_m(t), &\xi_m(t)<x,
  \end{cases}
\]
and each of $U_i$ and $\xi_i$ are differentiable on $(t_a,t_b)$.
For each $i$, we define the jump
\[
  [U]_i(t) := U\big(\xi_i(t)+,t\big) - U\big(\xi_i(t)-,t\big)
  = U_i(t) - U_{i-1}(t).
\]
If we  let $H$ denote the Heaviside function, using the convention that $H(0)=1/2$, we can rewrite 
\begin{equation}
  \label{pwc}
  U(\cdot,t) = U_0(t) + \sum_{i=1}^m[U]_i(t)\,H\big(x-\xi_i(t)\big),
  \qquad t\in(t_a,t_b).
\end{equation}

\begin{lemma}
  \label{lem:res}
  Suppose that $U\in \mc Z$ is piecewise constant and has the form
  \eqref{pwc}, with each state $U_i(t)=U_i$ constant for
  $t\in(t_a,t_b)$.  Then for such $t$, the residual has the form
  \begin{equation}
    \label{pwcres}
    \mc R(U) =
    \sum_{i=1}^m\big([f]_i-\xi_i'(t)\,[q]_i\big)\,\delta_{\xi_i} \in M_{loc},
  \end{equation}
  where $[q]_i$ and $[f]_i$ are the jumps in $q(u)$ and $f(u)$ at
  $\xi_i(t)$, respectively,
  \begin{equation}
    \label{qfjump}
    [q]_i(t) := q(U_i) - q(U_{i-1}), \com{and}
    [f]_i(t) := f(U_i) - f(U_{i-1}).
  \end{equation}
\end{lemma}

\begin{proof}
Since $U(\cdot,t)$ is $BV$, we can simply plug this in to
\eqref{resid} and evaluate it as a measure.  It follows immediately
that
\[
  \begin{aligned}
    q\big(U(\cdot,t)\big)
    &= q\big(U_0(t)\big) +
      \sum_{i=1}^m[q]_i(t)\,H\big(x-\xi_i(t)\big),\\
    f\big(U(\cdot,t)\big)
    &= f\big(U_0(t)\big) + \sum_{i=1}^m[f]_i(t)\,H\big(x-\xi_i(t)\big),   
  \end{aligned}
\]
where the jumps in $q$ and $f$ are given by \eqref{qfjump}.
Plugging these into \eqref{resid} immediately yields
\begin{equation}
  \label{pwcR}
  \mc R(U)(t) = B(t) +
  \sum_{i=1}^m\big([f]_i-\xi_i'\,[q]_i\big)\,\delta_{\xi_i} \in M_{loc},
\end{equation}
where $\delta_\xi$ is the Dirac mass, and $B(t)$ is the bounded
function
\[
  B(t) := q\big(U_0(t)\big)' +
      \sum_{i=1}^m[q]_i'(t)\,H\big(x-\xi_i(t)\big).
\]
If each state $U_i$ is constant for $t\in(t_a,t_b)$, it follows that
$B(t)$ vanishes identically there.
\end{proof}

Note that when we work in the space $\mc Z^{-b,p}$, the Lemma follows
as above, the only difference being that we should choose $b$ large
enough that the Dirac mass $\delta\in H^{-b}$, namely $b>1/2$, and in
that case we can again take $p=\infty$.

\section{The Riemann Problem and Extensions}
\label{sec:rp}

We briefly recall the solution of the Riemann problem~\cite{Lax,S}.
As is well known, this is the Cauchy problem in which the Cauchy data
consists of two constant states, namely
\[
  u^o(x) =
  \begin{cases}
    u_L, &x<0, \\
    u_R, &x>0.
  \end{cases}
\]
Note that this data respects the scale invariance of system
\eqref{eqsys}.  In one space dimension, the solution of the Riemann
problem has two major applications: it represents the asymptotic limit
of general solutions with fixed states $u(\pm\infty,t)$, and it serves
as the main building block of more general solutions, including FT
approximations.

For strictly hyperbolic systems, the solution of the Riemann problem
is given by $N+1$ constant states separated by \emph{elementary
  centered waves}.  For genuinely nonlinear (GNL) fields, these
elementary waves are either \emph{shocks} or centered
\emph{rarefactions}, while for linearly degenerate fields, they are
\emph{contact discontinuities}.  For fields that are not either
genuinely nonlinear or linearly degenerate, centered waves are
composites of rarefactions and (degenerate) \emph{contact shocks}, but
for ease of discussion we will not consider such families here.

We begin by describing a single shock or contact, which is a sharp
discontinuity separating two constant states.  These take the form
\eqref{pwc} with a single jump, so have the form
\[
  u =
  \begin{cases}
    u_-, &x<\sigma\,t,\\
    u_+, &x>\sigma\,t,
  \end{cases}
\]
where the discontinuity propagates along the ray $x=\sigma\,t$, which
is forced by self-similarity.  If this is to be an exact solution of
the system, the residual must vanish, so by Lemma~\ref{lem:res},
\begin{equation}
  \label{sc}
  [f] = \sigma\,[q], \com{or}
  f(u_+) - f(u_-) = \sigma\,\big(q(u_+)-q(u_-)\big),
\end{equation}
which is the \emph{Rankine-Hugoniot jump condition}.  Note that this
is $N$ equations in $2N+1$ variables, so for one state, say $u_-$
fixed, it gives a one-parameter family of solutions $(u_+,\sigma)$,
which describes a curve in state space.

This can be clarified by writing the Hugoniot condition as a
generalized eigenvalue problem, as follows.  We set
\[
  [u] := u_+-u_-, \quad \ol u = \frac{u_-+u_+}2, \com{and}
  u(\eta):=\ol u + \eta\,[u],
\]
and integrate along the segment joining $u_-$ and $u_+$, namely
$\big\{u(\eta)\;:\;\eta\in(-\frac12,\frac12)\big\}$, to get
\[
  \begin{aligned}
    f(u_+)-f(u_-)
    &= \int_{-1/2}^{1/2}Df\big(u(\fs\eta)\big)\;d\fs\eta\;[u]
      =:\<Df\>\,[u],\\
    q(u_+)-q(u_-)
    &= \int_{-1/2}^{1/2}Dq\big(u(\fs\eta)\big)\;d\fs\eta\;[u]
      =:\<Dq\>\,[u],
  \end{aligned}
\]
so that $\<Df\>$ and $\<Dq\>$ are averaged Jacobian matrices.  The
Hugoniot condition is thus simply a generalized eigenvalue problem,
namely
\[
  \<Df\>\,[u] = \sigma\,\<Dq\>\,[u],
\]
in which the matrices $\<Df\>$ and $\<Dq\>$ are nonlinear functions
of $\ol u$ and $[u]$.  In general we expect $N$ different shock or
contact families, each of which will yield a one-parameter family of
solutions for each reference state, which could be either of $u_-$,
$u_+$ or $\ol u$.  We will refer to these as shock or contact curves
according to whether the family is genuinely nonlinear or linearly
degenerate.

For genuinely nonlinear (GNL) families, the \emph{Lax entropy condition} is
used to rule out inadmissible discontinuities: this states that for a
$k$-shock, the left and right $k$-characteristics must impinge on the
shock, namely
\begin{equation}
  \label{Laxent}
  \lambda_k(u_-) > \sigma(u_-,u_+) > \lambda_k(u_+), \com{while also}
  \lambda_{k-1}(u_-) < \sigma(u_-,u_+) < \lambda_{k+1}(u_+),
\end{equation}
and this in turn rules out one half of the Hugoniot locus.  We will
denote the admissible half of the $k$-shock curve by
\[
  \mc S_k(u_-) := \big\{ u_+\;;\;
  [f] = \sigma_k\,[q],\;
  \lambda_k(u_-) > \sigma_k > \lambda_k(u_+)\big\},
\]
and we will write
\begin{equation}
  \label{Sk}
  u_+ \in \mc S_k(u_-) \com{and}
  u_- \in \mc S_k^\T(u_+),
\end{equation}
if left state $u_-$ can be joined to right state $u_+$ by an
admissible $k$-shock.

For strong shocks, we assume the following \emph{extension}
of Lax's entropy condition.  Given any $u_m\in\mc U$ and
$k<k'$, we require that
\begin{equation}
  \label{sig}
  \sigma_k(u_\ell,u_m) < \sigma_{k'}(u_m,u_r),
\end{equation}
for any $u_\ell\in\mc S_k^\T(u_m)$, and $u_r\in \mc S_{k'}(u_m)$.
Unlike Lax's condition, this compares shock speeds across
\emph{different} shock curves.  To the authors' knowledge, this
extension comparing shock speeds of different families has not been
previously formulated, and should extend consistently to more general
entropy conditions such as those given by \cite{Wendroff, TPLadm,
  LR03}.

We next describe \emph{simple waves}, which are $C^1$ solutions
of \eqref{eqsys}, and which include the centered rarefaction waves.  A
simple wave is a solution $u(x,t)$ of \eqref{eqsys} which has a
one-dimensional image in state space, and can thus be written
\[
  u(x,t) = U\big(\zeta(x,t)\big), \com{where}
  \zeta:\Omega\subset\B R\times[0,T)\to\B R, \com{and}
  U:\B R\to\mc U,
\]
see \cite{John,S,Ysurv}.  Plugging this in to \eqref{eqsys} or
\eqref{w*ode}, and using classical differentiability, we immediately
get
\[
  Dq(U)\,\frac{dU}{d\zeta}\,\zeta_t
  + Df(U)\,\frac{dU}{d\zeta}\,\zeta_x = 0.
\]
This PDE system can immediately be broken down into three simpler
problems, namely a generalized eigenvalue problem,
\begin{equation}
  \label{ev}
  Df(U)\,r_k(U) = \lambda_k(U)\,Dq(U)\,r_k(U),
\end{equation}
an autonomous ODE,
\begin{equation}
  \label{ODE}
  \frac{dU}{d\zeta} = r_k(U),
\end{equation}
and a scalar transport equation,
\[
  \zeta_t + \lambda_k\big(U(\zeta)\big)\,\zeta_x = 0.
\]
These are solved in the order in which they appear: namely, first, the
eigenvalue problem is solved for all $U\in \mc U$; next, the
(appropriately normalized) eigenvectors are integrated to get $N$
\emph{simple wave curves}, parameterized by $\zeta$; and finally, the
transport equation is solved as a scalar hyperbolic PDE.  Assuming
that the wave curves $U(\zeta)$ have been described, we solve by
characteristics: that is, set
\[
  \frac{dx}{dt} = \lambda_k\big(U(\zeta)\big), \com{so that}
  \frac{d\zeta}{dt} = 0,
\]
so that $\zeta$ is constant and the characteristics are straight
lines,
\begin{equation}
  \label{simple}
  x = x_0 + \lambda_k\big(U(\zeta)\big)\,(t-t_0).
\end{equation}
This solution is determined by Cauchy data
$\left\{\big(x_0(\zeta),t_0(\zeta)\big),U(\zeta)\right\}$, and extends
to the largest neighborhood of the plane for which different
characteristics do not focus or intersect.  For any fixed value of
$t$, the \emph{profile} of the simple wave is the shape of the wave,
which we can write as $\big(x(\zeta),U(\zeta)\big)$, with $x=x(\zeta)$
given by \eqref{simple}.

The simple wave is \emph{centered} if all characteristics
\eqref{simple} focus at a single \emph{center} $(x_c,t_c)$, and in
this case the simple wave consists of one half of the characteristic
cone through the center,
\[
  x-x_c = \lambda_k\big(U(\zeta)\big)\,(t-t_c).
\]
Here $U(\zeta)$ parameterizes the wave curve connecting states $u_-$
and $u_+$, and the centered simple wave is a rarefaction if
$\lambda_k(u_-)<\lambda_k(u_+)$, in which case the upper cone $t>t_c$
is used, while if $\lambda_k(u_-)>\lambda_k(u_+)$ the wave is a
\emph{centered compression} and the lower cone $t<t_c$ is used.

We now describe the \emph{simple wave curves} in more detail.  The
generalized eigenvalue problem yields $N$ eigenvector fields $r_k(U)$
defined throughout $\mc U$, and each such (smoothly varying) field is
integrated as the autonomous ODE \eqref{ODE}.  For each $k$, this
generates a flow on $\mc U$, and each orbit of this flow is a simple
wave curve.  We denote the $k$-wave curve through $u_0$ by
\begin{equation}
  \label{Wk}
  \mc W_k(u_0) = \Big\{ u(\zeta) : 
  u(\zeta) = W_k(u_0,\zeta)
  := u_0 + \int_0^\zeta r_k(u(\fs\zeta))\;d\fs\zeta \Big\},
\end{equation}
Because the ODE is autonomous, the functions $W_k$ have the local
group property
\begin{equation}
  \label{lgp}
  W_k(u_0,z_1+z_2) = W_k\big(W_k(u_0,z_1),z_2\big),
\end{equation}
as long as these expressions make sense, and in particular, the wave
curves are disjoint,
\[
\mc W_k(u_1) = \mc W_k(u_0) \com{whenever} u_1 \in \mc W_k(u_0).
\]
If the $k$-th field is genuinely nonlinear (GNL), for each fixed left state
$u_-$, the wave curve $\mc W_k(u_-)$ consists of two distinct pieces,
namely the \emph{rarefaction curve} of right states $u_+$ that can be
reached by a $k$-rarefaction, for which
$\lambda_k(u_-)<\lambda_k(u_+)$, and the \emph{compression curve} of
those $u_+$ that can be connected by a $k$-compression, and for which
$\lambda_k(u_+)<\lambda_k(u_-)$.  By convention, 
$r_k\cdot\nabla\lambda_k>0$, so that the rarefaction curve corresponds
to positive parameter $\zeta>0$, while compressions correspond to
$\zeta<0$.  This in turn is consistent with the description of the
$k$-shock curve above.

We put these curves together to describe the \emph{elementary centered
  wave curves}: the $k$-th such curve is the union of the $k$-shock
curve and the $k$-rarefaction curve,
\begin{equation}
  \label{Ek}
  \mc E_k(u_-) := \mc S_k(u_-) \cup
  \big\{W_k(u_-,\zeta)\;;\;\zeta\ge 0\big\},  
\end{equation}
and this consists of all right states $u_+$ that can be reached from
left state $u_-$ by an elementary centered $k$-wave.  Recall that for
a linearly degenerate family, the Hugoniot locus $\mc S_k$ and simple
wave curves $\mc W_k$ coincide, and the corresponding wavespeed
$\lambda_k$ is constant on that curve~\cite{Lax,S}.

We now describe the solution of the Riemann problem, which is well
known~\cite{Lax,S,Dafermos}.  Setting $u_0:=u_L$, we inductively take
\begin{equation}
  \label{Ekz}
  u_k := E_k(u_{k-1},\zeta_k) \in \mc E_k(u_{k-1}),
\end{equation}
where $E_k(u_-,\zeta)$ is a consistent parameterization of
$\mc E_k(u_-)$.  This yields a map $u_N$, given by
\[
  u_N(\zeta_1,\zeta_2,\dots,\zeta_N) =
  E_N\Big(E_{N-1}\big(\dots,E_1(u_0,\zeta_1),
  \dots,\zeta_{N-1}\big),\zeta_N\Big)\in\mc U,
\]
defined on some maximal set $Z\in\B R^N$ containing $u_L$, and on
which the map is invertible.  We now solve the equation
\[
  u_N(\zeta_1,\zeta_2,\dots,\zeta_N) = u_R,
\]
to get the strengths of the individual waves, and generate the Riemann
solution by connecting the intermediate states by the corresponding
centered elementary waves.

Existence and uniqueness of small amplitude solutions of the Riemann
problem was shown by Lax for general systems~\cite{Lax}, and global
solutions have been found for physically interesting systems
\cite{S,Ystr}, but because we are working with a general system,
existence (and uniqueness) of solutions of the Riemann problem is an
\emph{assumption}.  We will specify our main assumption more precisely
after generalizing the Riemann problem below.

\subsection{Generalized Riemann Problem (gRP)}

In our development of the modified Front Tracking scheme, we wish to
account for the widths of simple waves and understand the consequences
thereof for general interactions.  In order to consistently take this
into account, we will incorporate the explicit use of compressions
into the scheme.  We begin with the observation that providing a given
width to a GNL simple wave gives a well-defined way of centering the
wave.

\begin{lemma}
  \label{lem:center}
  Assume the $k$-th family is genuinely nonlinear, and suppose we are
  given states $u_-$ and $u_+\in\mc W_k(u_-)$, or equivalently $u_+$
  and $u_-\in\mc W_k(u_+)$.  Suppose also we are given a positive
  initial width $w$.  Then there is a unique profile $u^o=u^o(x)$,
  with
  \[
    u^o(x) = u_-,\ x<0, \qquad u^o(x) = u_+,\ x>w,
  \]
  and a unique center $(x_c,t_c)$, such that the Cauchy problem
  $u(x,0)=u^o(x)$ has a solution initially consisting of a simple
  centered $k$-wave.  If $t_c<0$, this simple wave is a rarefaction
  and is defined for all $t\in(t_c,\infty)$, while if $t_c>0$, the
  simple wave is a compression for $t\in(-\infty,t_c)$, and can be
  continued as a translated Riemann solution for all times $t\ge t_c$.
  Moreover, the $k$-th wavespeed $\lambda_k(u^o)$ is a monotonic
  function of $x$, increasing for a rarefaction and decreasing for a
  compression.

  The position $(x_c,t_c)$ of the center scales linearly with width
  $w$, and the profile of the data scales similarly: that is, if
  $u^o(x;w)$ denotes the profile determined by $w>0$, then for any
  $\alpha>0$,
  \[
    u^o(x;\alpha\,w) = u^o(x/\alpha;w).
  \]
\end{lemma}

\begin{proof}
  As functions of $t$, the left and right edges of the wave are
  \[
    x_-(t) = \lambda_k(u_-)\,t \com{and}
    x_+(t) = w + \lambda_k(u_+)\,t,
  \]
  respectively.  These intersect at the point
  \[
    t_c := \frac w{\lambda_k(u_-)-\lambda_k(u_+)}, \qquad
    x_c := \lambda_k(u_-)\,t_c = w + \lambda_k(u_+)\,t_c,
  \]
  which is the center of the simple wave.  Linear dependence of the
  center on width $w$ is evident.  Genuine nonlinearity implies that
  $t_c$ is finite and nonzero.  Since $w>0$, we have $t_c<0$ if and
  only if $\lambda_k(u_+)>\lambda_k(u_-)$, corresponding to a
  rarefaction, and $t_c<0$ if and only if
  $\lambda_k(u_+)<\lambda_k(u_-)$, corresponding to a compression.

  For clarity suppose that the wave is a compression, and parameterize
  $\mc W_k(u_-)$ by $\zeta\in[\zeta_+,0]$, where
  $u_+=W_k(u_-,\zeta_+)$ with $\zeta_+<0$ as in \eqref{Wk}.  We choose
  the profile of $u^o$ so that all characteristics in the simple wave
  focus at $(x_c,t_c)$, so we choose $x^o(\zeta)$ by setting
  \[
    x_c-x^o(\zeta) = \lambda_k\big(u(\zeta)\big)\,t_c, \com{and}
    u^o\big(x^o(\zeta)\big) := u(\zeta),
  \]
  for $\zeta\in[\zeta_+,0]$.  Since
  \[
    \frac{d}{d\zeta}\lambda_k = \nabla\lambda_k\cdot\frac{du}{d\zeta}
    = r_k\cdot\nabla\lambda_k> 0,
  \]
  uniqueness and monotonicity of the profile follow by the implicit
  function theorem.

  By construction, the characteristics intersect only at $(x_c,t_c)$, so
  the compression is defined for all $t<t_c$.  Moreover, the solution
  at time $t_c$ is exactly $u_-$ for $x<x_c$, $u_+$ for $x>x_c$, which
  is Riemann data translated to the focus $(x_c,t_c)$, and so
  extends to a Riemann solution for all $t\ge t_c$.  The case of a
  rarefaction is similar, but we do not attempt to extend the solution
  backwards beyond $t_c<0$.
\end{proof}

In the case of a compression, we refer to the time $t_c$ in the above
lemma as the \emph{time of collapse} of the compression; at that time
the compression will generally form a shock in the $k$-th family, with
associated waves in other families being reflected as a consequence of
nonlinearity.

Linearity in the width $w$ is a consequence of scale invariance of the
PDE: this is most evident if we use translation invariance to shift
the center of the simple wave to the origin.  For a rarefaction, this
is the `usual' centered rarefaction, while for a compression the
profile is chosen so that the wave focuses exactly at the origin and
the solution continued as a Riemann solution.

Recall that in our solution to the Riemann problem, we used three
types of waves, namely shocks and contacts, which have no width, and
rarefactions centered at the origin, which also have no width at
$t=0$.  We then pieced these waves together in increasing order of
wavespeed to get the Riemann solution for $t\ge 0$.

Following Lemma~\ref{lem:center}, we can enlarge the class of waves we
use by adding centered simple waves of a given initial width.  We can
again piece these together in increasing order of wavespeed, to get a
solution for some positive time $t_\#$, which is \emph{beyond} the
earliest time of collapse of any compressions.  This time $t_\#$ is
the first time of interaction of any reflected wave from a collapsed
compression with any other wave in the solution.

\begin{defn}
  \label{def:gRP}
  The \emph{generalized Riemann problem} (gRP) is the (implicitly
  defined) initial value problem determined by data consisting of two
  states $u_L$ and $u_R$ together with $N$ nonnegative initial
  widths $\{w_1,w_2,\dots,w_N\}$, such that $w_k=0$ for every linearly
  degenerate family $k$.  This problem is solved by piecing together,
  in increasing order of wavespeed, \emph{either} elementary centered
  waves given by \eqref{Ek}, \eqref{Ekz} in case $w_k=0$, \emph{or}
  centered simple waves of initial width $w_k>0$.  The profile of
  centered waves is determined by Lemma~\ref{lem:center}, and the
  waves are adjacent and non-overlapping at time $t=0$, so that the
  data satisfies $u^o(x)=u_L$ for $x<0$, $u^o(x)=u_R$ for
  $x>\sum w_k$.
\end{defn}

More precisely, to solve the gRP, we again set $u_0:=u_L$, and
referring to \eqref{Wk} and \eqref{Ek}, inductively set
\begin{equation}
  \label{gRsol}
  u_k :=
  \begin{cases}
    E_k(u_{k-1},\zeta_k) \in \mc E_k(u_{k-1}),&w_k=0,\\
    W_k(u_{k-1},\zeta_k) \in \mc W_k(u_{k-1}),&w_k>0,    
  \end{cases}  
\end{equation}
and again solve $u_N = u_R$.  Once we have found the intermediate
states $\{u_k\}$, we resolve any simple waves with $w_k>0$ into a
centered wave with initial profile given by Lemma~\ref{lem:center};
any shocks, contacts or origin centered rarefactions have zero width.
Finally, we piece each of these waves together to determine the
initial profile, and to give an exact weak or weak* solution of
\eqref{eqsys} for some finite time $t_\#$ beyond the first collapse
time but before multiple waves interact.  These adjacent waves do not
overlap in the time interval $[0,t_\#)$ by the Lax entropy condition
\eqref{Laxent}.  This gRP solution is shown in the left panel of
Figure~\ref{fig:gdRP}.

Because we are working with an abstract system, we do not prove
existence of solutions, but make the following abstract assumption on
the system.

\begin{assum}
  \label{ass:gRP}
  The generalized Riemann problem is globally uniquely solvable in
  $\mc U$.  That is, given states $u_L$, $u_R\in \mc U$ and initial
  widths $w_k$, the equation $u_N=u_R$ has a unique solution
  $\{\zeta_k\}$, in which each $u_k$ is given by \eqref{gRsol}.
\end{assum}

Note that the branch chosen in \eqref{gRsol} depends only on whether
or not $w_k>0$ , and once the $u_k$'s are known, the profile of the
solution is fully and uniquely determined by Lemma~\ref{lem:center}.

Our construction of solutions to the gRP follows Lax's solution of the
Riemann problem, in which the states $u_k$ are inductively defined.
We provide elementary conditions which are sufficient for
Assumption~\ref{ass:gRP} to hold.  These conditions are easily checked
in important physical systems such as gas dynamics, as shown in
Section \ref{sec:gd} below.

\begin{lemma}
  \label{lem:gRPsol}
  For a strictly hyperbolic system, assume Lax's extended entropy
  condition \eqref{sig} holds.  Given any $u_0\in\mc U$ and $u_k$,
  $k = 1,\dots,N$, satisfying
  \begin{equation}
  \label{ukzr}
    u_k\in \mc W_k(u_{k-1})\cup\mc E_k(u_{k-1}),
    \com{set}
    \z_k\,\ol r_k := u_k-u_{k-1},
  \end{equation}
  where say $|\ol r_k|=1$.  A \emph{sufficient} condition for
  existence of a unique solution to the gRP is that the set of
  \emph{non-zero} $\big\{\ol r_k\big\}$ be independent for all such
  states $\{u_k\}$.  If this independence holds, then the gRP has a
  unique solution provided
  \begin{equation}
  \label{range}
    u_r \in (\mc W_N\cup\mc E_N)\Big(\dots
    (\mc W_1\cup\mc E_1)\big(u_l\big)..\Big),
  \end{equation}
  in which each branch $\mc W_k$ or $\mc E_k$ is chosen consistently
  with the given width $w_k$.
\end{lemma}

For generic systems, we expect the conditions of this lemma to be
easier to check for the wave curves $\mc W_k$ than for the mixed
curves $\mc E_k$, because the wave curves satisfy the same autonomous
ODE on both branches.

\begin{proof}
  Given gRP data $u_L$, $u_R$ and widths $w_k$, we form the sequence
  $\big\{u_k\;;\;k=0,\dots,N\big\}$ using \eqref{gRsol}.  It then
  follows that
  \begin{equation}
  \label{urul}
    u_N = u_L + \sum_{k=1}^N\z_k\,\ol r_k,
  \end{equation}
  and by independence the coefficients $\z_k$ are uniquely determined.
  We now put each of these together as waves: that is, for a centered
  wave, corresponding to $w_k>0$, we use Lemma~\ref{lem:center} to
  define the profile and simple wave, while for a discontinuity or
  shock, we use use a jump having the correct speed $\sigma_k$.
  Genuine nonlinearity and the
  extended Lax entropy condition imply that these waves do not
  overlap, and so combining them according to Definition~\ref{def:gRP}
  yields a weak* solution and corresponding profiles for the gRP.
\end{proof}

It appears that the extended Lax entropy condition is a necessary
condition for strictly hyperbolic systems, as if it was violated, the
order of $k$ and $k'$ shocks would change, and a non-physical overlap
would occur.  We note that this can happen for nonstrictly hyperbolic
systems such as the nonlinear elastic string,~\cite{Ynonstr}, but we
will not pursue this here.

\subsection{Discretized Riemann Solution (dgRS)}

We cannot apply the gRP directly to generate an FT approximation,
because the simple waves have continuous profiles.  On the other hand,
the gRP solution is an exact solution, so is not piecewise constant,
and so we discretize to get a piecewise constant approximation
thereof.  We will refer to this as the discretized (generalized)
Riemann solution (dgRS).

The data for the dgRS is the same as that of the gRP, namely left and
right states $u_L$ and $u_R$, together with widths $\{w_k\}$
satisfying $w_k\ge0$ and $w_k=0$ for linearly degenerate $k$.  Here we
regard the widths $w_k$ as ``virtual widths'' of the waves, which will
in turn be represented by piecewise constant approximations.  That is,
in the discretization, we do not give waves actual width, but treat
them as jumps, while using these virtual widths as a bookkeeping
device to track whether a discretized wave is simple or not.

\begin{figure}[thb]
  \centering
\begin{tikzpicture}

\begin{groupplot}[group style={group size=2 by 1}]
\nextgroupplot[
hide x axis,
hide y axis,
tick align=outside,
tick pos=left,
xmin=-0.205, xmax=0.785,
ymin=-0.006, ymax=0.126,
]
\path [draw=bakShock, fill=bakShock, opacity=0.2]
(axis cs:0.1,0)
--(axis cs:0,0)
--(axis cs:-0.1,0.1)
--(axis cs:-0.1,0.1)
--(axis cs:-0.1,0.1)
--(axis cs:0.1,0)
--cycle;

\path [draw=extra, fill=extra, opacity=0.2]
(axis cs:0.3,0)
--(axis cs:0.1,0)
--(axis cs:0.22,0.12)
--(axis cs:0.54,0.12)
--(axis cs:0.54,0.12)
--(axis cs:0.3,0)
--cycle;

\path [draw=fwdShock, fill=fwdShock, opacity=0.2]
(axis cs:0.5,0)
--(axis cs:0.3,0)
--(axis cs:0.66,0.12)
--(axis cs:0.74,0.12)
--(axis cs:0.74,0.12)
--(axis cs:0.5,0)
--cycle;

\path [draw=fwdShock, fill=fwdShock, opacity=0.2]
(axis cs:-0.1,0.1)
--(axis cs:-0.1,0.1)
--(axis cs:0.0599999999999999,0.12)
--(axis cs:0.0999999999999999,0.12)
--(axis cs:0.0999999999999999,0.12)
--(axis cs:-0.1,0.1)
--cycle;

\addplot [thick, contact]
table {%
0.1 0
0.1 0.12
};
\addplot [thick, bakShock]
table {%
-0.1 0.1
-0.16 0.12
};
\addplot [thick, contact]
table {%
-0.1 0.1
-0.1 0.12
};
\addplot [thick, extra]
table {%
-0.1 0.1
-5.55111512312578e-17 0.12
};

\nextgroupplot[
hide x axis,
hide y axis,
tick align=outside,
tick pos=left,
xmin=-0.22575, xmax=0.45075,
ymin=-0.00575, ymax=0.12075,
]
\path [draw=bakShock, fill=bakShock, opacity=0.2]
(axis cs:-0.05,0)
--(axis cs:0.05,0)
--(axis cs:-0.15,0.1)
--(axis cs:-0.15,0.1)
--(axis cs:-0.15,0.1)
--(axis cs:-0.05,0)
--cycle;

\path [draw=extra, fill=extra, opacity=0.2]
(axis cs:0,0)
--(axis cs:0,0)
--(axis cs:0.2875,0.115)
--(axis cs:0.1725,0.115)
--(axis cs:0.1725,0.115)
--(axis cs:0,0)
--cycle;

\path [draw=fwdShock, fill=fwdShock, opacity=0.2]
(axis cs:-0.075,0)
--(axis cs:0.0750000000000001,0)
--(axis cs:0.42,0.115)
--(axis cs:0.385,0.115)
--(axis cs:0.385,0.115)
--(axis cs:-0.075,0)
--cycle;

\path [draw=fwdShock, fill=fwdShock, opacity=0.2]
(axis cs:-0.15,0.1)
--(axis cs:-0.15,0.1)
--(axis cs:-0.03,0.115)
--(axis cs:0,0.115)
--(axis cs:0,0.115)
--(axis cs:-0.15,0.1)
--cycle;

\addplot [thick, bakShock]
table {%
0 0
-0.15 0.1
};
\addplot [thick, contact]
table {%
0 0
0 0.115
};
\addplot [thick, extra]
table {%
0 0
0.17 0.1
};
\addplot [thick, extra]
table {%
0 0
0.2 0.1
};
\addplot [thick, extra]
table {%
0 0
0.23 0.1
};
\addplot [thick, extra]
table {%
0.17 0.1
0.1955 0.115
};
\addplot [thick, extra]
table {%
0.2 0.1
0.23 0.115
};
\addplot [thick, extra]
table {%
0.23 0.1
0.2645 0.115
};
\addplot [thick, fwdShock]
table {%
0 0
0.35 0.1
};
\addplot [thick, fwdShock]
table {%
0.35 0.1
0.4025 0.115
};
\addplot [thick, bakShock]
table {%
-0.15 0.1
-0.195 0.115
};
\addplot [thick, contact]
table {%
-0.15 0.1
-0.15 0.115
};
\addplot [thick, extra]
table {%
-0.15 0.1
-0.075 0.115
};
\addplot [line width=0.32pt, fwdShock]
table {%
-0.15 0.1
-0.0225 0.115
};
\addplot [line width=0.32pt, fwdShock]
table {%
-0.15 0.1
-0.00750000000000001 0.115
};
\end{groupplot}

\end{tikzpicture}
  \caption{Generalized (left) and discretized (right) Riemann
    solutions}
  \label{fig:gdRP}
\end{figure}

Given gRP data, we again define intermediate states $u_k$ as in
\eqref{gRsol}, and solve the equation $u_N=u_R$ exactly.  However,
instead of then giving each simple $k$-wave a profile with initial
width $w_k$, we approximate the wave by (one or more) jumps having no
actual width but with initial virtual width $w_k$.  We then place each
of these jumps at the origin and propagate them as jumps with
wavespeed $s_k$ which approximates the actual wavespeed
$\lambda_k\big(u(\zeta)\big)$ or $\sigma(u_{k-1},u_k)$, according as
whether the wave is simple or a shock.  This discretized solution thus
resembles an actual Riemann solution, with all waves emerging from the
origin, but allows for compressions and discretizes centered
rarefactions.  By introducing and using virtual widths, we have
essentially introduced ``centered compressions'' into the approximate
solution.  As in solutions of the gRP, the dgRS is defined only for
short times, although again it can be continued beyond the earliest
collapse time of compressions, to say $t_\#>t_c$.  However, for our
purposes it will be sufficient to continue only up to $t_c$ and to
treat the collapse as an explicit interaction.  The dgRS is
illustrated in the right panel of Figure~\ref{fig:gdRP}: the shaded
areas represent the shifted individual simple waves, while the dark
lines are are the actual discontinuities in the dgRS.

The dgRS is not an actual solution of the system, so does not
have vanishing residual.  However, because the solution is piecewise
constant, we can calculate the residual as in Lemma~\ref{lem:res}, to
give
\[
  \mc R = \sum_{j=1}^m\big([f]_j-s_j\,[q]_j\big)\,\delta_{s_j\,t},
\]
valid for $0\le t\le t_c$, where
\[
  [f]_j := f(u_j)-f(u_{j-1}) \com{and}
  [q]_j := q(u_j)-q(u_{j-1}).
\]
In fact, since a weak* solution corresponds to vanishing residual, but
our construction develops approximations to the solution, we can
modify the residual slightly by applying a \emph{preconditioner},
provided that this is uniformly bounded and invertible, as follows.
For each jump $[u]_j$, we pick some $N\times N$ preconditioning matrix
$A_j$, and define the \emph{preconditioned residual}
\begin{equation}
  \label{pRes}
  \mc R^P := \sum_{j=1}^mA_j\,
  \big([f]_j-s_j\,[q]_j\big)\,\delta_{s_j\,t}.
\end{equation}
Since $A_j$ is uniformly bounded and invertible, it is clear that
there is some constant $K$ such that
\[
  \frac1K\,\|\mc R\|_M \le \|\mc R^P\|_M \le K\,\|\mc R\|_M,
\]
so convergence of the approximation to a solution is equivalent
for either the unconditioned or preconditioned residuals.

For moderate simple waves, the dgRS may have a large residual, because
a moderate simple wave cannot be well approximated by a single jump.
We thus choose to further discretize \emph{rarefactions} in the dgRS,
using the group property \eqref{lgp} of simple waves.  There is no way
to further discretize compressions, as all waves in the dgRS emanate
from the origin.  In particular, in asking that the residual be small,
so that our approximations are close to solutions, \emph{requires}
that all compressions in the approximation be sufficiently weak; if
compressions become too strong, they are immediately replaced by
shocks.

We describe how to further discretize a moderate $k$-rarefaction, for
which
\[
  u_k = W_k(u_{k-1},\zeta_k),
\]
and $\zeta_k>0$ is relatively large.  If we want a maximum rarefaction
strength of $\eta$, say, we choose
\[
  M_k=M(\zeta_k)\ge\lceil \zeta_k/\eta\rceil, \qquad M_k\in\B N,
\]
and choose $M_k$ strengths $\zeta_k^{(m)}$ satisfying
\begin{equation}
  \label{zki}
  \sum_{i=1}^{M_k}\zeta_k^{(i)} = \zeta_k, \qquad
  0<\zeta_k^{(i)}\le \eta,\ i=1,\dots,M_k.
\end{equation}
Then we set
\[
  u_k^{(0)} := u_{k-1}, \qquad
  u_k^{(i)} := W_k\big(u_k^{(i-1)},\zeta_k^{(i)}\big),
\]
for $i=1,\dots,M_k$, so that $\big\{\la_k(u_k^{(i)})\big\}$ is
monotonically increasing, and $u_k^{(M_k)}=u_k$, by the group property
\eqref{lgp}.  We now choose an approximate speed
$s_k^{(i)}\in \big(\la_k(u_k^{(i-1)}),\la_k(u_k^{(i)})\big)$, and
approximate the part of the centered rarefaction between $u_k^{(i-1)}$
and $u_k^{(i)}$ by a jump with speed $s_k^{(i)}$.  We summarize the
foregoing in the following lemma.

\begin{lemma}
  \label{lem:dgRS}
  Given any gRP data $\big\{u_L,u_R,w_1,\dots,w_N\big\}$, there is a
  dgRS approximation $u$ of the gRP solution, defined for
  \[
    0\le t\le t_\# :=
    \min_{k,\;w_k>0}\Big\{\frac{w_k}
    {\big(\la_k(u_{k-1})-\la_k(u_k)\big)_+}\Big\}>0,
  \]
  which has the form
  \[
    u(x,t) =
    \begin{cases}
      u_1^{(0)}, &x<s_1^{(1)}\,t,\\[3pt]
      u_k^{(i)}, &s_k^{(i)}\,t < x < s_k^{(i+1)},\\[3pt]
      u_k^{(0)}, &s_{k-1}^{(M_{k-1})} < x < s_k^{(1)}\,t,\\[3pt]
      u_N^{(M_N)}, &s_N^{(M_N)}\,t < x,
    \end{cases}
  \]
  in which
  \[
    u_1^{(0)} = u_0 = u_L, \quad
    u_{k-1}^{(M_{k-1})} = u_k^{(0)} = u_{k-1}, \quad
    u_N^{(M_N)} = u_N = u_R,
  \]
  and in which $M_k=1$ unless $\zeta_k>\eta$.  The wavespeeds
  $s_k^{(i)}$ are monotone increasing, and this solution
  is unique up to choices of the strengths $\zeta_k^{(i)}$ in
  \eqref{zki} and speeds $s_k^{(i)}$.

  The preconditioned residual of the dgRS approximation is
  \begin{equation}
    \label{dgRSres}
    \mc R^P(u) = \sum_{k=1}^N\sum_{i=1}^{M_k}A_k^{(i)}\,
    \big([f]_k^{(i)}-s_k^{(i)}\,[q]_k^{(i)}\big)\,\delta_{s_k^{(i)}\,t},
  \end{equation}
  where we have set
  \[
    [f]_k^{(i)} := f(u_k^{(i)})-f(u_k^{(i-1)}) \com{and}
    [q]_k^{(i)} := q(u_k^{(i)})-q(u_k^{(i-1)}),
  \]
  and where $A_k^{(i)}$ is the preconditioner of the wave $\g_k^{(i)}$.
\end{lemma}

We have not yet chosen the preconditioners, which could in principle
be any uniformly bounded matrix with bounded inverse.  It turns out
that there is a natural preconditioner in the most important case that
the system \eqref{eqsys} is symmetric, which corresponds to most
physical systems for which energy is conserved.  Henceforth we assume
that the preconditioner $A$ is implicitly given and drop the
superscript $P$.

\begin{proof}
  Existence and uniqueness of intermediate states $u_k$ and positivity
  of $t_\#$ is a result of Assumption~\ref{ass:gRP}.  If the $k$-wave
  is a shock, contact, compression or weak rarefaction, we have
  $M_k=1$, and if it is a moderate rarefaction, $\zeta_k>\eta$, then we
  discretize the corresponding wave as above, with smaller wave
  strengths satisfying \eqref{zki}.  Finally, because the
  approximation is piecewise constant, the form of the residual
  follows by Lemma~\ref{lem:res}.
\end{proof}

\section{Modified Front Tracking (mFT)}
\label{sec:mft}

Having developed the dgRS as a piecewise constant approximation of the
gRP solution, we now combine these to build piecewise constant
approximations of more general solutions.  As in other FT
applications, we glue different dgRSs together with common adjacent
states.  Effectively, we pose a series of non-overlapping gRPs,
resolve these to get a solution until two nonlinear waves interact,
and then try to repose the gRPs to resolve the interactions.  Because
the system is highly nonlinear and wave interactions cannot be easily
resolved, we instead use dgRSs, which keep us in the class of
piecewise constant solutions.  It is then easy to decide when waves
interact, and because we remain in the piecewise constant class, we
can resolve these interactions using more dgRSs.

There are several questions that need to be resolved to carry this
out: first, we must choose wavespeeds for the approximate waves;
second, we must decide exactly which dgRS to pose in order to resolve
interactions; and third, we must endeavour to keep the number of waves
finite and avoid accumulation of interaction times, so that the scheme
can be continued for long times.  In particular, the last requirement,
that of keeping the number of waves finite and avoiding accumulation
of interaction times, requires some technical modifications of the
scheme.  Note that we must also consistently include dependence on the
discretization parameters.

\subsection{Bookkeeping}

At any given time, the modified FT scheme has a finite number of
waves, each of which is a single jump discontinuity propagating at a
given speed and with several other data of interest.  Following the
method of reorderings, developed in~\cite{Yth}, we define a
\emph{wave sequence} which names each wave and each associated piece
of data.  We find this to be an efficient bookkeeping tool for
tracking the position, speed, strength and other data associated to
each wave.

For a given time $t$, we thus define the wave sequence $\Gamma$ to be
a finite sequence of waves,
\[
  \Gamma^t = \big\{\gamma^t_j\;;\;j= 1,\dots,n^t\big\},
\]
in which $n^t$ is the number of waves in the sequence.  The time $t$
is usually implicitly understood, and when this is the case we will
drop the superscript $t$ and write
$\Gamma= \big\{\gamma_j\;:\;j= 1,\dots,n\big\}$.  Each wave in the
sequence carries its own data, and we can write it as an ordered
tuple, such as
\begin{equation}
  \label{wave}
  \gamma_j = \big(\zeta_j,k_j,u_{j-1},u_j, s_j,
  w_j(t),x^c_j,t^c_j,t^!_j, \dot w_j, A_j\big);
\end{equation}
these components are, respectively: wave strength $\zeta_j$, family
$k_j\in 1,\dots,N$, left state $u_{j-1}$, right state $u_j$, wavespeed
$s_j$, virtual width $w_j(t)\ge 0$, virtual center or origin
$(x^c_j,t^c_j)$, interaction time $t^!_j$, rate of expansion
$\dot w_j$, and preconditioner $A_j$.  Several of these data are
related, for example
\[
  \dot w_j = \la_{k_j}(u_j) - \la_{k_j}(u_{j-1}), \com{or}
  u_j = W_{k_j}(u_{j-1},\zeta_j),
\]
for a simple wave or contact.  We can also store other information as
convenient or necessary, such as time of origin of the wave.  We can
similarly calculate other quantities, such as the \emph{position}
$x^p_j(t)$ of the wave $\g_j$ at time $t$, namely
\[
  x^p_j(t) = x^c_j + s_j\,(t-t^c_j).
\]
We note that the origin and center of a simple wave are generally
distinct, because the wave is generated by a nonlinear
interaction.  For a rarefaction, the center occurs earlier than the
origin, while for a compression, the center is in the future.
However, both points lie on the wave trajectory $x^p(t)$, which
strictly speaking applies only from the origin to the next interaction
time of the simple wave.  By a slight abuse of notation, we record
only the center, and implicitly take $t$ in the appropriate interval.
For consistency, if the wave is a shock or contact, we let $(x^c,t^c)$
denote the origin of the wave.

In this wave sequence notation, the (preconditioned) \emph{residual}
of the wave sequence is defined as above, and we get an expression
similar to \eqref{pwcres} and \eqref{pRes}, namely
\begin{equation}
  \label{RGam}
  \mc R(\Gamma) = \sum_{j=1}^nA_j\,\big([f]_j-s_j\,[q]_j)\,
  \delta_{x^p_j(t)},  
  \end{equation}
with $[f]_j=f(u_j)-f(u_{j-1})$ and $[q]_j = q(u_j)-q(u_{j-1})$.
Another functional of great interest is the \emph{variation} of the
wave sequence, defined by
\begin{equation}
  \label{Vdef}
  \mc V(\Gamma) := \sum_{j=1}^n|\z_j|,
\end{equation}
and which is equivalent to the $BV$ norm, as well as other norms in
which the Heaviside function is bounded, such as $H^b$ for $b<1/2$.

Note that at an interaction, the number of waves and the corresponding
indices of noninteracting waves will generally change, so care must be
taken in associating waves between different interaction times.  In the
implementation of the scheme, the wave sequence and corresponding
\emph{state sequence}
\[
  u := \big\{u_j\;:\;j = 0,\dots n\big\},
\]
are stored as linked lists, which combines well with changing numbers
of waves.  By allowing adjacent waves to occupy the same physical
point $(x,t)$ when an interaction occurs, this notation extends to
interaction times as well, so for a discretized system, we have a
well-defined wave sequence for all times $t\in[0,T)$.

Recall that one of our guiding principles is to treat states
\emph{exactly}, so that the total variation of approximations is
always accurately given, while other features of waves may be
approximate.  Thus in our description \eqref{wave}, some data,
particularly those data that are defined by states, must be exact,
while others may be approximated.  We have no choice in specifying any
of $\zeta_j$, $k_j$, $u_{j-1}$, $u_j$ or $\dot w_j$, while we may vary
speed $s_j$ and width $w_j$, and choose preconditioner $A_j$, as
needed for consistency and to bound the number of waves and/or
residual; on the other hand $(x^c_j,t^c_j)$ and $t^!_j$ are calculated
quantities, so are not directly adjustable.  We conclude with the
important observation that \emph{we are able to modify the speed
  $s_j$, virtual width $w_j$, and preconditioner $A_j$ of each wave as
  needed, but we cannot adjust the other parameters in \eqref{wave}.}

We will use the following notation for the (various) discretization
parameters: first, we let $\e$ denote the basic discretization
parameter, so we are interested in the limit $\e\to 0+$.  We then
introduce various exponents $e_\square$ to denote the different
thresholds which we use.  For example, we will refer to the maximum
(absolute) compression or rarefaction strengths of an individual wave
as $\ee c$ and $\ee r$, respectively.  We will then establish
convergence by taking $\e\to 0$, while keeping the various exponents
$e_\square$ finite.

\subsection{Choice of Wavespeed}

A shock or contact is a single jump with wavespeed given by
\eqref{sc}, and so if this exact speed is chosen, the corresponding
approximate wave contributes nothing to the residual.  It is thus
natural to use the exact wavespeed $s = \sigma(u_-,u_+)$ for these
waves.

For simple waves, we observe that each individual wave contributes a
single term to the residual in \eqref{RGam}, namely
\[
  R\,\delta_{x^(t)} := A\,\big([f]-s\,[q]\big)\,\delta_{x^(t)},
\]
in which the left and right states are $u_\mp$.  Since Dirac measures
supported at different points are independent, by the triangle
inequality each wave contributes its Euclidean norm
$\big|A\,[f]-s\,A\,[q]\big|$ to the residual (since
$\|\delta\|=1$).  It thus makes sense to try to minimize this norm in
choosing the wavespeed $s$.

We define the \emph{``vector residual''}
$R:=A\,\big([f]-s\,[q]\big)\in\B R^N$ of each wave, and use least
squares to find $s$ that minimizes this residual: that is, we minimize
\begin{equation}
  \label{lsq}
  \big(A\,([f]-s\,[q])\big)^2 =
  \wh{[f]}\cdot\wh{[f]} -
  2\,s\,\wh{[f]}\cdot\wh{[q]} +
  s^2\,\wh{[q]}\cdot\wh{[q]},
\end{equation}
where we have set 
\[
  \wh{[q]} := A[q]  \qcom{and}
  \wh{[f]} := A[f].
\]
Note that this least square minimization gives $s=\sigma(u_-,u_+)$ and
$R=0$ if the wave is a shock or contact, so this is a consistent
choice for all waves.

\begin{lemma}
  \label{lem:spdRes}
  For $u_+ = W_k(u_-,\pm\z)$, $\z>0$, the least squared wave speed and
  corresponding vector residual are
  \begin{equation}
    \label{sR}
    s_* = \frac{\wh{[q]}\cdot\wh{[f]}}{\wh{[q]}\cdot\wh{[q]}},
    \com{and}
    R_* = \wh{[f]} -
    \frac{\wh{[q]}\cdot\wh{[f]}}{\wh{[q]}\cdot\wh{[q]}}\,\wh{[q]}
    =: \Pi_{\wh{[q]}^\perp}\wh{[f]},
  \end{equation}
  respectively, where
  \[
    \wh{[q]} := A[q], \qquad
    \wh{[f]} := A[f], \qcom{and}
    \Pi_{\wh{[q]}^\perp} := I -
    \frac{\wh{[q]}\,\wh{[q]}^\T}{\wh{[q]}^\T\wh{[q]}}
  \]
  is projection orthogonal to $\wh{[q]}$.  If the same jump is given
  speed $s$, then the corresponding residual length is
  \begin{equation}
    \label{Rs}
    \begin{aligned}
      \big|R(s)\big|
      &= \sqrt{\big|R_*\big|^2 + (s-s_*)^2\big|\wh{[q]}\big|^2}\\
      &= \big|\wh{[f]}\big|\,\sqrt{\sin^2\theta +
        \big(1-\frac s{s_*}\big)^2\,\cos^2\theta},
    \end{aligned}
  \end{equation}
  where $\theta$ is the angle between the vectors $\wh{[q]}$ and
  $\wh{[f]} \in \B R^N$.

  If the wave is weak and either simple or elementary,
  \[
    u_+ = W_k(u_-,\pm\z) \com{or}u_+ = E_k(u_-,\pm\z),
    \com{and} 0\le \z\ll 1,
  \]
  then the states satisfy
  \begin{equation}
    \label{fexp}
    \wh{[f]} = \la_k(\ol u)\,\wh{[q]} + O(\z^3), \qquad
    [f] = \la_k(\ol u)\,[q] + O(\z^3),
  \end{equation}
  and the least squared wave speed satisfies
  \begin{equation}
    \label{smRes}
    s_* = \la_k(\ol u) + O(\z^2),
  \end{equation}
  where $\ol u:=W_k(u_-,\pm\frac\z2)=W_k(u_+,\mp\frac\z2)$.  The
  residuals satisfy
  \begin{equation}
    \label{Res*s}
    R(s) = \big(\la_k(\ol u)-s\big)\,\wh{[q]} + O(\z^3)
    = (s_*-s)\,\wh{[q]} + O(\z^3),
  \end{equation}
  and in particular, $R_* = O(\z^3)$.
\end{lemma}

\begin{proof}
  The expressions \eqref{sR} for the least square speed and
  corresponding residual are immediate from \eqref{lsq}.  The residual
  is the vector measure $R_*\,\delta$, and, using $\|\delta\|=1$, the
  squared size of this is
  \[
    |R_*|^2 = R_*^\T R_* = \wh{[f]}\cdot\wh{[f]} - s_*\,\wh{[q]}\cdot\wh{[f]}
    = \wh{[f]}\cdot\wh{[f]} - s_*^2\,\wh{[q]}\cdot\wh{[q]},
  \]
  where we have used \eqref{sR}.

  For a general speed $s$, we have
  \[
    \begin{aligned}
      \big|R(s)\big|^2 &= \wh{[f]}\cdot\wh{[f]} - 2\,s\,\wh{[f]}\cdot\wh{[q]} +
                         s^2\,\wh{[q]}\cdot\wh{[q]}\\
      &= \big|R_*(s)\big|^2 + \big(s_*^2 - 2\,s\,s_* + s^2\big)\,
                         \wh{[q]}\cdot\wh{[q]},
    \end{aligned}
  \]
  where we have again used \eqref{sR}, and the first part of
  \eqref{Rs} follows.  Next, if $\theta$ denotes the angle between
  the vectors $\wh{[f]}$ and $\wh{[q]}$, then
  \[
    \wh{[f]}\cdot\wh{[q]} = \big|\wh{[f]}\big|\,\big|\wh{[q]}\big|\,\cos\theta
    \com{and}
    s_*\,\big|\wh{[q]}\big| = \big|\wh{[f]}\big|\,\cos\theta,
  \]
  which yields
  \[
    |R_*|^2 = \big|\wh{[f]}\big|^2 - \big|\wh{[f]}\big|^2\,\cos^2\theta,
    \com{so}
    |R_*| = \big|\wh{[f]}\big|\,|\sin\theta|,
  \]
  so in turn,
  \[
    \big|R(s)\big|^2 = \big|\wh{[f]}\big|^2\,
    \Big(\sin^2\theta + \big(1-\frac s{s_*}\big)^2\,\cos^2\theta\Big),
  \]
  which yields the second part of \eqref{Rs}.

  Recall that the simple wave curve $\mc W_k(u_-)$, given by
  \eqref{Wk}, is the integral of the generalized eigenvector $r_k(u)$,
  \eqref{ev}, \eqref{ODE}.  Using the group property \eqref{lgp}, we
  define
  \begin{equation}
    \label{midu}
    \begin{gathered}
      \ol u:=W_k\Big(u_-,\pm\frac\z2\Big) =
      W_k\Big(u_+,\mp\frac\z2\Big), \com{so}\\
      u_- = W_k\Big(\ol u,\mp\frac\z2\Big), \qquad
      u_+ = W_k\Big(\ol u,\pm\frac\z2\Big).
    \end{gathered}
  \end{equation}  
  For this single wave, we regard $\ol u$ as fixed and parameterize
  the curve $\mc W_k(\ol u)$ by $z$, so that $u=u(z)$, and write
  \begin{equation}
    \label{[qf]}
    \begin{aligned}{}
      [q] &:= q\big(u(\pm\zeta/2)\big) - q\big(u(\mp\zeta/2)\big)
            = \int_{\mp\z/2}^{\pm\z/2} Dq\big(u(\fs z)\big)\,
            r_k\big(u(\fs z)\big)\;d\fs z,\com{and}\\
      [f] &:= f\big(u(\pm\z/2)\big) - f\big(u(\mp\z/2)\big)
            = \int_{\mp\z/2}^{\pm\z/2} \la_k\big(u(\fs z)\big)\,
            Dq\big(u(\fs z)\big)\,r_k\big(u(\fs z)\;d\fs z.
    \end{aligned}
  \end{equation}

  For weak simple waves, $0<\z\ll 1$, we expand $\la_k$ along
  the wave curve to get
  \[
    \la_k\big(u(z)\big) = \la_k(\ol u) + \dot\la_k(\ol u)\,z
     + O(z^2), \qquad |z|<\z.
  \]
  Now use \eqref{[qf]} to write
  \[
    \begin{aligned}
      {}[f]
      &= \int_{\mp\z/2}^{\pm\z/2}\Big(\la_k(\ol u) + \dot\la_k(\ol u)\,\fs z
        + O(\fs z^2)\Big)\,
        Dq\big(u(\fs z)\big)\,r_k\big(u(\fs z)\big)\;d\fs z\\[3pt]
      &= \la_k(\ol u)\,[q] + \dot\la_k(\ol u)\,
        \int_{\mp\z/2}^{\pm\z/2}\fs z\,
        Dq\big(u(\fs z)\big)\,r_k\big(u(\fs z)\big)\;d\fs z + O(\z^3).
    \end{aligned}
  \]
  Similarly expanding
  \[
    Dq\big(u(z)\big)\,r_k\big(u(z)\big)
    = Dq(\ol u)\,r_k(\ol u) + O(z),
  \]
  and substituting in, we get
 \[
   [f] = \la_k(\ol u)\,[q] + \dot\la_k(\ol u)\,
   \int_{\mp\z/2}^{\pm\z/2}\fs z\,Dq(\ol u)\,r_k(\ol u)\;d\fs z
   + O(\z^3),
 \]
 and the integral vanishes by symmetry.  We have effectively used the
 midpoint rule to carry out the integration.  We conclude that
 \[
   [f] = \la_k(\ol u)\,[q] + O(\z^3), \com{so also}
   \wh{[f]} = \la_k(\ol u)\,\wh{[q]} + O(\z^3),
 \]
 which is \eqref{fexp}, and so by \eqref{sR},
 \[
   s_* = \frac{\la_k(\ol u)\,\wh{[q]}\cdot\wh{[q]}}{\wh{[q]}\cdot\wh{[q]}}
   + \frac{O(\z^3)}{\big|\wh{[q]}\big|} = \la_k(\ol u) + O(\z^2),
 \]
 since $\wh{[q]} = O(\z)$, which proves \eqref{smRes}.  Similarly,
 \[
   R(s) = \wh{[f]} - s\,\wh{[q]} = \big(\la_k(\ol u)-s\big)\,\wh{[q]} + O(\z^3),
 \]
 which is \eqref{Res*s}.
\end{proof}

Although the least square choice is best in terms of minimizing the
residual, below we will sometimes adjust the wavespeed of (weak) waves
to keep the number of waves under control.  In particular, if
$\big|s-\la_k(\ol u)\big| = O(\z^\beta)$ for $\beta<2$, then
$R(s) = O(\z^{1+\beta})$.

Although the order of the residual in \eqref{Res*s} is primarily
determined by the choice of $s$, the choice of preconditioner $A$ is
clearly reflected in the formulas \eqref{sR}.  In particular, for
specific systems, certain calculations may be simplified by a
convenient choice of $A$.  One such example is the $p$-system,
described in Section~\ref{sec:psys} below.

\subsection{Interactions}

Different waves generally propagate with different speeds, and so
adjacent waves may meet, at which point a nonlinear interaction
occurs.  Because waves are represented by jump discontinuities, they
meet at a single point, and we will continue the solution beyond that
point by placing a dgRS at that point.  The data for the dgRS consists
of left and right states, which are exact and given, together with
outgoing virtual wave widths, which implicitly determine the type of
wave and virtual center of a simple wave.  If the wave is not simple,
we take the interaction point as the wave origin.  Because we are
regarding the virtual width as a bookkeeping device, which
nevertheless affects future wave patterns, we are able to choose the
virtual widths of outgoing waves.  We will do so in such a way as to
maximize consistency and avoid inadmissible overlaps of simple waves.

To be precise, consider the waves $\gamma_{j-1}$ and $\gamma_j$.
Their trajectories are given by
\begin{equation}
  \label{xp}
  x^p_{j-1}(t) = x^c_{j-1}+s_{j-1}\,(t-t^c_{j-1}) \com{and}
  x^p_j(t) = x^c_j+s_j\,(t-t^c_j),
\end{equation}
respectively, and we have $x^p_{j-1}(t)<x^p_j(t)$ for times $t<t^!$,
the time of interaction, and $s_{j-1}>s_j$.  We calculate $t^!$ by
setting $x^p_{j-1}(t^!)=x^p_j(t^!)=:x^!$, which gives
\begin{equation}
  \label{xt!}
  \begin{aligned}
    t^! &= \frac{s_{j-1}}{s_{j-1}-s_j}\,\Big(t_{j-1}^c
          -\frac{x^c_{j-1}}{s_{j-1}}\Big)
          + \frac{-s_j}{s_{j-1}-s_j}\,\Big(t_j^c
          -\frac{x^c_j}{s_j}\Big),\\
    x^! &=\frac{s_{j-1}}{s_{j-1}-s_j}\,\big(x_j^c-s_j\,t_j^c\big)
          +\frac{-s_j}{s_{j-1}-s_j}\,\big(x_{j-1}^c-s_{j-1}\,t_{j-1}^c\big).
  \end{aligned}
\end{equation}
Thus the new dgRS is placed at the interaction point $(x^!,t^!)$, and
the left and right states are those of the incident waves, namely
$u_{j-2}$ and $u_j$, respectively.

To continue the solution, it remains to specify the virtual widths
$w'_k\ge 0$ of the outgoing waves.  As always, we set $w'_k=0$ for
linearly degenerate families $k$.  The exact value of the assigned
widths $w'_k$ is less important than whether or not these vanish,
because this changes the type of compressive waves: if $w'_k>0$, the
$k$-wave is a compression, while if $w'_k=0$ it is a shock.  Because
we want to avoid overlap of waves, we use the convention that \emph{at
  an interaction, the virtual width $w'_k$ of an emergent $k$-wave is
  no larger than that of incident $k$-waves.}

We are considering the interaction of the $k_{j-1}$-wave
$\gamma_{j-1}$ and the $k_j$-wave $\gamma_j$, and these have
non-negative virtual widths $w_{j-1}(t^!)$ and $w_j(t^!)$,
respectively.  Because $s_{j-1}>s_j$, there are two cases: either
$k_{j-1}=k_j$ or $k_{j-1}>k_j$.  Both of these are illustrated in
Figure~\ref{fig:int}: on the left panel a rarefaction catches up to a
shock, and on the right two simple waves of different families cross;
in this case, the widths of the incident waves do not increase.

\begin{figure}[thb]
  \centering
\begin{tikzpicture}

\begin{groupplot}[group style={group size=2 by 1}]
\nextgroupplot[
hide x axis,
hide y axis,
tick align=outside,
tick pos=left,
xmin=-0.1838, xmax=0.2738,
ymin=-0.0555, ymax=0.0655,
]
\path [draw=extra, fill=extra, opacity=0.2]
(axis cs:-0.025,0)
--(axis cs:0.025,0)
--(axis cs:-0.09,-0.05)
--(axis cs:-0.1,-0.05)
--(axis cs:-0.1,-0.05)
--(axis cs:-0.025,0)
--cycle;

\path [draw=bakShock, fill=bakShock, opacity=0.2]
(axis cs:-0.025,0)
--(axis cs:0.025,0)
--(axis cs:-0.137,0.06)
--(axis cs:-0.163,0.06)
--(axis cs:-0.163,0.06)
--(axis cs:-0.025,0)
--cycle;

\path [draw=fwdShock, fill=fwdShock, opacity=0.2]
(axis cs:-0.025,0)
--(axis cs:0.025,0)
--(axis cs:0.253,0.06)
--(axis cs:0.167,0.06)
--(axis cs:0.167,0.06)
--(axis cs:-0.025,0)
--cycle;

\addplot [thick, extra]
table {%
0 0
-0.0325 -0.05
};
\addplot [thick, extra]
table {%
0 0
-0.095 -0.05
};
\addplot [thick, bakShock]
table {%
0 0
-0.15 0.06
};
\addplot [thick, contact]
table {%
0 0
0 0.06
};
\addplot [thick, extra]
table {%
0 0
0.051 0.06
};
\addplot [thick, fwdShock]
table {%
0 0
0.21 0.06
};
\addplot [line width=0.32pt, black, dashed]
table {%
-0.15 0
0.2 0
};

\nextgroupplot[
hide x axis,
hide y axis,
tick align=outside,
tick pos=left,
xmin=-0.1778, xmax=0.2358,
ymin=-0.0555, ymax=0.0655,
]
\path [draw=bakShock, fill=bakShock, opacity=0.2]
(axis cs:-0.015,0)
--(axis cs:0.015,0)
--(axis cs:0.165,-0.05)
--(axis cs:0.085,-0.05)
--(axis cs:0.085,-0.05)
--(axis cs:-0.015,0)
--cycle;

\path [draw=extra, fill=extra, opacity=0.2]
(axis cs:-0.025,0)
--(axis cs:0.025,0)
--(axis cs:-0.09,-0.05)
--(axis cs:-0.1,-0.05)
--(axis cs:-0.1,-0.05)
--(axis cs:-0.025,0)
--cycle;

\path [draw=bakShock, fill=bakShock, opacity=0.2]
(axis cs:-0.015,0)
--(axis cs:0.015,0)
--(axis cs:-0.141,0.06)
--(axis cs:-0.159,0.06)
--(axis cs:-0.159,0.06)
--(axis cs:-0.015,0)
--cycle;

\path [draw=extra, fill=extra, opacity=0.2]
(axis cs:-0.025,0)
--(axis cs:0.025,0)
--(axis cs:0.091,0.06)
--(axis cs:0.011,0.06)
--(axis cs:0.011,0.06)
--(axis cs:-0.025,0)
--cycle;

\path [draw=fwdShock, fill=fwdShock, opacity=0.2]
(axis cs:-0.025,0)
--(axis cs:0.025,0)
--(axis cs:0.217,0.06)
--(axis cs:0.203,0.06)
--(axis cs:0.203,0.06)
--(axis cs:-0.025,0)
--cycle;

\addplot [thick, bakShock]
table {%
0 0
0.125 -0.05
};
\addplot [thick, extra]
table {%
0 0
-0.095 -0.05
};
\addplot [thick, bakShock]
table {%
0 0
-0.15 0.06
};
\addplot [thick, contact]
table {%
0 0
0 0.06
};
\addplot [thick, extra]
table {%
0 0
0.051 0.06
};
\addplot [thick, fwdShock]
table {%
0 0
0.21 0.06
};
\addplot [line width=0.32pt, black, dashed]
table {%
-0.15 0
0.2 0
};
\end{groupplot}

\end{tikzpicture}
  \caption{Interactions of waves in the same family (left) and
    different families (right)}
  \label{fig:int}
\end{figure}

In the case of two waves in the same family, $k_{j-1}=k_j$, at least
one of the waves must be a shock, so one of $w_{j-1}(t^!)$ or
$w_j(t^!)$ vanishes.  In this case, the emerging $k_j$ wave must also
be a shock, so we set $w'_{k_j}=0$.  On the other hand, if the other
incident wave has positive width, so is simple, then the interaction
necessarily takes some positive time, and the reflected waves in other
GNL families will be simple, so we choose
\[
  0 \le w'_k \le \max\big\{w_{j-1}(t^!),\,w_j(t^!)\big\},
\]
for GNL families $k\ne k_{j-1}=k_j$.  Note that if both incident waves
are shocks, then all emerging waves will have zero width.  Here we use
an interval condition to choose the emerging virtual widths, because
we may need to adjust these to avoid inadmissible overlaps of simple
waves, as discussed below.

In case the two interacting waves are from different families, so
$k_{j-1}>k_j$, the virtual widths of the outgoing waves are chosen so that
they do not increase,
\[
  0\le w'_{k_{j-1}}\le w_{j-1}(t^!), \com{and}
  0\le w'_{k_j} \le w_j(t^!),
\]
while other GNL waves $k\ne k_j,\ k_{j-1}$ are again assigned widths
\begin{equation}
  \label{wid}
  0 \le w'_k \le \max\big\{w_{j-1}(t^!),\,w_j(t^!)\big\},  
\end{equation}
because we expect the interaction zone to extend over the larger width.

Note that pairwise interactions do not exhaust all possible
interactions.  On the one hand, three or more waves can in principle
interact at the same point $(x^!,t^!)$, although this is a measure
zero occurrence, and virtual widths can be consistently assigned in a
similar way.  More importantly, because the scheme has compressions,
we need to introduce \emph{self-interactions}, in which a single
simple compression collapses.  In the absence of other waves, a
compressive wave $\gamma_j$ with virtual width $w_j(t_0)>0$ at time
$t_0$ will collapse at time $t_j^!=t_j^c$ when this width becomes
zero, which is
\[
  w_j(t_0) + \dot w_j\,(t_j^!-t_0) = 0, \com{that is}
  t^!_j = t_0 - \frac{w_j(t_0)}{\dot w_j} > t_0.
\]
When this collapse of a compression occurs, we again place a dgRS with
left state $u_{j-1}$ and right state $u_j$, but since the wave is
virtually centered and the collapse takes place at a single point, we
set all emergent virtual widths to zero, $w_k'=0$, $k=1,\dots,N$.
Thus the new dgRS is a standard Riemann solution, consistent with
Lemma~\ref{lem:center},

Note that the wave $\gamma_j$ could have up to three potential
interaction times, corresponding to interaction with either of its
adjacent waves $\gamma_{j-1}$ and $\gamma_{j+1}$, as well as a
possible self-interaction.  In all cases, when setting $t_j^!$ upon
generation of the wave, we set $t_j^!$ to be the minimum such
interaction time.

\subsection{Accelerated Collapse of Compressions}

We have observed above that if the strength of a single simple wave is
too large, then so to will the residual be too large, and the
approximation insufficiently accurate.  We can impose this upon
initialization, but moderate simple waves could appear as a result of
some nonlinear interactions.  If the large simple wave is a
rarefaction, we will split it into smaller pieces using the group
property \eqref{lgp} as described below, but if the wave is a
compression, this is not an option, as all waves emanate from the
point of interaction.

In this case, suppose that some interaction produces a moderate
compression, say in the GNL $k$-th family.  That is, having solved the
dgRS with given virtual widths including $w'_k>0$, the $k$-th
resultant wave $\gamma'_k$ is a compression whose strength is beyond
some threshold, say $\zeta_k'<-\ee c$.  This means that this $k$-th
compression is too strong, and hence the residual too large.  In this
case, we will \emph{force} the $k$-wave to instead be a shock, by
adjusting the virtual width of the emergent $k$-wave to zero: that is,
we reset $w'_k=0$ and re-evaluate the dgRS with this adjusted width
data.

Use of compressions eventually leads to an increase in the number of
waves, because when the compression collapses, it generates waves in
other families.  To avoid excessive increase in the number of waves,
we therefore \emph{also} collapse very weak compressions that arise as
a result of interactions into shocks, immediately.  Thus if a
compression satisfies \emph{either}
\[
  \z'_k < -\ee c, \qcom{or}
  -\ee w < \z'_k < 0,
\]
then we immediately collapse it into a shock.

Because the character of the solution will change, since a compression
has been replaced by a shock, we will need to re-check that no new
emergent compressions are outside the interval $[-\ee c,-\ee w]$.
This in turn means that we could in principle need to repeat this step
(that is solve the dgRS) up to $N-1$ times.  During this process, once
any emerging virtual width has been set to zero, we \emph{do not}
allow it to again become non-zero: this guarantees that eventually all
compressions emerging from the interaction have strength
$\z\in[-\ee c,-\ee w]$.  Note also, however, that this could in
principle break uniqueness of the resolution of this interaction: if
two or more emerging waves (of different families) are moderate
compressions, there may be some choice as to which ones should be forced
to become shocks.

\subsection{Avoidance of Overlaps}
\label{sec:avoid}

We wish to avoid overlap of adjacent waves of the same family
anywhere in the approximation.  Such an overlap is nonphysical because
it indicates that characteristics would have overlapped previously, so
a shock should have formed earlier.  We enforce this condition
inductively: first upon initialization, and then by restricting
virtual wave widths as necessary when resolving interactions.

We use the same notation as above: that is, we assume waves
$\gamma_{j-1}$ and $\gamma_j$ interact at the space-time point
$(x^!,t^!)$ calculated in \eqref{xt!}.  If the edgemost emerging waves
are simple, and of the same family of the adjacent waves, then their
combined widths must be chosen small enough at time $t^!$ that they do
not overlap.  Thus we must also consider the adjacent waves
$\gamma_{j-2}$ and $\gamma_{j+1}$.  We will consider only
$\gamma_{j+1}$, the case of $\gamma_{j-2}$ being essentially the same.

Thus suppose that $\gamma_{j+1}$, which is to the right of the
interaction, is a $k_{j+1}$-simple wave.  The \emph{position}
$x^p_{j+1}(t)$ and \emph{left virtual edge} $x^{e-}_{j+1}(t)$ of that
wave are the segments
\[
  \begin{aligned}
    x^p_{j+1}(t) &= x^c_{j+1} + s_{k_{j+1}}\,(t-t^c_{j+1}),\com{and}\\
    x^{e-}_{j+1}(t) &= x^c_{j+1} + \lambda_{k_{j+1}}(u_j)\,(t-t^c_{j+1}),
  \end{aligned}
\]
respectively.  We check for overlap if the rightmost wave emerging
from the interaction is also a $k_{j+1}$-simple wave.  The interaction
occurs at $(x^!,t^!)$, and so, if
\[
  x^{e-}_{j+1}(t^!) \le x^!,
\]
then there is an overlap, and we conclude that the wave
$\gamma_{j+1}$, which is \emph{not} part of the interaction, is too
wide.  If this is the case, we \emph{discontinuously} reduce the width
of $\gamma_{j+1}$ at time $t^!$, while leaving the position and speed
unchanged.  Note that if $\g_{j+1}$ is not simple, so is a shock or
contact, it has zero width and so there will be no overlap.  This is
illustrated in Figure~\ref{fig:ol}: the widths of (both) simple waves
of the fourth (green) family are decreased to avoid overlap.

\begin{figure}[thb]
  \centering
\begin{tikzpicture}

\begin{groupplot}[group style={group size=2 by 1}]
\nextgroupplot[
hide x axis,
hide y axis,
tick align=outside,
tick pos=left,
xmin=-0.16945, xmax=0.25845,
ymin=-0.00725, ymax=0.04225,
]
\path [draw=extra, fill=extra, opacity=0.2]
(axis cs:-0.025,0)
--(axis cs:0.025,0)
--(axis cs:0.011,-0.005)
--(axis cs:-0.03,-0.005)
--(axis cs:-0.03,-0.005)
--(axis cs:-0.025,0)
--cycle;

\path [draw=bakShock, fill=bakShock, opacity=0.2]
(axis cs:-0.025,0)
--(axis cs:0.025,0)
--(axis cs:-0.083,0.04)
--(axis cs:-0.117,0.04)
--(axis cs:-0.117,0.04)
--(axis cs:-0.025,0)
--cycle;

\path [draw=fwdSimple, fill=fwdSimple, opacity=0.2]
(axis cs:-0.025,0)
--(axis cs:0.025,0)
--(axis cs:0.177,0.04)
--(axis cs:0.103,0.04)
--(axis cs:0.103,0.04)
--(axis cs:-0.025,0)
--cycle;

\path [draw=fwdSimple, fill=fwdSimple, opacity=0.2]
(axis cs:0.005,0)
--(axis cs:0.055,0)
--(axis cs:0.239,0.04)
--(axis cs:0.157,0.04)
--(axis cs:0.157,0.04)
--(axis cs:0.005,0)
--cycle;

\path [draw=fwdSimple, fill=fwdSimple, opacity=0.2]
(axis cs:0.005,0)
--(axis cs:0.055,0)
--(axis cs:0.032,-0.005)
--(axis cs:-0.014,-0.005)
--(axis cs:-0.014,-0.005)
--(axis cs:0.005,0)
--cycle;

\addplot [thick, extra]
table {%
0 0
-0.00325 -0.005
};
\addplot [thick, extra]
table {%
0 0
-0.0095 -0.005
};
\addplot [thick, bakShock]
table {%
0 0
-0.1 0.04
};
\addplot [thick, contact]
table {%
0 0
0 0.04
};
\addplot [thick, extra]
table {%
0 0
0.034 0.04
};
\addplot [thick, fwdSimple]
table {%
0 0
0.14 0.04
};
\addplot [thick, fwdSimple]
table {%
0.03 0
0.198 0.04
};
\addplot [thick, fwdSimple]
table {%
0.03 0
0.009 -0.005
};
\addplot [line width=0.32pt, black, dashed]
table {%
-0.15 0
0.2 0
};

\nextgroupplot[
hide x axis,
hide y axis,
tick align=outside,
tick pos=left,
xmin=-0.168825, xmax=0.245325,
ymin=-0.00725, ymax=0.04225,
]
\path [draw=extra, fill=extra, opacity=0.2]
(axis cs:-0.025,0)
--(axis cs:0.025,0)
--(axis cs:0.011,-0.005)
--(axis cs:-0.03,-0.005)
--(axis cs:-0.03,-0.005)
--(axis cs:-0.025,0)
--cycle;

\path [draw=bakShock, fill=bakShock, opacity=0.2]
(axis cs:-0.025,0)
--(axis cs:0.025,0)
--(axis cs:-0.083,0.04)
--(axis cs:-0.117,0.04)
--(axis cs:-0.117,0.04)
--(axis cs:-0.025,0)
--cycle;

\path [draw=fwdSimple, fill=fwdSimple, opacity=0.2]
(axis cs:-0.0125,0)
--(axis cs:0.0125,0)
--(axis cs:0.1645,0.04)
--(axis cs:0.1155,0.04)
--(axis cs:0.1155,0.04)
--(axis cs:-0.0125,0)
--cycle;

\path [draw=fwdSimple, fill=fwdSimple, opacity=0.2]
(axis cs:0.0175,0)
--(axis cs:0.0425,0)
--(axis cs:0.2265,0.04)
--(axis cs:0.1695,0.04)
--(axis cs:0.1695,0.04)
--(axis cs:0.0175,0)
--cycle;

\path [draw=fwdSimple, fill=fwdSimple, opacity=0.2]
(axis cs:0.005,0)
--(axis cs:0.055,0)
--(axis cs:0.032,-0.005)
--(axis cs:-0.014,-0.005)
--(axis cs:-0.014,-0.005)
--(axis cs:0.005,0)
--cycle;

\addplot [thick, extra]
table {%
0 0
-0.00325 -0.005
};
\addplot [thick, extra]
table {%
0 0
-0.0095 -0.005
};
\addplot [thick, bakShock]
table {%
0 0
-0.1 0.04
};
\addplot [thick, contact]
table {%
0 0
0 0.04
};
\addplot [thick, extra]
table {%
0 0
0.034 0.04
};
\addplot [thick, fwdSimple]
table {%
0 0
0.14 0.04
};
\addplot [thick, fwdSimple]
table {%
0.03 0
0.198 0.04
};
\addplot [thick, fwdSimple]
table {%
0.03 0
0.009 -0.005
};
\addplot [line width=0.32pt, black, dashed]
table {%
-0.15 0
0.2 0
};
\end{groupplot}

\end{tikzpicture}
  \caption{Avoiding overlap: reduction of widths}
  \label{fig:ol}
\end{figure}

This reduction of virtual width of $\gamma_{j+1}$ has the effect of
moving the center $(x^c_{j+1},t^c_{j+1})$ closer to the point
$\big(x^p_{j+1}(t^!),t^!\big)$, and simultaneously increasing
$x^{e-}_{j+1}(t^!)$.  As in the last conclusion of
Lemma~\ref{lem:center}, this is a consequence of scale invariance, and
indeed the linear dependence of center position on width allows us to
calculate exactly the minimum adjustment necessary to undo the
overlap, namely
\[
  \frac{w_{j+1}(t^!+)}{w_{j+1}(t^!-)} <
  \frac{x^p_{j+1}(t^!)-x^!}{x^p_{j+1}(t^!)-x^{e-}_{j+1}(t^!-)} < 1.
\]
Because we have left the position $x^p_{j+1}(t^!)$ and speed $s_{j+1}$
of $\gamma_{j+1}$ unchanged, the new center of $\gamma_{j+1}$ will
remain on the path of the wave, and the only other change to
$\gamma_{j+1}$ may be that the next time of interaction may change, if
$\g_{j+1}$ is a compression which may collapse earlier after the
reduction in virtual width.

Having reduced the width of $\gamma_{j+1}$ as necessary, we now avoid
overlap by choosing the width of the (rightmost) emerging $k$-simple
wave $\g_k'$ to satisfy \emph{both} \eqref{wid} and the further
restriction
\begin{equation}
  \label{noOL}
  {x'}^{e+}_k(t^!+)\le x^{e-}_{j+1}(t^!+),
\end{equation}
which is always possible because the width of $\g_{j+1}$ has
previously been reduced.  Since the right edge of the emerging simple
wave and the left edge of $\gamma_{j+1}$ propagate with the same speed
$\la_{k_{j+1}}(u_j)$, these two waves will not overlap at later times.

\subsection{Multi-rarefactions}

As with compressions, it is possible that after a nonlinear
interaction, say crossing a strong shock of a different family, a
rarefaction of positive width and large strength emerges.  In this
case again the residual could grow, and an adjustment must be made so
that the approximation remains near a solution, that is so that the
residual remains small.

The solution is to split the moderate rarefaction that emerges from the
interaction into smaller rarefactions, as in Lemma~\ref{lem:dgRS}.
However, we cannot do this exactly as in that lemma, because the
virtual center of the rarefaction to be split, $(x^c,t^c)$
\emph{differs from} the interaction point $(x^!,t^!)$, from which each
wave must emerge.  

To resolve this issue, we first locate the virtual center $(x^c,t^c)$
and actual ``background'' characteristics of the moderate rarefaction
to be split.  We then split this centered rarefaction into smaller
\emph{``split waves''} as in Lemma~\ref{lem:center}, which are still
centered at $(x^c,t^c)$, with virtual center satisfying $t^c<t^!$.  We
calculate the corresponding speeds, widths and positions of these
split waves using the virtual center $(x^c,t^c)$.  If these waves
could be propagated as is, we would have reduced the wave strengths
and residual without any change to the approximation.  However, we
cannot do this directly because all waves must emanate from the
interaction point $(x^!,t^!)$, which satisfies $t^!>t^c$.  Instead,
keeping the states exact as usual, we place each of these smaller
waves at $(x^!,t^!)$, but \emph{temporarily} assign each of them a
\emph{finite but inaccurate wavespeed}, so that they ``catch up'' to
where they should be in a relatively short time.  This assignation of
incorrect wavespeed has the effect of increasing the residual from the
minimum, but only until the time that the wavespeeds are corrected.
At the same time, we keep the virtual widths of the split waves as
calculated from $(x^c,t^c)$.  When the split waves reach the correct
position, they are reassigned the correct wave speeds, and the
multi-rarefaction reverts to several individual virtually centered
rarefactions.

This is illustrated in Figure~\ref{fig:mr}.  Here a backward (blue)
compression and shock collide, generating a (relatively strong)
forward (green) rarefaction of positive width, and with virtual center
$(x^c,t^c)$. This rarefaction is first split into smaller pieces
(grey/green), which however are centered at the wrong point, $(x^c,t^c)$.
To center these at the correct point, $(x^!,t^!)$, we linearly
interpolate the paths by temporarily changing the wavespeed, to get the
multirarefaction wave (red), which reverts to the actual rarefaction
at $t=t^\#$.

\begin{figure}[thb]
  \centering
\begin{tikzpicture}

\begin{axis}[
hide x axis,
hide y axis,
tick align=outside,
tick pos=left,
xmin=-0.29, xmax=0.59,
ymin=-0.1175, ymax=0.2675,
]
\path [draw=fwdSimple, fill=fwdSimple, opacity=0.2]
(axis cs:-0.05,0)
--(axis cs:0.05,0)
--(axis cs:0.55,0.25)
--(axis cs:0.2,0.25)
--(axis cs:0.2,0.25)
--(axis cs:-0.05,0)
--cycle;

\path [draw=fwdSimple, fill=fwdSimple, opacity=0.2]
(axis cs:-0.05,0)
--(axis cs:0.05,0)
--(axis cs:-0.15,-0.1)
--(axis cs:-0.15,-0.1)
--(axis cs:-0.15,-0.1)
--(axis cs:-0.05,0)
--cycle;

\path [draw=bakShock, fill=bakShock, opacity=0.2]
(axis cs:-0.05,0)
--(axis cs:0.05,0)
--(axis cs:0.25,-0.1)
--(axis cs:0.05,-0.1)
--(axis cs:0.05,-0.1)
--(axis cs:-0.05,0)
--cycle;

\addplot [thick, fwdShock]
table {%
0 0
-0.15 -0.1
};
\addplot [thick, multiRf]
table {%
-0.03 0
-0.15 -0.1
};
\addplot [thick, fwdSimple]
table {%
0.03 0
-0.15 -0.1
};
\addplot [thick, multiRf]
table {%
-0.03 0
0.15 0.15
};
\addplot [thick, multiRf]
table {%
0.03 0
0.3 0.15
};
\addplot [thick, fwdShock]
table {%
0.15 0.15
0.27 0.25
};
\addplot [thick, fwdShock]
table {%
0.3 0.15
0.48 0.25
};
\addplot [thick, fwdShock]
table {%
0 0
0.375 0.25
};
\addplot [thick, fwdShock]
table {%
0 0
0.15 0.15
};
\addplot [thick, fwdShock]
table {%
0 0
0.3 0.15
};
\addplot [thick, bakShock]
table {%
0 0
0.15 -0.1
};
\addplot [thick, bakShock]
table {%
0 0
0.3 -0.1
};
\addplot [thick, bakShock]
table {%
0 0
-0.25 0.2
};
\addplot [thick, multiRf]
table {%
-0.03 0
-0.15 -0.1
};
\addplot [thick, multiRf]
table {%
0 0
-0.15 -0.1
};
\addplot [thick, multiRf]
table {%
0.03 0
-0.15 -0.1
};
\addplot [line width=0.32pt, black, dashed]
table {%
-0.15 0
0.3 0
};
\addplot [line width=0.32pt, black, dashed]
table {%
0.05 0.15
0.4 0.15
};
\draw (axis cs:0.32,0) node[
  scale=1,
  anchor=base west,
  text=black,
  rotate=0.0
]{$t^!$};
\draw (axis cs:0.42,0.15) node[
  scale=1,
  anchor=base west,
  text=black,
  rotate=0.0
]{$t^{\#}$};
\end{axis}

\end{tikzpicture}
  \caption{Multi-rarefaction}
  \label{fig:mr}
\end{figure}

To make this precise, suppose we are resolving an interaction at the
point $(x^!,t^!)$, which yields a dgRS in which the emerging GNL
$k$-wave, say $\g$, is a moderate rarefaction of strength
$\z\ge \ee r$ of finite width $w$.  Denote the left and right states
adjacent to this wave by $u_\mp$, so that $u_+ = W_k(u_-,\z)$.  The
virtual center of $\g$ will be the point where the left and right edge
characteristics intersect, so that
\[
  x^c := {x}^{e-}(t^c) = {x}^p(t^c) = {x}^{e+}(t^c),
\]
while also
\[
  {x}^p(t^!) = x^!, \qquad w = w(t^!) = \dot w\,(t^!-t^c).
\]
We write this (unused) single jump solution as
\begin{equation}
  \label{unuse}
  u_J(x,t) := u_- + (u_+-u_-)\,H\big(x-x^c-\la_k(\ol u)\,(t-t^c)\big),
  \qquad  \ol u := W_k(u_-,\z/2).
\end{equation}
Here it is convenient to use the averaged wavespeed rather than the
actual least squares wavespeed determined by Lemma~\ref{lem:spdRes}.
Since $\dot w = \la_k(u_+)-\la_k(u_-)$, this yields
\begin{equation}
  \label{txc!}
  \begin{aligned}
  x^c = x^! &- \la_k(\ol u)\,(t^!-t^c), \com{where}\\
  t^!-t^c &= \frac{w}{\la_k(u_+)-\la_k(u_-)} \approx O(\ee s/\z).
  \end{aligned}
\end{equation}
This size estimate follows because the wave is created by an
interaction of some incident simple wave, which has maximum strength
$\ee s := \max\{\ee r,\ee c\}$, so we expect its width to be
\[
  w\approx \ee s, \com{while}
  \la_k(u_+)-\la_k(u_-) = O(\z).
\]

Since we are assuming this wave $\g$ has moderate strength $\z$, the
residual is too large and we wish to split it up.  In order to take
maximum advantage of symmetry, we split the large wave into an
\emph{odd} number $2m+1$ of smaller rarefactions $\g^{(i)}$, each of
strength $\z^{(i)}$, as follows.  First fix some constant
$0<\kappa<1/3$ to be chosen, and then choose $2m+1$ smaller strengths
$\z^{(-m)},\dots,\z^{(m)}$, so that
\[
  0 < \kappa\,\ee r \le \z^{(-i)} = \z^{(i)} \le \ee r, \quad i=0,\dots,m,
  \com{and} \sum_{i=-m}^m\z^{(i)} = \z.
\]
Here $\kappa$ is the \emph{minimum} ratio of the size of a split
rarefaction to the maximum rarefaction size, and this in turn places
an upper bound on the number of rarefactions in a split.  Next, set
$\ol u^{(0)} := \ol u$ and inductively define adjacent states by
\begin{equation}
  \label{u(i)}
  \begin{aligned}
    u^{(1)} := W_k\big(\ol u^{(0)},\z^{(0)}/2\big), &\qquad
    u^{(i+1)} := W_k\big( u^{(i)},\z^{(i)}\big),\\
    u^{(-1)} := W_k\big(\ol u^{(0)},-\z^{(0)}/2\big), &\qquad
    u^{(-i-1)} := W_k\big( u^{(-i)},-\z^{(i)}\big),
  \end{aligned}
\end{equation}
for $i=1,\dots,m$, and similarly define the ``midpoint'' of each such
wave by
\[
  \begin{aligned}
    \ol u^{(-i)} &:= W_k\big(u^{(-i-1)},\z^{(-i)}/2\big)
                   = W_k\big(u^{(-i)},-\z^{(-i)}/2\big),\\
    \ol u^{(i)} &:= W_k\big(u^{(i)},\z^{(i)}/2\big)
                  = W_k\big(u^{(i+1)},-\z^{(i)}/2\big),
  \end{aligned}
\]
$i=1,\dots,m$.  By the group property \eqref{lgp}, $u^{(m+1)}=u_+$ and
$u^{(-m-1)}=u_-$, while $u^{(0)}$ is left undefined.  The $i$-th split
wave is then given the calculated speed $s^{(i)}$, so has target
trajectory
\begin{equation}
  \label{s(i)}
  x^{(i)}(t) = x^c + s^{(i)}\,(t-t^c), \qquad
  s^{(i)} = \la_k\big(\ol u^{(i)}\big) + O\big(\z^{(i)\,2}\big),
\end{equation}
which emanates from the virtual center $(x^c,t^c)$, and we have used
\eqref{smRes}.   We denote this
discretized rarefaction by
\begin{equation}
  \label{rareSplit}
  \begin{gathered}
    u_D(x,t) = u_- + \sum_{i=-m}^m[u]^{(i)}\,
    H\big(x-x^c-s^{(i)}(t-t^c)\big), \com{where}\\
    [u]^{(0)} := u^{(1)}-u^{(-1)}, \quad
    [u]^{(i)} := u^{(i+1)}-u^{(i)}, \quad
    [u]^{(-i)} := u^{(-i)}-u^{(-i-1)}.
  \end{gathered}
\end{equation}
Although this discretized rarefaction $u_D$ is the preferred solution,
we cannot use it directly for $t\ge t^!$, because all waves must
emanate from the interaction point $(x^!,t^!)$.

Our solution is to use the states identified in \eqref{u(i)} to define
the waves, but to initiate them from $(x^!,t^!)$, while temporarily
assigning them an incorrect wavespeed $\wt s^{(i)}$.  That is, the
multi-rarefaction has (correct) states \eqref{u(i)} and actual
trajectory
\[
  \wt x^{(i)}(t) = x^! + \wt s^{(i)}\,(t-t^!).
\]
The actual and target trajectories \eqref{s(i)} then intersect at the
point
\[
  x^{(i)}(t^\#) = \wt x^{(i)}(t^\#), \com{that is}
  x^c + s^{(i)}\,(t^\# - t^c) =
  x^! + \wt s^{(i)}\,(t^\#-t^!),
\]
which determines each $\wt s^{(i)}$ in terms of $t^\#$.  Using
\eqref{txc!}, this simplifies to
\begin{equation}
  \label{ratio}
  \big(\wt s^{(i)}-s^{(i)}\big)\,(t^\#-t^!) =
  \big(s^{(i)}-\la_k(\ol u)\big)\,(t^!-t^c).
\end{equation}
Using these modified wavespeeds and initial point $(x^!,t^!)$, we thus
describe the \emph{multi-rarefaction} as
\begin{equation}
  \label{mr}
  u_M(x,t) := u_- +
  \sum_{i=-m}^m [u]^{(i)}\,
  H\big(x-x^!-\wt s^{(i)}(t-t^!)\big), \qquad t^!\le t\le t^\#.
\end{equation}
The multi-rarefaction is used until time $t=t^\#$, when the solution
reverts to the standard centered rarefaction $u_D$ given in
\eqref{rareSplit}.

Although we do not use the single jump \eqref{unuse} because its
residual is too large, we can use it to conveniently write the
multi-rarefaction as a linear interpolation of the single jump $u_J$
and the ``correct'' discretized rarefaction $u_D$, namely
\[
  u_M = \frac{t^\#-t}{t^\#-t^!}\,u_J + \frac{t-t^!}{t^\#-t^!}\,u_D,
  \qquad t\in[t^!,t^\#].
\]

\begin{lemma}
  \label{lem:mr}
  The residual of the multi-rarefaction $u_M$ satisfies the estimate
  \begin{equation}
    \label{mrRes1}
    \begin{aligned}
      \big\|\mc R(u_M)\big\|_{M} &\le K_M\,\frac{t^!-t^c}{t^\#-t^!}\,
    \sum_{i=1}^m\z^{(i)}\,\Big(\frac{\z^{(0)}}2 +
    \sum_{j=1}^{i-1}\z^{(j)} + \frac{\z^{(i)}}2\Big)
                   \\{}&\qquad{}{}
    + O\Big(\sum_{i=0}^m\z^{(i)}\sum_{l=0}^i
                         \z^{(l)}\sum_{j=0}^l\z^{(j)}\Big),
    \end{aligned}
  \end{equation}
  where $K_M$ is a constant depending only on the system.  In
  particular, if we choose $t^\#$ large enough that
  \begin{equation}
    \label{mrRes}
    \frac{t^!-t^c}{t^\#-t^!} = O(m\,\ee r),
    \com{then}
    \big\|\mc R(u_M)\big\|_{M} \le
    \sum_{i=-m}^m\z^{(i)}\,K_{mr}\,m^2\,\e^{2e_r},
  \end{equation}
  for a different constant $K_{mr}$.
\end{lemma}

Note that the sums appearing in \eqref{mrRes} are significant only
when $m$ is large, which in turn implies a dramatic increase in wave
strength, which is $O(\z/\ee s)$ across the interaction, which we
expect to be a rare occurrence, because the absolute strength of
shocks is limited by the compact domain $K_{\mc U}$ in state space
$\mc U$.  We are implicitly assuming that $m\,\ee r$ is small; if not,
we should not use a multi-rarefaction.  In any case, after the
multi-rarefaction $u_M$ reverts to the discretized rarefaction $u_D$
at time $t=t^\#$, the residual decreases to
$O\big(\sum\z^{(i)\,3}\big)$.  Also, the individual terms increase as
$|i|$ does, so the middle wave strengths ($i\sim0$) should be chosen
as large as possible, while the outside waves ($i\sim m$) should be as
small as possible, within the allowable range.  The reason for the
restriction $\kappa<1/3$ is that any rarefaction slightly stronger
than $\ee r$ must be split into at least three pieces, so the minimum
split strength should be (strictly) smaller than $\ee r/3$.

\begin{proof}
  According to \eqref{Res*s}, if we define
  \[
    \begin{aligned}
      {}[q]^{(0)} &:= q\big(u^{(1)}\big)-q\big(u^{(-1)}\big), \\
      [q]^{(i)} &:= q\big(u^{(i+1)}\big)-q\big(u^{(i)}\big), \\
      [q]^{(-i)} &:= q\big(u^{(-i)}\big)-q\big(u^{(-i-1)}\big),
    \end{aligned}
  \]
  $i=1,\dots,m$, and $\wh{[q]}^{(i)}:=A^{(i)}[q]^{(i)}$, then the
  vector residual of each jump in \eqref{mr} is
  \begin{equation}
    \label{R(i)}
    \begin{aligned}
      R^{(i)} &:= \Big(\la_k\big(\ol u^{(i)}\big) - \wt s^{(i)}\Big)\,
                \wh{[q]}^{(i)} + O\big(\z^{(i)\,3}\big) \\
              &= \Big(\la_k\big(\ol u^{(i)}\big) - \la_k(\ol u)\Big)\,
                \frac{t^!-t^c}{t^\#-t^!}\,
                \wh{[q]}^{(i)} + O\big(\z^{(i)\,3}\big),
    \end{aligned}
  \end{equation}
  $i = -m,\dots,m$, where we have used \eqref{s(i)} and \eqref{ratio}.
  Each $\wh{[q]}^{(i)} = O\big(\z^{(i)}\big)$, and by construction,
  $\ol u^{(0)}=\ol u$.  We now approximate each
  $\la_k\big(\ol u^{(i)}\big)$ by Taylor expansion.  Using
  \eqref{u(i)}, we get
  \[
    \la_k\big(u^{(\pm 1)}\big)
    = \la_k(\ol u) \pm \frac{\z^{(0)}}2\,\dot\la_k(\ol u)
    + O\big(\z^{(0)\,2}\big),
  \]
  and similarly, 
  \[
    \begin{aligned}
      \la_k\big(u^{(\pm i\pm 1)}\big)
      &= \la_k\big(u^{(\pm i)}\big) \pm
        \z^{(i)}\,\dot\la_k\big(u^{(\pm i)}\big)
        + O\big(\z^{(i)\,2}\big)\\
      &= \la_k\big(u^{(\pm i)}\big) \pm
        \z^{(i)}\,\dot\la_k(\ol u)
        + O\Big(\z^{(i)}\sum_{j=0}^i\z^{(j)}\Big).
    \end{aligned}
  \]
  Inductively combining these expressions, it follows after
  simplification that
  \[
    \la_k\big(\ol u^{(\pm i)}\big) - \la_k(\ol u)
    = \pm\Big(\frac{\z^{(0)}}2+\sum_{j=1}^{i-1}\z^{(j)} +
    \frac{\z^{(i)}}2\Big)\,\dot\la_k(\ol u)
    + O\Big(\sum_{l=0}^i\z^{(l)}\sum_{j=0}^l\z^{(j)}\Big),
  \]
  for $i=1,\dots,m$.

  Next, expand $q(u)$ along the curve $W_k(u,\z)$, to get
  \[
    \begin{aligned}
      {}[q]^{(\pm i)}
      &= q\Big(W_k(\ol u^{(\pm i)}),\z^{(\pm i)}/2\Big)
        - q\Big(W_k(\ol u^{(\pm i)}),-\z^{(\pm i)}/2\Big)\\
      &= \dot q\big(\ol u^{(\pm i)}\big)\,\z^{(\pm i)}
        + O\big(\z^{(\pm i)\,3}\big),
    \end{aligned}
  \]
  where $\dot q(u) := Dq(u)\,\dot u = Dq(u)\,r_k(u)$ is the derivative
  along the wave curve.

  We now use these expressions in \eqref{R(i)} to get
  \[
    \begin{aligned}
      \big|R^{(i)}\big| = \z^{(i)}\,
    \Big(\frac{\z^{(0)}}2&+\sum_{j=1}^{i-1}\z^{(j)} +
    \frac{\z^{(i)}}2\Big)\,\frac{t^!-t^c}{t^\#-t^!}\,
    \big|\dot\la_k(\ol u)\,A^{(i)}\dot q\big(\ol u^{(i)}\big)\big|
    \\&{}+ O\Big(\z^{(i)}\sum_{l=0}^i\z^{(l)}\sum_{j=0}^l\z^{(j)}\Big),
    \end{aligned}
  \]
  for $i=-m,\dots,m$; note that the leading (quadratic) term vanishes
  for $i=0$.  Finally, the residual of $u_M$ is
  \[
    \mc R(u_M) = \sum_{i=-m}^m R^{(i)}\,
    \delta_{x^!+\wt s^{(i)}(t-t^!)},
  \]
  so the triangle inequality together with $\|\delta\|=1$ gives the
  expression \eqref{mrRes1}.  Finally, estimate each \emph{interior} sum by
  \[
    \sum_{j=0}^l\z^{(j)} \le l\,\max\z^{(j)} \le m\,\ee r,
  \]
  and combine the two terms to get the estimate \eqref{mrRes}.
\end{proof}

We remark that we could avoid multi-rarefactions by treating
moderate rarefactions in the same way as strong compressions: that
is, if a rarefaction has strength $\z\ge \ee r$ after an interaction,
we could set the width to zero, which effectively centers the
rarefaction at $(x^!,t^!)$, and which could then be split as a
centered rarefaction with no temporary increase in residual.  However,
we believe that the use of multi-rarefactions allows a more accurate
$L^1$ profile of the actual solution.  Thus although use of
multi-rarefactions has a cost in terms of the residual, we believe
that they will yield a better rate of convergence in $L^1$.

\subsection{Composite Waves}
\label{sec:comp}

Because we are treating a general $N\times N$ system \eqref{eqsys},
the number of waves can in principle become unbounded in finite time.
This is particularly evident for $N\ge 4$, or even $N=3$ for general
systems.  In particular, there are known examples in which, if all
waves are discretized, infinitely many waves and interactions
accumulate around a single finite point.  Three such pictures of wave
patterns are shown in Figure~\ref{fig:accum}.  The left-most picture
can be realized exactly as waves in a linearly degenerate
(geometrically) nonlinear elastic string~\cite{Ystr}, which we expect
does not blow up despite infinitely many interactions.  The middle
pattern occurs in the example of the solution of a non-physical system
in which the amplitude blows up in finite time~\cite{Jens}.  The third
pattern, which we expect is non-realizable in a strictly hyperbolic
system, shows a wave pattern in which there are finitely many waves at
any one time, but which nevertheless has a finite accumulation of
interaction points.

\begin{figure}[thb]
  \centering
\begin{tikzpicture}[scale=0.7]

\begin{groupplot}[group style={group size=3 by 1}]
\nextgroupplot[
hide x axis,
hide y axis,
tick align=outside,
tick pos=left,
xmin=-0.48, xmax=1.28,
ymin=-0.049925989400576, ymax=1.0484457774121,
]
\addplot [line width=0.48pt, black]
table {%
0 0
0.396 0.99
};
\addplot [line width=0.48pt, black]
table {%
0.8 0
0.404 0.99
};
\addplot [line width=0.48pt, black]
table {%
0.64 0
0.128 0.32
-0.4 0.65
};
\addplot [line width=0.48pt, black]
table {%
0.16 0
0.672 0.32
1.2 0.65
};
\addplot [line width=0.48pt, black]
table {%
0.128 0.32
0.5632 0.592
1.2 0.99
};
\addplot [line width=0.48pt, black]
table {%
0.672 0.32
0.2368 0.592
-0.4 0.99
};
\addplot [line width=0.48pt, black]
table {%
0.5632 0.592
0.30208 0.7552
-0.0735999999999999 0.99
};
\addplot [line width=0.48pt, black]
table {%
0.2368 0.592
0.49792 0.7552
0.8736 0.99
};
\addplot [line width=0.48pt, black]
table {%
0.30208 0.7552
0.458752 0.85312
0.67776 0.99
};
\addplot [line width=0.48pt, black]
table {%
0.49792 0.7552
0.341248 0.85312
0.12224 0.99
};
\addplot [line width=0.48pt, black]
table {%
0.458752 0.85312
0.3647488 0.911872
0.239744 0.99
};
\addplot [line width=0.48pt, black]
table {%
0.341248 0.85312
0.4352512 0.911872
0.560256 0.99
};
\addplot [line width=0.48pt, black]
table {%
0.3647488 0.911872
0.42115072 0.9471232
0.4897536 0.99
};
\addplot [line width=0.48pt, black]
table {%
0.4352512 0.911872
0.37884928 0.9471232
0.3102464 0.99
};
\addplot [line width=0.48pt, black]
table {%
0.42115072 0.9471232
0.387309568 0.96827392
0.35254784 0.99
};
\addplot [line width=0.48pt, black]
table {%
0.37884928 0.9471232
0.412690432 0.96827392
0.44745216 0.99
};
\addplot [line width=0.48pt, black]
table {%
0.387309568 0.96827392
0.4076142592 0.980964352
0.422071296 0.99
};
\addplot [line width=0.48pt, black]
table {%
0.412690432 0.96827392
0.3923857408 0.980964352
0.377928704 0.99
};
\addplot [line width=0.48pt, black]
table {%
0.4076142592 0.980964352
0.39543144448 0.9885786112
0.3931572224 0.99
};
\addplot [line width=0.48pt, black]
table {%
0.3923857408 0.980964352
0.40456855552 0.9885786112
0.4068427776 0.99
};
\addplot [line width=0.48pt, black]
table {%
0.39543144448 0.9885786112
0.402741133312 0.99314716672
0.39770566656 0.99
};
\addplot [line width=0.48pt, black]
table {%
0.40456855552 0.9885786112
0.397258866688 0.99314716672
0.40229433344 0.99
};
\addplot [line width=0.48pt, black]
table {%
0.402741133312 0.99314716672
0.3983553200128 0.995888300032
0.407776600064 0.99
};
\addplot [line width=0.48pt, black]
table {%
0.397258866688 0.99314716672
0.4016446799872 0.995888300032
0.392223399936 0.99
};
\addplot [line width=0.48pt, black]
table {%
0.3983553200128 0.995888300032
0.40098680799232 0.9975329800192
0.3889340399616 0.99
};
\addplot [line width=0.48pt, black]
table {%
0.4016446799872 0.995888300032
0.39901319200768 0.9975329800192
0.4110659600384 0.99
};
\addplot [line width=0.48pt, black]
table {%
0.40098680799232 0.9975329800192
0.399407915204608 0.99851978801152
0.41303957602304 0.99
};
\addplot [line width=0.48pt, black]
table {%
0.39901319200768 0.9975329800192
0.400592084795392 0.99851978801152
0.38696042397696 0.99
};
\addplot [line width=0.48pt, black]
table {%
0.672 0.32
0.724164383561644 0.692602739726027
};
\addplot [line width=0.48pt, black]
table {%
0.128 0.32
0.0758356164383562 0.692602739726027
};
\addplot [line width=0.48pt, black]
table {%
0.5632 0.592
0.653097028972135 0.99
};
\addplot [line width=0.48pt, black]
table {%
0.724164383561644 0.692602739726027
0.656990689785054 0.99
};
\addplot [line width=0.48pt, black]
table {%
0.0758356164383562 0.692602739726027
0.143009310214946 0.99
};
\addplot [line width=0.48pt, black]
table {%
0.2368 0.592
0.146902971027865 0.99
};
\addplot [line width=0.48pt, black]
table {%
0.724164383561644 0.692602739726027
0.843123287671233 0.99
};
\addplot [line width=0.48pt, black]
table {%
0.0758356164383562 0.692602739726027
-0.0431232876712328 0.99
};
\addplot [line width=0.48pt, black]
table {%
0.724164383561644 0.692602739726027
0.603407664350184 0.770010893066707
0.260224657534247 0.99
};
\addplot [line width=0.48pt, black]
table {%
0.0758356164383562 0.692602739726027
0.196592335649816 0.770010893066707
0.539775342465753 0.99
};
\addplot [line width=0.48pt, black]
table {%
0.603407664350184 0.770010893066707
0.693618460316963 0.827838326378744
0.946590671166122 0.99
};
\addplot [line width=0.48pt, black]
table {%
0.196592335649816 0.770010893066707
0.106381539683037 0.827838326378744
-0.146590671166122 0.99
};
\addplot [line width=0.48pt, black]
table {%
0.693618460316963 0.827838326378744
0.62622686702012 0.871038065671592
0.440646249467804 0.99
};
\addplot [line width=0.48pt, black]
table {%
0.106381539683037 0.827838326378744
0.17377313297988 0.871038065671592
0.359353750532196 0.99
};
\addplot [line width=0.48pt, black]
table {%
0.62622686702012 0.871038065671592
0.676571471547938 0.903310248061219
0.811807484572435 0.99
};
\addplot [line width=0.48pt, black]
table {%
0.17377313297988 0.871038065671592
0.123428528452063 0.903310248061219
-0.0118074845724355 0.99
};
\addplot [line width=0.48pt, black]
table {%
0.676571471547938 0.903310248061219
0.638961747826194 0.927419045318747
0.54133545852344 0.99
};
\addplot [line width=0.48pt, black]
table {%
0.123428528452063 0.903310248061219
0.161038252173806 0.927419045318747
0.258664541476561 0.99
};
\addplot [line width=0.48pt, black]
table {%
0.638961747826194 0.927419045318747
0.66705793274397 0.945429420266039
0.736588037128949 0.99
};
\addplot [line width=0.48pt, black]
table {%
0.161038252173806 0.927419045318747
0.13294206725603 0.945429420266039
0.063411962871051 0.99
};
\addplot [line width=0.48pt, black]
table {%
0.66705793274397 0.945429420266039
0.64606879672228 0.958883994638917
0.597527828358991 0.99
};
\addplot [line width=0.48pt, black]
table {%
0.13294206725603 0.945429420266039
0.15393120327772 0.958883994638917
0.202472171641009 0.99
};
\addplot [line width=0.48pt, black]
table {%
0.64606879672228 0.958883994638917
0.661748641947754 0.968935177475759
0.694609765085569 0.99
};
\addplot [line width=0.48pt, black]
table {%
0.15393120327772 0.958883994638917
0.138251358052247 0.968935177475759
0.105390234914431 0.99
};
\addplot [line width=0.48pt, black]
table {%
0.661748641947754 0.968935177475759
0.650035079939258 0.976443871070949
0.628887518809938 0.99
};
\addplot [line width=0.48pt, black]
table {%
0.138251358052247 0.968935177475759
0.149964920060742 0.976443871070949
0.171112481190062 0.99
};
\addplot [line width=0.48pt, black]
table {%
0.650035079939258 0.976443871070949
0.658785646858373 0.982053208839612
0.671182641068579 0.99
};
\addplot [line width=0.48pt, black]
table {%
0.149964920060742 0.976443871070949
0.141214353141627 0.982053208839612
0.128817358931421 0.99
};

\nextgroupplot[
hide x axis,
hide y axis,
tick align=outside,
tick pos=left,
xmin=-0.48, xmax=1.28,
ymin=-0.04987664900096, ymax=1.04740962902016,
]
\addplot [line width=0.48pt, black]
table {%
0 0
0.396 0.99
};
\addplot [line width=0.48pt, black]
table {%
0.8 0
0.404 0.99
};
\addplot [line width=0.48pt, black]
table {%
0.64 0
0.128 0.32
-0.4 0.65
};
\addplot [line width=0.48pt, black]
table {%
0.128 0.32
0.5632 0.592
1.2 0.99
};
\addplot [line width=0.48pt, black]
table {%
0.5632 0.592
0.30208 0.7552
-0.0735999999999999 0.99
};
\addplot [line width=0.48pt, black]
table {%
0.30208 0.7552
0.458752 0.85312
0.67776 0.99
};
\addplot [line width=0.48pt, black]
table {%
0.458752 0.85312
0.3647488 0.911872
0.239744 0.99
};
\addplot [line width=0.48pt, black]
table {%
0.3647488 0.911872
0.42115072 0.9471232
0.4897536 0.99
};
\addplot [line width=0.48pt, black]
table {%
0.42115072 0.9471232
0.387309568 0.96827392
0.35254784 0.99
};
\addplot [line width=0.48pt, black]
table {%
0.387309568 0.96827392
0.4076142592 0.980964352
0.422071296 0.99
};
\addplot [line width=0.48pt, black]
table {%
0.4076142592 0.980964352
0.39543144448 0.9885786112
0.3931572224 0.99
};
\addplot [line width=0.48pt, black]
table {%
0.39543144448 0.9885786112
0.402741133312 0.99314716672
0.39770566656 0.99
};
\addplot [line width=0.48pt, black]
table {%
0.402741133312 0.99314716672
0.3983553200128 0.995888300032
0.407776600064 0.99
};
\addplot [line width=0.48pt, black]
table {%
0.3983553200128 0.995888300032
0.40098680799232 0.9975329800192
0.3889340399616 0.99
};

\nextgroupplot[
hide x axis,
hide y axis,
tick align=outside,
tick pos=left,
xmin=-10.96, xmax=18.96,
ymin=-0.594062862967084, ymax=12.4753201223088,
]
\addplot [line width=0.48pt, black]
table {%
0 0
1.00903155966486 3.03675078193788
1.6396762844554 4.93472002064905
1.6396762844554 4.93472002064905
2.23508656445788 6.72664874285069
2.60721798945943 7.84660419422672
2.60721798945943 7.84660419422672
2.95855824055166 8.90398715912124
3.1781458974843 9.56485151218032
3.1781458974843 9.56485151218032
3.3854650792039 10.1887930342937
3.51503956777864 10.5787564856145
3.51503956777864 10.5787564856145
3.63737466310366 10.9469324785538
3.71383409768179 11.1770424741409
3.71383409768179 11.1770424741409
3.78602170436655 11.3942961060481
3.83113895854453 11.53007962599
3.83113895854453 11.53007962599
3.87373548855114 11.6582768509097
3.90035831980527 11.7384001164845
};
\addplot [line width=0.48pt, black]
table {%
8 0
6.99096844033514 3.03675078193788
6.3603237155446 4.93472002064905
6.3603237155446 4.93472002064905
5.76491343554212 6.72664874285069
5.39278201054057 7.84660419422672
5.39278201054057 7.84660419422672
5.04144175944834 8.90398715912124
4.8218541025157 9.56485151218032
4.8218541025157 9.56485151218032
4.61453492079611 10.1887930342937
4.48496043222136 10.5787564856145
4.48496043222136 10.5787564856145
4.36262533689635 10.9469324785538
4.28616590231821 11.1770424741409
4.28616590231821 11.1770424741409
4.21397829563345 11.3942961060481
4.16886104145547 11.53007962599
4.16886104145547 11.53007962599
4.12626451144886 11.6582768509097
4.09964168019473 11.7384001164845
};
\addplot [line width=0.32pt, black]
table {%
4 0
1.00903155966486 3.03675078193788
4 4.93472002064905
4 4.93472002064905
2.23508656445788 6.72664874285069
4 7.84660419422672
4 7.84660419422672
2.95855824055166 8.90398715912124
4 9.56485151218032
4 9.56485151218032
3.3854650792039 10.1887930342937
4 10.5787564856145
4 10.5787564856145
3.63737466310366 10.9469324785538
4 11.1770424741409
4 11.1770424741409
3.78602170436655 11.3942961060481
4 11.53007962599
4 11.53007962599
3.87373548855114 11.6582768509097
4 11.7384001164845
};
\addplot [line width=0.32pt, black]
table {%
4 0
6.99096844033514 3.03675078193788
4 4.93472002064905
4 4.93472002064905
5.76491343554212 6.72664874285069
4 7.84660419422672
4 7.84660419422672
5.04144175944834 8.90398715912124
4 9.56485151218032
4 9.56485151218032
4.61453492079611 10.1887930342937
4 10.5787564856145
4 10.5787564856145
4.36262533689635 10.9469324785538
4 11.1770424741409
4 11.1770424741409
4.21397829563345 11.3942961060481
4 11.53007962599
4 11.53007962599
4.12626451144886 11.6582768509097
4 11.7384001164845
};
\addplot [line width=0.32pt, black]
table {%
1.00903155966486 3.03675078193788
-2.13684166041597 4.93472002064905
-0.472064743108921 0
};
\addplot [line width=0.32pt, black]
table {%
6.99096844033514 3.03675078193788
10.136841660416 4.93472002064905
8.47206474310892 0
};
\addplot [very thin, black]
table {%
-2.13684166041597 4.93472002064905
-2.89537101418625 5.93472002064905
};
\addplot [very thin, black]
table {%
10.136841660416 4.93472002064905
10.8953710141863 5.93472002064905
};
\addplot [line width=0.32pt, black]
table {%
2.95855824055166 8.90398715912124
1.86317933345918 9.56485151218032
2.23508656445788 6.72664874285069
};
\addplot [line width=0.32pt, black]
table {%
5.04144175944834 8.90398715912124
6.13682066654082 9.56485151218032
5.76491343554212 6.72664874285069
};
\addplot [very thin, black]
table {%
1.86317933345918 9.56485151218032
1.58227259100175 9.89818484551365
};
\addplot [very thin, black]
table {%
6.13682066654082 9.56485151218032
6.41772740899825 9.89818484551365
};
\addplot [line width=0.32pt, black]
table {%
3.63737466310366 10.9469324785538
3.25596865397266 11.1770424741409
3.3854650792039 10.1887930342937
};
\addplot [line width=0.32pt, black]
table {%
4.36262533689635 10.9469324785538
4.74403134602734 11.1770424741409
4.61453492079611 10.1887930342937
};
\addplot [very thin, black]
table {%
3.25596865397266 11.1770424741409
3.25298185585117 11.3770424741409
};
\addplot [very thin, black]
table {%
4.74403134602734 11.1770424741409
4.74701814414883 11.3770424741409
};
\addplot [line width=0.32pt, black]
table {%
3.87373548855114 11.6582768509097
3.74093163149371 11.7384001164845
3.78602170436655 11.3942961060481
};
\addplot [line width=0.32pt, black]
table {%
4.12626451144886 11.6582768509097
4.25906836850629 11.7384001164845
4.21397829563345 11.3942961060481
};
\addplot [very thin, black]
table {%
3.74093163149371 11.7384001164845
3.82260823286156 11.8812572593417
};
\addplot [very thin, black]
table {%
4.25906836850629 11.7384001164845
4.17739176713844 11.8812572593417
};
\addplot [line width=0.48pt, black]
table {%
-9.6 0
3.66121828733793 11.7384001164845
};
\addplot [line width=0.48pt, black]
table {%
17.6 0
4.33878171266207 11.7384001164845
};
\end{groupplot}

\end{tikzpicture}
  \caption{Wave patterns with accumulation of interactions}
  \label{fig:accum}
\end{figure}

It follows that to continue our mFT approximation for sufficiently
large times, we must find a way to limit the number of waves in the
scheme.  However, we also wish to avoid non-physical waves and to keep
states exact throughout the scheme.  We can achieve this consistently
because we are allowed to adjust the speeds (with a residual penalty)
and virtual widths (which affects future wave patterns only), while
preserving exact states and wave strengths.

Our assumption that the system is globally strictly hyperbolic has two
important consequences for us: first, there is a natural ordering of
the wave families by increasing (actual) wavespeed; and second, the
vectors spanning waves of different families are always independent.
The upshot of this is that in our scheme, in which waves are
represented by jump discontinuities, we can \emph{superpose} waves of
different families, by giving them the same initial position and
wavespeed.

Specifically, suppose a dgRS has waves $\z_k$ as in
Lemma~\ref{lem:dgRS}, so each state $u_k$ is exactly connected to
$u_{k-1}$ as in \eqref{gRsol}.  Then we can allow two or more consecutive
waves to be superposed, say by setting $s_{k-i}=s_k=s_{k+1}=:s$, and
the approximation still has the piecewise constant form \eqref{pwc}.
In this case, we would have left and right states
\[
  u_\ell := u\big(\xi(t)-,t\big) = u_{k-i-1} \com{and}
  u_r := u\big(\xi(t)+,t) = u_{k+1},\qquad \xi'(t) = s,
\]
and although $u_\ell$ and $u_r$ are not connected by a single wave, we
do have
\begin{equation}
  \label{comp}
  u_r = \wt W_{k+1}\big(\wt W_k\big(\dots
  \wt W_{k-i}(u_\ell,\z_{k-i})..,\z_k\big),\z_{k+1}\big),
\end{equation}
where each $\wt W_j$ is either $W_j$ or $E_j$, given by \eqref{Wk},
\eqref{Ek}, respectively.  Note that this description is forced on us
by the ordering of waves, as long as we ensure that the speeds $s_k$
are non-decreasing in $k$.  Also, note that we can combine waves in
this way \emph{only if} they emanate from the same point, that is as a
result of a wave interaction.  We call such a superposition of jumps a
\emph{composite wave}.  We do not give the composite wave a single
strength, but rather treat it as several superposed waves (of
consecutive families) each having their own strengths.

We identify and name the \emph{base wave} or \emph{carrier wave} of a
composite wave to be that wave having largest strength; this need not
itself be a weak wave.  Since all other components are small, we
regard the base wave as dominant, and use the virtual width of the
base wave as that of the composite.  In this way, the smaller
components ``piggyback'' onto the carrier; we will call these
\emph{passengers}.

It remains to specify precisely when to use composite waves, which
consecutive waves should be included in a composite wave, and what
speed to assign to the composite wave.  We must use composite waves
often enough that interactions do not accumulate, but we wish to also
control the total residual.  Since we are effectively using an
incorrect wavespeed for several of the different components, according
to \eqref{Rs}, the strength of any such component must be small if we
are to control the residual.

After resolving an interaction by solving the dgRS, and carrying out
the checks above, including splitting into multi-rarefactions, we
check the strength of each emerging wave $\g_j$.  If the wave has
strength \emph{at least} $\ee p$, $|\z_j|\ge \ee p$, then that wave
\emph{may not} be a passenger, so it must be a base or carrier wave.
Thus any passenger is necessarily weaker than $\ee p$.

If there are any passengers in the emerging waves, each such passenger
must be associated with a carrier which emerges fron the same point.
There are three possibilities for each passenger $\g_j$: first, there
are base waves both to the left and right of $\g_j$; second, there is
a base wave either to the left or right of $\g_j$ emerging from the
interaction; and third, there are \emph{no base waves} emerging from
the interaction.  In the first case, we choose the \emph{stronger} of
the two base waves to be the carrier, and in the second case there is
no choice to make.  In the third case, there are no ``strong'' wave
candidates for the carrier; in this case we choose as the carrier the
\emph{largest} of the waves emerging from the interaction point.  In
cases in which two or more candidates have the same strength, we
merely enforce consistency by not allowing wave to cross, so that
\emph{all} waves between a passenger and its carrier are passengers
with the \emph{same} carrier.  This is illustrated in
Figure~\ref{fig:comp}: the solid waves in the left panel are strong
enough to be carriers, while the dotted waves become passengers, and
the backward carrier is stronger.  Both carriers become composite
waves, and we keep the states adjacent to these exact.  Effectively we
have just changed the speeds and widths of the passengers.

\begin{figure}[thb]
  \centering
  \begin{tikzpicture}[scale=1.5]
  
  \draw[thick] (-0.9,0.2) -- (1,2);
  \draw[thick] (-0.9,0.2) -- (-2.5,2);

  \draw[thick,dotted] (-0.9,0.2) -- (-3,1.8);
  \draw[thick,dotted] (-0.9,0.2) -- (0,2);
  \draw[thick,dotted] (-0.9,0.2) -- (-1,2);
  \draw[thick,dotted] (-0.9,0.2) -- (1.4,1.4);

  \draw[dashed] (-3,0.2) -- (1.4,0.2);

  \draw[->,draw=direct,thick] (-0.2,1.8) -- (-0.5,1.8);
  \draw[->,draw=direct,thick] (-2.9,1.8) -- (-2.6,1.8);
  \draw[->,draw=direct,thick] (1,1.35) -- (0.7,1.35);
  \draw[->,draw=direct,thick] (-1.1,1.8) -- (-1.4,1.8);

  \draw[thick] (-0.6,-0.2) -- (-0.9,0.2);
  \draw[thick] (-0.4,-0.2) -- (-0.9,0.2);

  \node at (-2.4,1) {$u_0$};
  \node at (-2,1.2) {$u_1$};
  \node at (-1.3,1.3) {$u_2$};
  \node at (-0.7,1.3) {$u_3$};
  \node at (-0.1,1.25) {$u_4$};
  \node at (0.5,1.2) {$u_5$};
  \node at (1.2,0.9) {$u_6$};
  
\end{tikzpicture}
\qquad
\begin{tikzpicture}[scale = 1.5]  
  \draw[thick] (-0.9,0.2) -- (1,2);
  \draw[thick] (-0.9,0.2) -- (-2.5,2);

  \draw[dashed] (-3,0.2) -- (1.4,0.2);

  \draw[thick] (-0.6,-0.2) -- (-0.9,0.2);
  \draw[thick] (-0.4,-0.2) -- (-0.9,0.2);

  \node at (-2,1) {$u_0$};
  \node at (-1,1.3) {$u_4$};
  \node at (0.8,1) {$u_6$};


  
\end{tikzpicture}  

  \caption{Composite waves}
  \label{fig:comp}
\end{figure}

Thus for every wave $\g_j$ emerging from the interaction, we identify
a carrier wave $\g_{b_j}$, satisfying the following properties:
\begin{itemize}
\item the passenger is weaker than the carrier, $|\z_j|\le|\z_{b_j}|$;
\item if a wave is not weak, it carries itself: $|\z_j|\ge \ee p$
  implies $b_j=j$;
\item waves do not cross: if $b_j<j$, then
  \[
    |\z_{j'}|<\ee p \com{and} b_{j'}=b_j \com{for each}
    b_j< j'\le j,
  \]
  and similarly for $b_j>j$.
\end{itemize}
In the third case above, in which no waves emerging from the
interaction are base waves, that is $|\z_j|<\ee p$ for all $j$, then
all waves have the same carrier, which is the strongest emerging wave.
We call such a composite wave a \emph{lonely weak wave}.

Notice that, like multi-rarefactions, composite waves are generally
short-lived: whenever a composite wave interacts with another wave,
the dgRS is resolved with left and right states that incorporate the
composite wave, and after resolution of the dgRS, a check is made for
new composite waves which are then assigned regardless of the various
waves entering the interaction.

Because passengers are weak waves, we generally expect that any
rarefaction that is split, either into centered rarefactions or into a
multi-rarefaction, \emph{will not} yield any waves weak enough to
become passengers.  We can ensure that this is the case by
appropriately relating the exponents $\ee r$ for rarefactions and
$\ee p$ for weak waves.  This is done in Section~\ref{sec:conv}
below.

We illustrate the effect of using composite waves in limiting the
number of waves in Figure~\ref{fig:noaccum}.  This shows an
accumulation of infinitely many interaction by self-similarity in a
$3\times 3$ system such as shown in Figure~\ref{fig:accum}.  On the
left, there is an accumulation of interactions due to self-similarity
at the focusing point.  However, if the interacting waves weaken
enough due to multiple interactions, they will eventually become
passengers.  This in turn breaks the self-similarity and the two
focusing waves will meet after finitely many interactions, so the
solution can be continued.

\begin{figure}[thb]
  \centering
\begin{tikzpicture}[scale=0.7]

\begin{groupplot}[group style={group size=3 by 1}]
\nextgroupplot[
hide x axis,
hide y axis,
tick align=outside,
tick pos=left,
xmin=-0.48, xmax=1.28,
ymin=-0.04987664900096, ymax=1.04740962902016,
]
\addplot [line width=0.32pt, black]
table {%
0 0
0.396 0.99
};
\addplot [line width=0.32pt, black]
table {%
0.8 0
0.404 0.99
};
\addplot [line width=0.32pt, black]
table {%
0.64 0
0.128 0.32
-0.4 0.65
};
\addplot [line width=0.32pt, black]
table {%
0.128 0.32
0.5632 0.592
1.2 0.99
};
\addplot [line width=0.32pt, black]
table {%
0.5632 0.592
0.30208 0.7552
-0.0735999999999999 0.99
};
\addplot [line width=0.32pt, black]
table {%
0.30208 0.7552
0.458752 0.85312
0.67776 0.99
};
\addplot [line width=0.32pt, black]
table {%
0.458752 0.85312
0.3647488 0.911872
0.239744 0.99
};
\addplot [line width=0.32pt, black]
table {%
0.3647488 0.911872
0.42115072 0.9471232
0.4897536 0.99
};
\addplot [line width=0.32pt, black]
table {%
0.42115072 0.9471232
0.387309568 0.96827392
0.35254784 0.99
};
\addplot [line width=0.32pt, black]
table {%
0.387309568 0.96827392
0.4076142592 0.980964352
0.422071296 0.99
};
\addplot [line width=0.32pt, black]
table {%
0.4076142592 0.980964352
0.39543144448 0.9885786112
0.3931572224 0.99
};
\addplot [line width=0.32pt, black]
table {%
0.39543144448 0.9885786112
0.402741133312 0.99314716672
0.39770566656 0.99
};
\addplot [line width=0.32pt, black]
table {%
0.402741133312 0.99314716672
0.3983553200128 0.995888300032
0.407776600064 0.99
};
\addplot [line width=0.32pt, black]
table {%
0.3983553200128 0.995888300032
0.40098680799232 0.9975329800192
0.3889340399616 0.99
};

\nextgroupplot[
hide x axis,
hide y axis,
tick align=outside,
tick pos=left,
xmin=-0.48, xmax=1.28,
ymin=-0.05, ymax=1.05,
]
\addplot [line width=0.32pt, black]
table {%
0 0
0.4 1
};
\addplot [line width=0.32pt, black]
table {%
0.8 0
0.4 1
};
\addplot [line width=0.32pt, black]
table {%
0.64 0
0.128 0.32
-0.4 0.65
};
\addplot [line width=0.32pt, black]
table {%
0.128 0.32
0.5632 0.592
1.2 0.99
};
\addplot [line width=0.32pt, black]
table {%
0.5632 0.592
0.30208 0.7552
-0.0735999999999999 0.99
};
\addplot [line width=0.32pt, black]
table {%
0.30208 0.7552
0.458752 0.85312
};
\addplot [line width=2pt, composite]
table {%
0.458752 0.85312
0.4 1
};
\addplot [very thin, black]
table {%
0.398752 0.77312
0.398752 0.93312
};
\addplot [very thin, black]
table {%
0.398752 0.77312
0.528752 0.77312
};
\addplot [very thin, black]
table {%
0.528752 0.77312
0.528752 0.93312
};
\addplot [very thin, black]
table {%
0.398752 0.93312
0.528752 0.93312
};

\nextgroupplot[
hide x axis,
hide y axis,
tick align=outside,
tick pos=left,
xmin=-0.65, xmax=13.65,
ymin=-0.8, ymax=16.8,
]
\path [draw=direct, fill=direct]
(axis cs:4.75,8.6)
--(axis cs:4,8.35)
--(axis cs:4,8.59975)
--(axis cs:3,8.59975)
--(axis cs:3,8.60025)
--(axis cs:4,8.60025)
--(axis cs:4,8.85)
--cycle;
\path [draw=direct, fill=direct]
(axis cs:6.55,8.6)
--(axis cs:7.3,8.85)
--(axis cs:7.3,8.60025)
--(axis cs:8.3,8.60025)
--(axis cs:8.3,8.59975)
--(axis cs:7.3,8.59975)
--(axis cs:7.3,8.35)
--cycle;
\addplot [very thin, black]
table {%
0 0
0 16
};
\addplot [very thin, black]
table {%
0 0
13 0
};
\addplot [very thin, black]
table {%
13 0
13 16
};
\addplot [very thin, black]
table {%
0 16
13 16
};
\addplot [black, dotted]
table {%
2 6
11 6
};
\addplot [line width=0.32pt, black]
table {%
6.5 6
3.17 4
};
\addplot [line width=0.32pt, black]
table {%
6.5 6
8 2
};
\addplot [line width=2pt, composite]
table {%
6.5 6
4.7 10.8
};
\addplot [line width=0.28pt, black, dashed]
table {%
6.5 6
1.1 9
};
\addplot [line width=0.28pt, black, dashed]
table {%
6.5 6
9.7 9
};
\end{groupplot}

\end{tikzpicture}

  \caption{Avoiding accumulation of trajectories}
  \label{fig:noaccum}
\end{figure}

\subsection{Speed of Composite Waves}

In all cases, the composite wave inherits the width of the base wave.
It remains to choose the speed of the composite wave: to do so, we
check the relative strengths of the carrier wave $\g_b$ and its
passengers $\g_j$, $b_j=b$.  Fix a constant $Q=Q(\e)<1$ to be chosen
later.  We use this $Q$ to distinguish between \emph{``light''} and
\emph{``heavy''} passengers according as whether
\begin{equation}
  \label{lh}
  |\z_j| \le Q(\e)\,|\z_{b_j}|, \com{or}
  Q(\e)\,|\z_{b_j}| < |\z_j|\le |\z_{b_j}|,
\end{equation}
respectively.  Thus a passenger is heavy if its strength is comparable
to that of its carrier, and light otherwise.  We call the composite
wave ``unbalanced'' if all passengers in the wave are light, and
``balanced'' if at least one of the passengers is heavy, and we choose
the composite wave speeds accordingly.  We write $Q(\e)$ rather than
$\ee q$ because we do not require a rate for $Q$: all we'll require is
that $Q(\e) = o(1)$ as $\e\to 0^+$, so it suffices to take say
$Q(\e) = 1/|\log\e|$.

If a composite wave is unbalanced, so that all of its passengers are
light, then we choose the speed of the composite wave to be the least
square speed of the carrier as in Lemma~\ref{lem:spdRes}.  Thus the
base wave is essentially unaffected by the passengers, while the
passengers inherit the speed and width of the base wave.

On the other hand, if the composite wave is balanced, the carrier does
not dominate and the heavy passengers have an effect on the wave
speed.  In this case, we choose the composite wave using least squares
again, but now using all parts of the composite wave.  That is, we
treat the composite wave as a single entity and compute the least
squared speed of the total, namely
\[
  \text{minimize }
  \big(\wh{[f]}-s\,\wh{[q]}\big)^2 :=
    \Big(\sum_{k'}A_{k'}([f]_{k'}-s\,\sum_{k'}[q]_{k'})\Big)^2 .  
\]
In this case we cannot use
Lemma~\ref{lem:spdRes} and we need to recalculate the residual.

For definiteness, suppose that the composite wave is given by
\eqref{comp}.  Then each individual wave satisfies \eqref{fexp},
namely
\begin{equation}
  \label{oluk}
  [f]_{k'} = \la_{k'}(\ol u_{k'})\,[q]_{k'} + O(|\z_{k'}|^3), \qquad
  \ol u_{k'} := \wt W_{k'}\Big(u_{k'-1},\frac{\z_{k'}}2\Big),
\end{equation}
for $k' = k-i,\dots,k+1$, and we have
\[
  [f] = \sum_{k'}[f]_{k'} := f(u_{k+1})- f(u_{k-i-1}) ,\qquad
  [q] = \sum_{k'}[q]_{k'} := q(u_{k+1})- q(u_{k-i-1}).
\]
Since we have a composite wave we need a common reference midpoint,
and so we set
\begin{equation}
  \label{midp}
  \ol u := \wt W_{k+1}\Big(\wt W_k\big(\dots
  \wt W_{k-i}(u_\ell,\frac{\z_{k-i}}2)..,
  \frac{\z_k}2\big),\frac{\z_{k+1}}2\Big),
\end{equation}
which we'll call the \emph{midpoint of the cube}.

\begin{lemma}
  \label{lem:compspd}
    Suppose the composite wave is given by \eqref{comp}, define $\ol u$
  by \eqref{midp}, and suppose that
  $\zeta := \sqrt{\sum\z_{k'}^2} \ll 1$.  Then the least square speed
  $s_*$ is a convex combination
  \[
    s_* = \sum_{k'} \la_{k'}(\ol u)\,a_{k'} + O(\z), \qquad
    \sum_{k'}a_{k'} = 1,
  \]
  and the corresponding residual of the composite wave satisfies
  $R_* = O(\z)$.  If the system is symmetric hyperbolic and the
  preconditioner is chosen to be
  \begin{equation}
    \label{precond}
    A(\ol u) := \sqrt{Dq(\ol u)}^{-\T},
  \end{equation}
  then each coefficient becomes $a_{k'} = \z_{k'}^2/\z^2 \ge 0$.
\end{lemma}

\begin{proof}
  Expanding $[q]_{k'}$ as in Lemma~\ref{lem:spdRes}, we get
  \[
    [q]_{k'} = \z_{k'}\,Dq(\ol u_{k'})\,r_{k'}(\ol u_{k'})
    + O\big(|\z_{k'}|^3\big),
  \]
  where $\ol u_{k'}$ is given in \eqref{oluk}.  Since each wave
  strength $\z_{k'}$ is small and the shock and wave curves
  $\mc S_{k'}$ and $\mc W_{k'}$ agree to quadratic order, this
  expansion holds for both cases.  Further expanding this and
  \eqref{oluk} around the midpoint $\ol u$, we get
  \begin{equation}
    \label{fqk}
    \begin{aligned}
      {}[q]_{k'} &= \z_{k'}\,Dq(\ol u)\,r_{k'}(\ol u) + O(\zeta^2),\\
      [f]_{k'} &= \la_{k'}(\ol u)\,[q]_{k'} + O(\z^2),
    \end{aligned}
  \end{equation}
  where we have lost one order by expanding in other directions, but
  now $Dq$, $\la_{k'}$ and $r_{k'}$ are all evaluated at the common
  point $\ol u$.

  As in \eqref{sR}, the least square solution is given by
  \[
    s_* = \frac{\wh{[q]}\cdot\wh{[f]}}{\wh{[q]}\cdot\wh{[q]}},
  \]
  where now we have defined
  \[
    \begin{aligned}
      \wh{[q]} &= \sum_{k'}A_{k'}[q]_{k'}, \com{and}\\
      \wh{[f]} &= \sum_{k'}A_{k'}[f]_{k'}
            = \sum_{k'}\la_{k'}\,A_{k'}[q]_{k'} + O(\z^2).
    \end{aligned}
  \]
  It follows that
  \[
    \wh{[q]}\cdot\wh{[f]} = \sum_{k'}\la_{k'}\,\ol a_{k'} + O(\z^3), \com{and}
    \wh{[q]}\cdot\wh{[q]} = \sum_{k'}\ol a_{k'} = O(\z^2),
  \]
  where the coefficients $\ol a_{k'}$ are given by
  \[
    \ol a_{k'} := \sum_{j'}\wh{[q]}_{j'}\cdot\wh{[q]}_{k'}
    \qcom{where} \wh{[q]}_{j'} := A_{j'}[q]_{j'},
  \]
  and we have set
  \[
    a_{k'} := \frac{\ol a_{k'}}{\sum_{j'}\ol a_{j'}} \qcom{so}
    \sum_{k'} a_{k'} = 1,
  \]
  and the expression for $s_*$ holds.  It is immediate that
  $R_*=O(\z)$.  We cannot do better than this because the expansion is
  transverse to wave curves and we are left with first order errors.

  If the system is symmetric hyperbolic, and the preconditioner is
  given by \eqref{precond}, then by \eqref{ortho}, the eigenvectors
  $\{r_k\}$ are orthonormal, and by \eqref{fqk}, we have
  \[
    \begin{aligned}
      \wh{[q]}_{j'}\cdot\wh{[q]}_{k'}
      &= \Big(\sqrt{Dq(\ol u_{j'})}^{-\T}[q]_{j'}\Big)\cdot
        \Big(\sqrt{Dq(\ol u_{k'})}^{-\T}[q]_{k'}\Big)\\
      &= \Big(\sqrt{Dq(\ol u)}^{-\T}[q]_{j'}\Big)\cdot
        \Big(\sqrt{Dq(\ol u)}^{-\T}[q]_{k'}\Big) + O(\z^3)\\
      &= [q]_{j'}^\T\,Dq(\ol u)^{-1}[q]_{k'} + O(\z^3)\\
      &= \z_{j'}\,\z_{k'}\,r_{j'}(\ol u)^\T\,
        Dq(\ol u)\,r_{k'}(\ol u) + O(\z^3)\\
      &= \z_{j'}\,\z_{k'}\,\delta_{j'\,k'} + O(\z^3),
    \end{aligned}
  \]
  and the result follows.
\end{proof}

We note that because errors are $O(\z)$, the choice of $\ol u$ can be
anywhere inside the region spanned by the composite wave, say
\[
  \ol u := \wt W_{k+1}\Big(\wt W_k\big(\dots
  \wt W_{k-i}(u_\ell,\theta_{k-i}\,\z_{k-i})..,
  \theta_k\,\z_k\big),\theta_{k+1}\,\z_{k+1}\Big),
\]
where each $\theta_{k'}\in[0,1]$.  

\section{Initialization and Reconstruction}
\label{sec:ir}

By assumption, the initial data $u^o$ is $BV_{loc}$, and we are
working on some compact interval $K = K_{\B R}\subset\B R$.  Our goal in
this section is to provide a sufficiently accurate initial wave
sequence $\Gamma^0$ with which to begin the mFT approximations.
Referring to \eqref{wave}, we must determine initial waves $\g_j$ with
small residual \eqref{RGam}, and which approximates $u^o$ closely.  In
doing so, we must avoid overlap of adjacent simple waves of the same
family; in fact, we will choose the initial wave sequence so that no
waves overlap, but this can be weakened if desired.

Recall that our discretization parameter is $\e>0$.  Since $u^o$ is
$BV$, left and right limits $u^o(x-)$ and $u^o(x+)$ exist for every
$x\in K$, and by \eqref{bvm} we have
\[
  u(x) = \frac{u(x+)+u(x-)}2 \com{for all} x\in K.
\]
The data may have discontinuities, and we explicitly note
finite discontinuities in the data by setting
\[
  X_d := \big\{x\;:\;|u^o(x+)-u^o(x-)|\ge\ee d\big\},
\]
which is a finite set.  Since $u^o\in BV$, between any consecutive
points $\{x_\flat,x_\sharp\}$ of $X_d$, there are finitely many
points, labeled $x_j$, with
\[
  x_\flat =: x_0 < x_1 < \dots < x_{m-1} < x_m := x_\sharp,
\]
and such that
\begin{equation}
  \label{BV0}
  \sum_{x_\sharp,x_\flat}\Big|\big\|u^o\big\|_{BV(x_\flat+,x_\sharp-)} -
  \sum_{j=1}^m\big|u^o(x_j) - u^o(x_{j-1})\big|\,\Big| < \ee i,
\end{equation}
where the outside sum is over all consecutive pairs
$\{x_\flat,x_\sharp\}$ of discontinuities in $X_d$ (together with the
end intervals), and the $x_j$'s are implicitly chosen inside the
interval $(x_\flat+,x_\sharp-)$.  By choosing more points if
necessary, we can also ensure that the corresponding $L^1$ difference
is also small.

We now further discretize the data between points $x_{j-1}$ and $x_j$
by setting an \emph{exact} generalized Riemann problem on each of the
subintervals $(x_{j-1}+,x_j-)$.    We use the following gRP data: for the
states, we use left and right states $u_-:=u^o(x_{j-1}+)$ and
$u_+:=u^o(x_j-)$, respectively, and we need to assign initial widths
$w_k$.  As usual, if the $k$-th family is linearly degenerate, we
assign $w_k=0$.  For the GNL families, we assign widths
\[
  w_k>0, \com{such that} \sum_{k=1}^Nw_k < x_j-x_{j-1}.
\]
We allow maximum flexibility in the widths $w_k>0$ in order to allow
slope information to be incorporated in the approximation, if desired.

We now solve the exact gRP as in Definition~\ref{def:gRP}, and shift
this to the middle of the interval $[x_{j-1},x_j]$: that is, we
translate the gRP solution at the origin to the point
\[
  x_{j}^{(0)} := x_{j-1} +
  \frac12\Big( x_j-x_{j-1}-\sum_{k=1}^Nw_k \Big),
\]
say.  If the slowest or fastest families are GNL, we need not reduce
the width, whereas if they are both linearly degenerate, a reduction
in width ensures that the first interaction of waves from consecutive
subintervals will take place at some positive time $t^!>0$.  In
keeping with our philosophy of maximal flexibility, we allow the
reduction to be chosen as desired.

Having chosen the widths $w_j$ and shift $x_{j}^{(0)}$, we now resolve
the gRP using Assumption \ref{ass:gRP}, and label the corresponding
states as $u_j^{(k)}$ as in \eqref{Wk}, \eqref{gRsol}, so that
\[
  u_j^{(0)} = u^o(x_{j-1}+), \qquad
  u_j^{(k)} \in \mc W_k\big(u_j^{(k-1)}\big), \qquad
  u_j^{(N)} = u^o(x_j-).
\]
We now define points $x_j^{(k)}$ inductively by
$x_j^{(k)}:=x_j^{(k-1)}+w_k$, and use the profile given by
Lemma~\ref{lem:center} to define the piecewise continuous function
$\wh u^o$ on $(x_{j-1}+,x_j-)$ which is the exact data for the gRP, namely
\begin{equation}
  \label{u0def}
  \wh u^o(x) :=
  \begin{cases}
    u_j^{(0)}, & x_{j-1}< x\le x_{j}^{(0)},\\[3pt]
    u_k^o(\z), & x_j^{(k-1)} < x = x_k^o(\z) < x_j^{(k)},\\[3pt]
    u_j^{(N)}, & x_j^{(N)} \le x < x_j,
  \end{cases}
\end{equation}
where $(x_k^o,u_k^o)$ denotes the $k$-simple wave profile given by
Lemma~\ref{lem:center}.  This function is defined for each interval
$(x_{j-1},x_j)\subset(x_\flat,x_\sharp)$, and for all consectuive
pairs $\{x_\flat,x_\sharp\}$ from $X_d$, to give a piecewise
continuous function defined on all of $K$.  Because all differences
are small away from the discontinuity set $X_d$, this interpolation is
a good approximation of the given data $u^o(x)$ on each subinterval
$(x_{j-1},x_j)$ in the $BV$ norm.

Combining all of these subintervals, we have constructed a piecewise
continuous approximation $\wh u^o$ of the initial data $u^o$, which
satisfies both
\begin{equation}
  \label{u0h}
  \sum\big\|u^o - \wh u^o\big\|_{BV(x_\flat+,x_\sharp-)} < \ee i
  \com{and}
  \sum\big\|u^o - \wh u^o\big\|_{L^1(x_\flat+,x_\sharp-)} < \ee i,
\end{equation}
where the sums are over all intervals $(x_\flat+,x_\sharp-)$ defined
by the explicit discontinuities.  Note that the approximation
$\wh u^o$ is implicitly defined, because it requires solution of the
gRP on each subinterval.

At this stage the gRP solution on each subinterval $(x_{j-1}+,x_j-)$
consists of exact simple waves, centered away from the ``origin'', and
contacts; there are no shocks because we chose $w_k>0$ for GNL
families.  Moreover, by \eqref{u0h}, all waves in this gRP solution
are weak, say with strength $|\zeta|<\ee d$.  We now replace each of
these simple waves with a single jump as in our construction of the
dgRS.  The center of each wave is the exact center obtained in the gRP
solution, and the speed $s$ of each wave is chosen by least squares,
as in \eqref{sR} and Lemma~\ref{lem:spdRes}.  Finally, the initial
position of the wave is calculated by \eqref{xp}, namely
\[
  x^p(0) := x^c - s\,t^c,
\]
and the width is chosen no larger than the corresponding $w_k$.  This
ensures that adjacent simple waves do not overlap.  This yields a
\emph{piecewise constant} approximation, denoted $\whh u^o$ of the
initial data, in which each discontinuity is represented as a wave
$\g$ as in \eqref{wave}.
Figure~\ref{fig:init} illustrates the initialization: first, identify
discontinuities, indicated by grey rectangles, and solve the
(discretized) Riemann problem there (bold waves).  Then solve the
refined dgRS between these points, with only one simple wave emerging
from each dgRS posed; here we have narrowed the initial virtual waves
slightly, which may not always be necessary.

\begin{figure}[thb]
  \centering
\begin{tikzpicture}[yscale=0.4,xscale=2]

\begin{axis}[
hide x axis,
hide y axis,
tick align=outside,
tick pos=left,
xmin=-1.223, xmax=16.883,
ymin=-0.002, ymax=0.042,
]
\path [draw=bakShock, fill=bakShock, opacity=0.2]
(axis cs:1.4,0)
--(axis cs:0,0)
--(axis cs:-0.4,0.04)
--(axis cs:0.919999999999999,0.04)
--(axis cs:0.919999999999999,0.04)
--(axis cs:1.4,0)
--cycle;

\path [draw=fwdShock, fill=fwdShock, opacity=0.2]
(axis cs:3.7,0)
--(axis cs:2.3,0)
--(axis cs:2.74,0.04)
--(axis cs:4.1,0.04)
--(axis cs:4.1,0.04)
--(axis cs:3.7,0)
--cycle;

\path [draw=bakShock, fill=bakShock, opacity=0.2]
(axis cs:5.6,0)
--(axis cs:4.2,0)
--(axis cs:3.76,0.04)
--(axis cs:5.2,0.04)
--(axis cs:5.2,0.04)
--(axis cs:5.6,0)
--cycle;

\path [draw=fwdShock, fill=fwdShock, opacity=0.2]
(axis cs:7.6,0)
--(axis cs:6.2,0)
--(axis cs:6.6,0.04)
--(axis cs:7.96,0.04)
--(axis cs:7.96,0.04)
--(axis cs:7.6,0)
--cycle;

\path [draw=fwdShock, fill=fwdShock, opacity=0.2]
(axis cs:8,0)
--(axis cs:8,0)
--(axis cs:8.32,0.04)
--(axis cs:8.48,0.04)
--(axis cs:8.48,0.04)
--(axis cs:8,0)
--cycle;

\path [draw=bakShock, fill=bakShock, opacity=0.2]
(axis cs:9.7,0)
--(axis cs:8.3,0)
--(axis cs:7.94,0.04)
--(axis cs:9.26,0.04)
--(axis cs:9.26,0.04)
--(axis cs:9.7,0)
--cycle;

\path [draw=fwdShock, fill=fwdShock, opacity=0.2]
(axis cs:11.7,0)
--(axis cs:10.3,0)
--(axis cs:10.66,0.04)
--(axis cs:12.1,0.04)
--(axis cs:12.1,0.04)
--(axis cs:11.7,0)
--cycle;

\path [draw=bakShock, fill=bakShock, opacity=0.2]
(axis cs:13.4,0)
--(axis cs:12,0)
--(axis cs:11.64,0.04)
--(axis cs:12.96,0.04)
--(axis cs:12.96,0.04)
--(axis cs:13.4,0)
--cycle;

\path [draw=fwdShock, fill=fwdShock, opacity=0.2]
(axis cs:15.7,0)
--(axis cs:14.3,0)
--(axis cs:14.7,0.04)
--(axis cs:16.06,0.04)
--(axis cs:16.06,0.04)
--(axis cs:15.7,0)
--cycle;

\addplot [bakShock]
table {%
0.7 0
0.260000000000001 0.04
};
\addplot [contact]
table {%
2 0
2 0.04
};
\addplot [fwdShock]
table {%
3 0
3.42 0.04
};
\addplot [thick, bakShock]
table {%
4 0
3.6 0.04
};
\addplot [thick, contact]
table {%
4 0
4 0.04
};
\addplot [thick, fwdShock]
table {%
4 0
4.4 0.04
};
\addplot [bakShock]
table {%
4.9 0
4.48 0.04
};
\addplot [contact]
table {%
6 0
6 0.04
};
\addplot [fwdShock]
table {%
6.9 0
7.28 0.04
};
\addplot [thick, bakShock]
table {%
8 0
7.56 0.04
};
\addplot [thick, contact]
table {%
8 0
8 0.04
};
\addplot [thick, fwdShock]
table {%
8 0
8.32 0.04
};
\addplot [thick, fwdShock]
table {%
8 0
8.48 0.04
};
\addplot [bakShock]
table {%
9 0
8.6 0.04
};
\addplot [contact]
table {%
10 0
10 0.04
};
\addplot [fwdShock]
table {%
11 0
11.38 0.04
};
\addplot [bakShock]
table {%
12.7 0
12.3 0.04
};
\addplot [contact]
table {%
14 0
14 0.04
};
\addplot [fwdShock]
table {%
15 0
15.38 0.04
};
\path [draw=none, fill=grey127]
(axis cs:3.9,0)
--(axis cs:4.1,0)
--(axis cs:4.1,0.002)
--(axis cs:3.9,0.002)
--cycle;
\path [draw=none, fill=grey127]
(axis cs:7.9,0)
--(axis cs:8.1,0)
--(axis cs:8.1,0.002)
--(axis cs:7.9,0.002)
--cycle;
\end{axis}

\end{tikzpicture}
  \caption{Initialization}
  \label{fig:init}
\end{figure}

When solving the exact RP at each $x'\in X_d$, we similarly discretize
any rarefaction as above, with all centers exactly at $(x',0)$, so
that each discretized rarefaction has strength
$0 < {\z'}_k^{(i)} < \ee r$, recalling that all compressive waves in
the exact RP are shocks.

Having discretized each explicit RP on $X_d$ and the gRPs between
point of $X_d$, we form the initial wave sequence $\Gamma^0$.  It
remains to calculate the interaction time of each wave in the
sequence: for each consecutive pair, we use \eqref{xt!} to calculate
$t^!(\g_{j-1},\g_j)$, and for each wave we calculate the
self-interaction time $t^!_s(\g_j)$; if any of these is negative, we
set them to $\infty$.  We then set $t^!_j$ to be the minimum of these
(up to three) interaction times, namely
\[
  t^!_j := \min\big\{t^!(\g_{j-1},\g_j),\;
  t^!(\g_j,\g_{j+1}),\;t^!_s(\g_j)\big\}.
\]

\begin{lemma}
  \label{lem:init}
  The procedure given above generates a wave sequence such that no
  adjacent waves overlap, and such that the total variation of the
  sequence approximates the total variation of the data.  That is,
  given data $u^o\in BV$ and (small) $\e>0$, there are piecewise
  function $\wh u^o$ defined by \eqref{u0def}, and piecewise constant
  function $\whh u^o$, such that
  \[
    \big\|u^o - \wh u^o\big\|_{BV} < C\,\ee i \qcom{and}
    \big\|u^o - \whh u^o\big\|_{M} < C\, \epsilon^{\min\{ e_i,e_d \}} .
  \]
  The piecewise constant function $\whh u^o$ in turn determines a wave
  sequence $\Gamma^{0+}$ such that the variation
  \[
    \mc V(\Gamma^{0+}) = \|\whh u^o\|_{BV}, \com{satisfies}
    \frac1C\,\mc V(\Gamma^{0+})\le \|u^o\|_{BV}\le
    C\,\mc V(\Gamma^{0+}),
  \]
  where the constant $C$ depends only on the compact region
  $K_{\mc U}$, and $\mc V(\Gamma)$ is given by \eqref{Vdef}.
  Moreover, there is a finite, positive time
  \[
    t^!_* := \min_j t^!_j > 0,
  \]
  such that no waves in $\Gamma^0$ interact or collapse in
  the interval $[0,t^!_*)$.
\end{lemma}

\begin{proof}
  Since $\wh u^o$ and $u^o$ have the same discontinuities at each
  point of $X_d$, the first inequality is \eqref{u0h}.  Next, since
  each eigenvector $r_k$ has uniformly bounded length and $\{r_k\}$
  are  uniformly independent on any compact subset of the state space
  $\mc U$, it follows that across any wave $u_+=W_k(u_-,\z)$ or
  $u_+=E_k(u_-,\z)$ given by \eqref{Wk} or \eqref{Ek} respectively,
  there is some $C$ such that
  \[
    \frac1C\,|\z| \le \big|u_+-u_-\big| \le C\,|\z|,
  \]
  and the second inequality follows by adding all waves.  Moreover,
  since $\Gamma^{0+}$ is such that all simple waves at $t=0^+$ are
  weaker than $\ee{d}$, we have
  \[
    \|\wh{u}^o - \whh{u}^o\|_M = \|\wh{u}^o - \whh{u}^o\|_{L^1} \leq
    \wh{C} \ee{d},
  \]
  where $\wh C$ is the length of the compact spatial domain.
  
  Lastly, since there
  are finitely many nonoverlapping waves in $\Gamma^{0+}$, and any
  approaching waves are a finite distance apart at time $t=0$, the
  minimum time for any two waves to meet or a compression to collapse
  is positive, so $t^!_*>0$ by finite speed of propagation.
\end{proof}

Note that for weak waves, we have $C\approx 1$, but for strong shocks
and contacts, the size of $C$ will depend on the definition of wave
strength.  In particular, the wave strength can be any monotonic
parameter along the curves $\mc W_k$ and $\mc E_k$; if we were to use
the variation of $u$ along those curves as the strength parameter,
then we would get $C=1$ and
\[
  \big\|\wh u^o\big\|_{BV} = \mc V(\Gamma^0), \com{if}
  \z := \pm\int_0^\z|u'(\fs\z)|\;d\fs\z.
\]
In particular this holds along $\mc W_k$ if we choose $|r_k|=1$, in
which case $\z$ is given by the signed arclength.

\subsection{Reconstruction}

The mFT approximation is evolved by repeated adjustment of the wave
sequence $\Gamma^t$ via interactions as described above.  However,
given the wave sequence at time $t$, the question arises as to how to
map this wave sequence back into a $BV$ function $u(\cdot,t)$ which is
an approximate solution of the system \eqref{eqsys}.  Typically this
reconstruction occurs only at a discrete set of times, so we regard
$t$ and $\Gamma^t$ as fixed and given.

We begin by locating all discontinuities, which are shocks and
contacts, and are identified by $w_j(t)=0$.  These yield the function
\[
  u^d(x,t) := \sum_{j'}\big(u_{j'} -
  u_{j'-1}\big)\,H\big(x-x^p_{j'}(t)\big),
\]
where $j'$ runs over shocks and contacts, and $H$ is the Heaviside
function.  This accounts for all discontinuities of the reconstructed
solution.

In order to reconstruct the continuous part of the solution, we treat
the simple waves in each GNL family separately.  Consider the $k$-th
family, which is GNL, and extract the subsequence of all those simple
waves for which $k_j=k$.  We will provide a profile for each of these,
but before we do so we must make sure that none of the waves overlap.
If there is some overlap of $k$-waves, including shocks, we reduce the
widths of the overlapping simple waves as in Section~\ref{sec:avoid}
above.  In doing so, we make sure that the reduced widths remain
non-zero: in particular, we would typically reduce widths as little as
possible, so  that $k$-waves may touch but not overlap: that is, after
reduction,
\begin{equation}
  \label{nokol}
  x^{e+}_{j'}(t)\le x^{e-}_j(t),
\end{equation}
where $e\pm$ refer to the right and left edges of the wave, and
$\g_{j'}$ and $\g_j$ are \emph{consecutive} $k$-waves, so $k_i\ne k$
for all $j'<i<j$.

Once the wave widths have been reduced so that \eqref{nokol} holds and
no consecutive $k$-waves overlap, we fill in the profile of the simple
waves as in Lemma~\ref{lem:center} above; according to that lemma,
this profile is unique for each simple wave $\g_j$, and
$u_j=W_k(u_{j-1},\z_j)$.  Moreover, we can parameterize this part of
the simple wave as
\[
  u(x) = W_k\big(u_{j-1},z(x)\big), \qcom{with}
  z\big(x^{e-}_j\big) = 0, \qquad
  z\big(x^{e-}_j\big) = \z_j,
\]
where $z(x)$ determines the profile of the centered wave.  Now define
the ``smoothed Heaviside function'' for the simple wave $\g_j$ to be
\[
  \Phi_j(x) =
  \begin{cases}
    0, & x\le x^{e-}_j,\\
    W_k\big(u_{j-1},z(x)\big) - u_{j-1},
       & x^{e-}_j \le x\le x^{e+}_j,\\
    u_j-u_{j-1}, & x\ge x^{e+}_j,
  \end{cases}
\]
and define the continuous profile function for the simple $k$ waves by
\[
  u^c_k(x,t) := \sum_{\substack{j, k_j=k\\\g_j \text{simple}}}\Phi_j(x).
\]

Having obtained the profile $u^c_k$ for each GNL family, we now
combine these together with the discontinuous part, to obtain the
reconstructed function
\[
  U(x,t) = u_0 + u^d(x,t) + \sum_{k\ \text{GNL}} u^c_k(x,t).
\]
We illustrate this reconstruction in Figure \ref{fig:recon}.  On the
left top panel, we show the original wave sequence.  Below that we
reduce widths to avoid overlap of all waves of the same family, and
below that, we extract the simple waves of the two GNL families.  On
the right, we first show each of the smoothed Heaviside functions
$\Phi_j$, and then then two continuous profiles $u^c_\pm$ and the full
reconstruction featuring all waves, including discontinuities.

\begin{figure}[thb]
  \centering

  \caption{Reconstruction of profile}
  \label{fig:recon}
\end{figure}

\begin{lemma}
  \label{lem:recon}
  There is a constant $C'>0$ such that the reconstruction $U(x,t)$ has
  variation equivalent to that of the wave sequence $\Gamma^t$, that
  is
  \[
    \frac1{C'}\,\big\|U(\cdot,t)\big\|_{BV} \le
    \sum_{j=1}^n\big|\z_j\big| \le
    C'\,\big\|U(\cdot,t)\big\|_{BV}.
  \]
  Moreover, each separate profile of the functions $u^c_k(x,t)$
  is part of an exact gRP solution, so, if waves of other families are
  ignored, can be extended to an interval $(t-\tau,t+\tau)$ for which
  no residual is generated.
\end{lemma}

We note that passengers may be in the wrong place because wavespeeds
have been changed, but these are weak and since $t$ is fixed we cannot
adjust their locations.

In this reconstruction we have used the virtual width of waves to
provide approximate simple waves profiles, which would, locally and in
the absence of waves of other families, be exact solutions of the
system.  That is, by replacing simple jump approximations with
centered profiles, we locally reduce the residual to zero away from
overlapping waves.  However the presence of other families makes this
effect difficult to quantify exactly.

\begin{proof}
  Since each contact and shock appears in $u^d$, and there are no
  simple waves in a linearly degenerate family, we can write
  \[
    \begin{aligned}
      U(x,t) &= u_0 + \sum_{k=1}^Ng_k(x,t), \com{where}\\
      g_k(x,t) &:= u^c_k(x,t) +
                 \sum_{j',\ k_{j'}=k} \big(u_{j'} -
                 u_{j'-1}\big)\,H\big(x-x^p_{j'}(t)\big)\Big),
    \end{aligned}
  \]
  where again $j'$ refers to waves of virtual width $w_{j'} = 0$.
  Here each summand $g_k(x,t)$ corresponds to all waves $\g_j$ in the
  $k$-th family, $k_j=k$.  Since the eigenvectors are uniformly
  independent and bounded on the compact set $K_{\mc U}$, there is a
  single constant $C'$ such that
  \[
    \frac1{C'}\,\sum_{j,\ k_j=k}|\z_j| \le
    \big\|g_k(x,t)\big\|_{BV} \le
    C'\,\sum_{j,\ k_j=k}|\z_j|,
  \]
  for each $k=1,\dots,N$.  Now adding over $k$, the result follows.
\end{proof}

It follows from Lemmas~\ref{lem:init} and \ref{lem:recon}, as well as
the convergence results below, that in order to prove the existence of
$BV_{loc}$ weak* solutions, it is both necessary and sufficient to
obtain a Glimm-like bound for the variation of wave sequences,
\[
  \mc V(\Gamma^t) \le G\big(\mc V(\Gamma^0)\big),
\]
for some nonlinear, nonlocal function $G$.  In our scheme, this
statement is independent of the amplitudes of waves.

\section{Continuation of the Scheme}
\label{sec:continue}

We now show that the mFT scheme can be continued for arbitrary times.
For fixed $\e>0$, we define the scheme inductively as described above.
This yields a sequence
\begin{equation}
  \label{times}
  0<t_1\le t_2\le \dots \le t_i \le \dots,
\end{equation}
of \emph{interaction times} of the mFT approximation.  For each
$t\in[0,T]$, this yields the wave sequence $\Gamma^t$ encoding the
various data as in \eqref{wave}, and which in turn generates the
piecewise constant approximation
\begin{equation}
  \label{Uhat}
  U^\e(x,t) := u_0 + \sum_j[u]_j\,H\big(x-x^p_j(t)\big), \qquad
  [u]_j := u_j-u_{j-1},
\end{equation}
and where $x^p_j(t)$ is the position of the wave at time $t$, given by
\eqref{xp}.  Here it is implicit that the sum is only over waves that
are actually in the sequence $\Gamma^t$ at time $t$.

We first show that the approximation $U^\e$ is in the right space.

\begin{theorem}
  \label{thm:hatU}
  For each $\e>0$ and $i\in\B N$, the wave sequence $\Gamma^{t_i+}$
  has bounded variation, that is
  $\mc V\big(\Gamma^{t_i+}\big)<\infty$.  Moreover, the approximation
  \[
    U^\e \in W^{1,\infty}_*\big(0,t_{i+1};BV_{loc},M_{loc}\big),
  \]
  and for $t<t_{i+1}$, satisfies the bounds
  \begin{equation}
    \label{VMbd}
    \big\|U^\e{}'(\cdot,t)\big\|_{M_{loc}}
    \le \Lambda\,\big\|U^\e(\cdot,t)\big\|_{BV_{loc}}
    \le C\,\mc V(\Gamma^t),
  \end{equation}
  where $\Lambda$ is a bound on the absolute wavespeed, and $C$ is a
  constant depending only on the system.  Moreover $U^\e$ is in
  the space $\mc Z$ given by \eqref{space}, namely
  \[
    U^\e\in \mc Z:= f^{-1}\big(L^\infty_*(0,t_{i+1};BV_{loc})\big) \cap
    q^{-1}\big(W^{1,\infty}_*(0,t_{i+1};BV_{loc},M_{loc})\big).
  \]
\end{theorem}

\begin{proof}
  After initialization there are a finite number of waves, and the
  first interaction time is $t_1>0$.  In particular, the initial
  variation $\mc V(\Gamma^{0+})$ is finite.  Since each interaction
  yields a finite number of waves of finite strength, it follows by
  induction that after any finite number of interactions, both the
  number of waves and variation remain finite.

  For any $t\in (t_i,t_{i+1})$, it is clear that
  \[
    \big\|U^\e(\cdot,t)\big\|_{BV_{loc}} \le
    \sum_j\big|[u]_j\big|,
  \]
  where we have used the triangle inequality in the case of composite
  waves.   Differentiating in time, we get
  \[
    U^\e{}'(\cdot,t) = \sum_j [u]_j\,s_j\,\delta_{x_p(t)},
  \]
  and since different Dirac masses have independent supports and
  $\|\delta\|=1$, we get
  \[
    \big\|U^\e{}'(\cdot,t)\big\|_{M_{loc}}
    \le \Lambda\,\sum_{x\in X^t_d}\Big|\sum_{j,\;x_j^p=x}[u]_j\Big|
    \le \Lambda\,\big\|U^\e(\cdot,t)\big\|_{BV_{loc}},
  \]
  where $X^t_d$ denotes the (finite) set of jumps at $t$, and each
  $s_j$ satisfies $|s_j|\le\Lambda$.  Now, treating composite waves as
  sums of individual waves $\g_j$, by \eqref{gRsol}, we have
  \[
    u_j = E_{k_j}(u_{j-1},\z_j) \com{or}
    u_j = W_{k_j}(u_{j-1},\z_j), \com{so}
    \big|[u]_j\big| \le C\,|\z_j|,
  \]
  for some constant $C$ determined by the compact set in which we
  solve the dgRS.  Now adding gives
  \[
    \sum_j\big|[u]_j\big| \le C\,\sum_j|\z_j| = C\,\mc V(\Gamma^t),
  \]
  which yields the bound \eqref{VMbd}.

  Next, we observe that, because states are treated exactly, and
  interactions in the mFT take place only at a single point, at which
  incident waves meet and/or reflected waves emerge, the function
  $U^\e$ is continuous in $M_{loc}$ at each point of interaction
  $t_i$ (noting that all sums are finite because there are only
  finitely many waves).  It follows that
  $U^\e:[0,t_{i+1}]\to M_{loc}$ is (strongly) Lipschitz with constant
  $\max_{i'\le i}\mc V(\Gamma^{t_{i'+}})$, so that
   \[
     U^\e \in W^{1,\infty}_*\big(0,t_{i+1};BV_{loc},M_{loc}\big)=
     W^{1,\infty}\big(0,t_{i+1};BV_{loc},M_{loc}\big).
  \]
  Finally, since both $f$ and $g$ are $C^1$ on $\mc U$, we can carry
  out the same argument for
  \[
    \begin{aligned}
      f\circ U^\e &= f(u_0) + \sum_j[f]_j\,H\big(x-x^p_j(t)\big)
      \com{and}\\
      q\circ U^\e &= q(u_0) + \sum_j[q]_j\,H\big(x-x^p_j(t)\big),      
    \end{aligned}
  \]
  which implies that $U^\e\in\mc Z$ given by \eqref{space}.
\end{proof}

We now define
\begin{equation}
  \label{tinf}
  t_\infty = t_\infty(\e) := \sup_i\{t_i\}.
\end{equation}
If $t_\infty(\e) = \infty$, then the scheme can be defined for all
times; if not, we say that there is an \emph{accumulation of
  interaction times} at $t_\infty<\infty$.

We claim that if interactions accumulate at $t_\infty<\infty$, then
some sort of blowup must be occurring, such as total variation blowup,
$\mc V(\Gamma^t)\to\infty$ as $t\to t_\infty^-$, or failure of the
\emph{``$\alpha$-ray condition''} which we introduce below.  Verifying
this condition for general systems will require a careful analysis of
wave interactions and self-similar wave patterns.  In Section
\ref{sec:psys} below, we show that in the important prototypical case
of the $p$-system modeling isentropic gas dynamics, there is no
accumulation of interaction times, so the scheme can indeed be defined
for all $t>0$.

\subsection{Trajectories}
\label{sec:traj}

Denoting the set of \emph{all} interaction times as in \eqref{times},
we define the set of \emph{trajectories} as follows.  Each trajectory
is a \emph{time-like} set of waves $\{\rho_m\}$, one from each wave
sequence $\Gamma^{t_m+}$, which generates a continuous path.  We write
\begin{equation}
  \label{traj}
  \rho_m = \g_{j_m}\in \Gamma^{t_m+}, \com{each} m,
\end{equation}
where $j_m$ is the index of $\rho_m$ in $\Gamma^{t_m+}$, so that using
\eqref{wave}, $\rho_m$ has family $k_{j_m}$, strength $\z_{j_m}$, and
speed $s_{j_m}$, etc., for times $t\in(t_m,t_{m+1})$.  Continuity
requires that the position given by \eqref{xp} satisfy
\[
  x^p_{j_m}(t_m+) = x^p_{j_{m-1}}(t_m-) \com{each} m,
\]
so that the path of each trajectory is continuous and piecewise
linear, so Lipschitz.  Note that if the $m$-th interaction takes place
away from this trajectory, we will have
\[
  k_{j_m} = k_{j_{m-1}}, \quad
  \z_{j_m} = \z_{j_{m-1}}, \com{and}
  s_{j_m} = s_{j_{m-1}},
\]
as well as other common properties, although the index will generally
change, $j_m\ne j_{m-1}$.  The interaction at $(x_m,t_m)$ is
\emph{nontrivial} if either of the states adjacent to $\rho_m$ changes
at $t_m$,
\[
  u_{j_m-1} \ne u_{j_{m-1}-1} \com{or}
  u_{j_m} \ne u_{j_{m-1}},
\]
which also implies that at least some other properties change.

If there is a finite accumulation of interaction times,
\[
  t_* := t_\infty = \lim_i t_i <\infty,
\]
then, since trajectories are uniformly Lipschitz and contained in a
compact set by finite speed of propagation, there must be at least one
limit point $x_* := \lim x_{j_m}$ for some subsequence $j_m$.  In
particular, there must be at least one \emph{accumulating trajectory},
which is a trajectory $\{\rho_m\}$ such that
\[
  x_m := x^p_{j_m}(t_m) \to x_* \com{as} m\to\infty,
\]
with infinitely many non-trivial interactions along the trajectory.
Our goal is to rule out the existence of such accumulating
trajectories, so that the mFT scheme can be defined for all time, for
finite discretization parameter $\e>0$.

Fix a point $(x_*,t_*)$ and speed $\al$.  Consider the trace of the
approximation on the ray with speed $\al$ and ending at $(x_*,t_*)$,
that is
\[
  V_{\al}(t) :=
  U^\e\big(x_*+\al\,(t-t_*),t\big), \qquad
  t\in [0,t_*),
\]
where we have used \eqref{bvm} for $U^\e\in BV$.  Because
$U^\e$ is piecewise constant, we can write $V_{\al}$ as a sum of
Heavisides,
\begin{equation}
  \label{Va}
  \begin{aligned}
    V_{\al}(t) &:= \sum_j[v]_j\,H(t-t_j), \quad 0\le t<t_*,
                  \com{where}\\
    [v]_j &:= U^\e\big(x_*+\al\,(t_j-t_*),t_j+\big)
            -U^\e\big(x_*+\al\,(t_j-t_*),t_j-\big),
  \end{aligned}
\end{equation}
where now there could be infinitely many jumps, and we take $H(0)=1/2$.

\begin{defn}
  \label{def:aray}
  The approximation $U^\e$ satisfies the \emph{$\al$-ray
    condition} at $(x_*,t_*)$, if the limit
  \[
    V_\al := \lim_{t\to t_*^-}V_{\al}(t)
  \]
  exists, uniformly for $\al$ in compact intervals.
\end{defn}

Note that this condition implicitly assumes that the scheme has been
defined up to time $t_*-$.

\begin{theorem}
  \label{thm:noaccum}
  For fixed $\e>0$, suppose that the mFT approximation $U^\e$
  satisfies the $\al$-ray condition for all $(x,t)$, with
  $t\le t_\infty$.  Then there are no accumulating trajectories and
  the scheme can be continued for all times, that is
  $t_\infty=\infty$.
\end{theorem}

\begin{proof}
  Suppose $\{\rho_m\}$ is an accumulating trajectory, and let the
  accumulation point be
  \[
    (x_\infty,t_\infty) := \lim(x_m,t_m),
  \]
  which exists because the trajectory is Lipschitz.  Denote the
  minimum and maximum possible wavespeeds by
  $\ul\la := \min_{\mc U}\la_1(u)$ and
  $\ol\la := \max_{\mc U}\la_N(u)$, respectively.  By the $\al$-ray
  condition, there exists some $K\in\B N$ such that for all $k\ge K$,
  and all $\al\in[\ul\la,\ol\la]$, we have
  \[
    \big|U^\e\big(x_\infty+\al\,(t_k-t_\infty),t_k\big) - V_\al\big|
    < \frac12\,\ee p,
  \]
  where $\ee p$ is the maximum strength of any passenger.  Now
  consider the space-time triangle
  \begin{equation}
    \label{Ddel}
    D_\Delta := \Big\{ (x,t) \;:\; t_K\le t\le t_\infty, \
    x_\infty - \ol\la t\le x\le x_\infty - \ul\la t\Big\}.
  \end{equation}
  It follows by the triangle inequality that all waves in this
  triangle are \emph{either} weaker than $\ee p$, and so by
  construction  are lonely weak waves, \emph{or} lie on a
  ray running into the accumulation point $(x_\infty,t_\infty)$.

  Waves that lie on rays into $(x_\infty,t_\infty)$ cannot interact
  before $t=t_\infty$, so will not interact inside $D_\Delta$.  It
  follows that any interaction inside $D_\Delta$ must involve at least
  one lonely weak wave.  However, any interaction that results in a
  lonely weak wave cannot increase the number of waves.  Thus
  because there are only finitely many waves at time $t_K$, and in
  particular entering $D_\Delta$, there can be only finitely many
  interactions inside $D_\Delta$, contradicting our assumption that
  $(x_\infty,t_\infty)$ is an accumulation point.
\end{proof}

Together with Theorem \ref{thm:hatU}, this theorem implies that the mFT
approximations can be defined for all time, and moreover the variation
of these approximations remains finite for all time.  However, we have
no control over the size of the variation as $\e\to0$, and indeed,
this can become infinite in the limit in the known examples of
solutions which blow up.  This is analogous to the use of Euler's
method in approximating the ODE
\[
  \begin{aligned}
    y' &= y^2, \qquad y(0) = 1, \com{namely}\\
    y_{k+1} &= y_k + h\,y_k^2, \qquad y_0 = 1,
  \end{aligned}
\]
for which the approximation $y_k$ is defined and finite for any $h>0$,
even though the ODE solution exists only for $t<1$.

As an immediate corollary, we have the following
\emph{characterization of blowup} of mFT approximations for general
systems:

\begin{cor}
  \label{cor:blowup}
  If the total variation of the mFT approximation blows up in finite time,
  \[
    \big\|U^\e(\cdot,t)\big\|_{BV}\to\infty \com{as}
    t\to t_\infty,
  \]
  then the $\al$-ray condition fails at some point $(x_\infty,t_\infty)$.
\end{cor}

This corollary indicates that $BV$ blowup generally occurs only in
cases similar to known blowup results, such as the focussing blowup
due to infinitely many interactions as depicted in Figure
\ref{fig:accum}, \cite{Jens,JY}, or uniform amplitude blowup in
which multiple separated interactions drive the solution out of any
compact set~\cite{Yblowup1,szY,Jens,BJ1}.

\section{Convergence of the Scheme}
\label{sec:conv}

In defining the scheme for a general system, we have thus far
introduced several exponents $e_\square$ related to the discretization
parameter $\e$.  We now impose conditions relating these which will
allow us to establish convergence of the scheme, assuming that the
total variation $\mc V(\Gamma^t)$ is uniformly bounded for $t\le T$.

The exponents and their corresponding small parameters we have
introduced are:
\begin{itemize}
\item Rarefaction $e_r$: the maximum rarefaction strength is $\ee r$;
\item Compression $e_c$: the maximum (absolute) compression strength
  is $\ee c$;
\item Weak $e_w$: the minimum (absolute) compression strength is
  $\ee w$;
\item Simple $e_s:=\min\{e_r,e_c\}$: maximum strength of a simple wave
  is $\ee s$;
\item Passenger $e_p$: the maximum passenger strength is $\ee p$;
\item Jump $e_d$: the minimum discrete jump in initialization is $\ee d$;
\item Initial $e_i$: the maximum initial $BV$ error is $\ee i$.
\end{itemize}
We have also yet to choose the parameter $Q(\e)$ which distinguishes
light passengers from heavy ones, and the parameter $\kappa<1/3$ which
gives the relative minimum strength of a split rarefaction.  There are
several constraints on these indices which are required for
consistency of the construction.  For example, since we do not want
multi-rarefactions to be passengers, we must have
\[
  \ee p < \kappa\,\ee r, \com{so that} e_r<e_p.
\]

\subsection{Residual Calculations}

Recall that upon initialization, all waves are either shocks or
contacts, compressions or rarefactions, and initially, there are no
multi-rarefactions and no composite waves.  Also recall that the
residual $\mc R(\Gamma)$ of a wave sequence is the measure given by
\eqref{RGam}.  We will examine convergence of the mFT via this
measure: that is, the scheme converges provided $\mc R(\Gamma)\to 0$
as $\e\to0$.

\begin{lemma}
  \label{lem:initRes}
  If the initial wave sequence $\Gamma^{0+}$ has variation
  $\mc V(\Gamma^{0+}) =: V^0$, then up to the first
  interaction time, the residual satisfies
  \[
    \big\|\mc R(\Gamma^t)\big\|_M \le O(1)\,V^0\,\e^{2e_s}, \qquad 0<t<t^!_*,
  \]
  where $e_s:=\min\{e_r,e_c\}$, $t^!_*$ is given in
  Lemma~\ref{lem:init}, and $O(1)$ depends only on the system.
\end{lemma}

\begin{proof}
  Each initial wave is given the least squared wave speed, and
  Lemma~\ref{lem:spdRes} applies.  Moreover, the residual vanishes for
  shocks and contacts having the correct (least square) speed.  It
  follows that these contribute nothing to the residual, whereas each
  simple wave of strength $\z$ contributes $O(|\z|^3)$.  Since all
  simple waves have strength $|\zeta|<\ee s$, we have
  \[
    \big\|\mc R\big(\Gamma^{0+}\big)\big\|_M \le \sum_j \big|R(\z_j)\big|
    \le O(1)\,\sum_j|\z_j|^3 \le O(1)\,\e^{2e_s}\,\sum_j|\z_j|
    \le O(1)\,V^0\,\e^{2e_s},
  \]
  where the sum is over initial simple waves and we have used
  \eqref{Res*s} of Lemma~\ref{lem:spdRes}.
\end{proof}

Our goal is to bound the residual after interactions in the same way:
try to get a bound of order $\zeta^{1+\ell}$ for each individual wave,
and add.  However, because we have used composite waves and
multi-rarefactions, more care must be taken.  In particular, in a
composite wave with a heavy passenger, we cannot get a uniform
superlinear bound for the residual.  We must thus make an assumption
about the number of such heavy passengers.  Recall that the passenger
$\g_h$ with carrier $\g_{b_h}$ is heavy if the ratio
$|\z_h/\z_{b_h}| > Q$, see \eqref{lh}.

We are assuming uniformity in $t$, so it suffices to consider the
residual of a single wave sequence $\Gamma$.  Recall from \eqref{wave}
that $\Gamma$ is indexed by $j=1,\dots,n$, and partition this set of
indices into
\[
  B_L \cup B_H \cup P,
\]
where $P$ is the set of passengers, $B_H$ the set of bases or carriers
$\g_b$ with at least one heavy passenger $\g_h$, and $B_L$ those of
bases with no heavy passengers.  To each $j\in P$ is associated a
carrier $b_j\in B_L\cup B_H$, while for each $j\in B_L\cup B_H$, we
have $b_j=j$.

Recall that the quantity $Q(\e)$ is used to distinguish light
passengers from heavy in \eqref{lh}.  If a passenger $\g_h\in P$ with
carrier $\g_b\in B_H$ is heavy, then their strengths satisfy
\[
  Q(\e)\,|\z_b| < |\z_h| \le \min\big\{|\z_b|,\,\ee p\big\}.
\]
Define the \emph{``weight of heavy passengers''} by
\begin{equation}
  \label{hp}
  W_H(\Gamma,\e) :=
  \sum_{b\in B_H}\sqrt{\sum_{b_h=b}\z_h^2},
\end{equation}
where $h$ runs over heavy passengers only, and interpret this as
follows.  If this sum is large, it indicates that there are many
interactions in which (weak) waves of different families, and hence
different wavespeeds, but comparable strength, do not separate.  This
indicates a strong nonlinear connection between different familes and
is likely a result of nonlinear resonances such as those described
in~\cite{Ysurv, Ysus, Yex}.  We expect that such resonances could be
understood in the context of WNGO approximations~\cite{MRS, HMR,
  JMRii}.  In general, we expect systems with several GNL families and
compactly supported data to separate after some transient
interactions, due to their distinct wave speeds~\cite{John, TPLdec},
in which case \eqref{hp} should be much smaller after the families
separate.  Similarly, we expect that when shocks are present, they are
usually strong, and will therefore have light passengers only, so not
contribute to this term.  In any event, these issues need to be
addressed in order to obtain $BV$ bounds which are necessary for the
proof of convergence, and are likely highly system dependent.

\begin{lemma}
  \label{lem:RG}
  For fixed $t$, the residual of the wave sequence $\Gamma=\Gamma^t$
  satisfies the estimate
  \begin{equation}
    \label{residG}
    \big\|\mc R(\Gamma)\big\|_M \le \mc V(\Gamma^t)\,
    \Big(K_s\,\e^{2e_s}
    + K_p\,Q(\e)\Big) + K_{hp}\, W_H(\Gamma,\e),
  \end{equation}
  where $K_\square$ are constants depending only on the system via a
  compact region of state space $\mc U$.
\end{lemma}

\begin{proof}
  From \eqref{RGam}, the residual of the wave sequence can be written
  as
  \[
    \mc R(\Gamma) = \sum_{b\in B_L\cup B_H}
    \sum_{b_j=b}A_j
    \big([f]_j-s_b\,[q]_j\big)\,\delta_{x^p_b},
  \]
  where $A_j$ is the preconditioner, since each passenger inherits the
  speed of the carrier.  Since the Dirac mass has $\|\delta\|_M=1$,
  this yields
  \begin{equation}
    \label{RGS}
    \big\|\mc R(\Gamma)\big\|_M \le
    \sum_{b\in B_L} \Big|\sum_{b_l=b}
    A_l\big([f]_l-s_b\,[q]_l\big)\Big| +
    \sum_{b\in B_H} \Big|\sum_{b_h=b}
    A_h\big([f]_h-s_b\,[q]_h\big)\Big|,
  \end{equation}
  and we treat the two sums separately.  Here we have separated
  only composite waves, which are still treated as combined groupings.

  The first term in \eqref{RGS} handles all bases carrying light
  passengers.  In this case, the carrier itself has a small residual,
  and the passengers can be controlled by the carrier's strength.  For
  any $b\in B_L$, we have
  \[
    \begin{aligned}
      \Big|\sum_{b_l=b}
      A_l\big([f]_l-s_b\,[q]_l\big)\Big|
      &\le \big|A_b([f]_b-s_b\,[q]_b)\big| +
        \sum_{\substack{l\in P\\b_l=b}}\big|A_l([f]_l-s_b\,[q]_l)\big|\\
      &\le R(\g_b) + (N-1)\,\Big|\big(\la_l(\ol u_l)-s_b\big)\,
        A_l[q]_l + O(|\z_l|^3)\Big|,
    \end{aligned}
  \]
  where $R(\g_b)$ is the residual of the carrier and we have used
  \eqref{fexp}.  The sum is estimated because each carrier has at most
  one passenger from each family.  Since these passengers are light,
  we use \eqref{lh} to write
  \[
    (N-1)\,\Big|\big(\la_l(\ol u_l)-s_b\big)\,
    A_l[q]_l + O(|\z_l|^3)\Big| \le K_{lp}\,|\z_l| \le K_{lp}\,|\z_b|\,Q(\e),
  \]
  where $K_{lp}$ depends only on the system and state space $\mc U$.

  Since $\g_b$ carries only weak passengers, its residual is
  calculated with either the least squared or multi-rarefaction wave
  speed.  If the wave is a shock or contact, its residual is zero,
  while if it is simple, then by \eqref{Res*s} its residual is
  \[
    R(\g_b) = O(|\z_b|^3) \le K_b\,|\z_b|\,\e^{2e_s}.
  \]
  On the other hand, if $\g_b$ is part of a multi-rarefaction
  $\wh\g_b$, then we treat all parts of the same multi-rarefaction
  together, and use Lemma~\ref{lem:mr} to describe its residual.
  According to \eqref{mrRes}, the (combined) multi-rarefaction
  residual is estimated by
  \[
    \big\|R(\wh\g_b)\big\|
    \le \sum_{i}\z^{(i)}\,K_{mr}\,\ol m^2\,\e^{2e_r}
    =: \wh\z_b\,K_{mr}\,\ol m^2\,\e^{2e_r},
  \]
  where each multi-rarefaction consists of \emph{at most} $2\,\ol m+1$
  pieces.  Since each simple wave incident in an interaction has
  strength at most $\ee s$, and the total variation is assumed
  bounded, the maximum size of any newly generated rarefaction is
  $C\,\ee s$, for some constant $C$ depending only on the variation
  bound and compact region of state space.  On the other hand, any
  rarefaction is split into $2\,m+1$ multi-rarefactions of size at
  least $\kappa\,\ee r$, which yields the estimate
  \[
    2\,\ol m+1 \le \frac{C\,\ee s}{\kappa\,\ee r}, \qcom{so that}
    \ol m^2\,\e^{2e_r} \le \frac{C^2}{4\,\kappa^2}\,\e^{2e_s}.
  \]
  Combining these, we estimate the first sum in \eqref{RGS} as
  \begin{equation}
    \label{RGS1}
    \sum_{b\in B_L} \Big|\sum_{b_l=b}
    A_l\big([f]_l-s_b\,[q]_l\big)\Big| \le
    \sum_{b\in B_L} |\z_b|\,
    \big(K_s\,\e^{2e_s} + K_{lp}\,Q(\e)\Big),
  \end{equation}
  where $K_s:=\max\big\{K_b, C^2/4\kappa^2\big\}$.

  We now estimate the second sum in \eqref{RGS}, in which each base
  $\g_b$ carries at least one heavy passenger $\g_h$.  For such a
  composite wave $\wh \g_b$, we use Lemma~\ref{lem:compspd} to write
  \[
    R_*(\wh\g_b) \le K_{hp}\,\sqrt{\sum_{b_j=b}\z_j^2}
    \le K_{hp}\,\sqrt{\sum_{b_l=b}\z_l^2} +
    K_{hp}\,\sqrt{\sum_{b_h=b}\z_h^2},
  \]
  where we have split the sum up into light ($l$) and heavy
  ($h$) passengers.  For the light passengers, we estimate the first
  sum by
  \[
    \sqrt{\sum_{b_l=b}\z_l^2} \le |\z_b|\,\sqrt{N-2}\,Q(\e),
  \]
  while the term for the heavy passengers is left as is.
  Combining these with \eqref{RGS1} in \eqref{RGS} yields
  \[
    \begin{aligned}
      \big\|\mc R(\Gamma)\big\|_M
      &\le \sum_{b\in B_L} |\z_b|\,
        \big(K_s\,\e^{2e_s}
        + K_{lp}\,Q(\e)\Big)\\
      &{}\qquad{} + K_{hp}\,\sum_{b\in B_H}\Big(|\z_b|\,\sqrt{N-2}\,Q(\e)
        + \sqrt{\sum_{b_h=b}\z_h^2}\Big)\\
      &\le \sum_b|\z_b|\,\big(K_s\,\e^{2e_s}
        + K_p\,Q(\e)\Big) + K_{hp}\,\sum_{b\in B_H}\sqrt{\sum_{b_h=b}\z_h^2}\,,
    \end{aligned}
  \]
  where $K_p:=\max\big\{K_{lp},\,\sqrt{N-1}\,K_{hp}\big\}$, and this
  immediately implies \eqref{residG}.
\end{proof}

\subsection{Convergence}

We have defined the mFT scheme by essentially tracking and modifying
the wave sequences $\Gamma^t$.  Lemma~\ref{lem:RG} provides a bound
for the residual of $\Gamma^t$, uniform in $t$.  Our goal is to obtain
existence of a solution via the mFT scheme, as follows.  First use the
wave sequences to generate the corresponding piecewise constant
approximations $U^\e\in \mc Z$ given in \eqref{Uhat} and
Theorem~\ref{thm:hatU}.  Next, we use compactness to extract a
convergent subsequence and limit $\wh u$, and finally, we use the
residual calculation above to show that the limit is a solution.  In
order to apply compactness, we \emph{assume} that the variation is
uniformly bounded.

\begin{theorem}
  \label{thm:conv}
  Suppose the mFT scheme is defined for $t\le T$ with the uniform
  variation bound
  \[
    \mc V(\Gamma^t) < V, \com{for} 0\le t\le T.
  \]
  Suppose also that $Q(\e)$ can be chosen so that
  \begin{equation}
    \label{convAss}
    Q(\e) = o(1)  \qcom{and}
    W_H(\Gamma,\e) = o(1) \qcom{as}
    \e\to 0^+,      
  \end{equation}
  uniformly for $t\in[0,T]$.  Then some subsequence of the
  corresponding approximations $U^\e$ converges as $\e\to0^+$, and the
  limit is a weak* solution of the system \eqref{eqsys}.
\end{theorem}

In particular, referring to \eqref{hp}, and if $Q(\e) = O(\ee q)$,
then \eqref{convAss} holds if say
\[
  \#\big\{B_H\big\}\,\e^{e_p-e_q} = o(1), \com{as} \e\to 0^+.
\]

\begin{proof}
  By Theorem \ref{thm:noaccum}, the wave sequence $\Gamma^t$ is
  defined for all $t\in[0,T]$, and by Theorem \ref{thm:hatU}, the
  corresponding approximations $U^\e$ live in the space $\mc Z$
  given by \eqref{space}.  Since $\mc V(\Gamma)$ is bounded uniformly
  in $t$ and $\e$, by \eqref{VMbd}, so are both $U^\e$ and $\wh
  U^\e{}'$,
  \[
    \big\|U^\e(\cdot,t)\big\|_{BV_{loc}}
    \le \frac{C}{\Lambda}\,\mc V(\Gamma), \com{and}
    \big\|U^\e{}'(\cdot,t)\big\|_{M_{loc}}
    \le C\,\mc V(\Gamma^t),
  \]
  or in other words
  \[
    \begin{aligned}
      U^\e &\in L^\infty_*(0,T;BV_{loc})
                 = \big(L^1(0,T;C_0^{-1})\big)^*,
    \com{and}\\
      U^\e{}' &\in L^\infty_*(0,T;M_{loc})
                    = \big(L^1(0,T;C_0^0)\big)^*,
    \end{aligned}
  \]
  where we have used \eqref{dualSp}.

  It follows by the Banach-Alaoglu theorem that there is some sequence
  $\{\epsilon_i\}_{i=1}^\infty$, with $\e_i\to0^+$ monotonically, and
  limit $U$ such that
  \[
    U^{\e_i} \to U \in L^\infty_*(0,T;BV_{loc})
  \]
  as $i\to\infty$, where the convergence is in the dual topology of
  the space $C_0^{-1}$ of test functions.  Since $q$ and
  $f:\mc U\to\B R^N$ are $C^2$, it follows that
  \[
    \Big\|f\big(U^{\e_i}(\cdot,t)\big)
    - f\big(U(\cdot,t)\big)\Big\|_{BV}
    \le \sup_{u\in\mc U}|Df(u)|\,
    \big\|U^{\e_i}(\cdot,t)-U(\cdot,t)\big\|_{BV}
    \to 0,
  \]
  uniformly in $t$, and similarly for $q$, and so also
  \[
    \begin{aligned}
    f\circ U^{\e_i} &\to
    f\circ U \in L^\infty_*(0,T;BV_{loc}), \com{and}\\
    q\circ U^{\e_i} &\to
    q\circ U \in L^\infty_*(0,T;BV_{loc}).
    \end{aligned}
  \]

  Next, for each $\e>0$, and for each $t\in[0,T]$, the function
  $q\circ U^\e(\cdot,t)$ is piecewise constant and $BV$, so its
  (Gelfand) weak derivative $(q\circ U^\e)'$ is a measure of bounded
  variation,
  \[
    (q\circ U^\e)' \in L^\infty_*(0,T;M_{loc}), \com{with}
    \big\|(q\circ U^\e)'(\cdot,t)\big\|_{M_{loc}}
    \le \big\|q\circ U^\e(\cdot,t)\big\|_{BV_{loc}},
  \]
  uniformly for $t\in [0,T]$.  Again using \eqref{dualSp} and applying
  Banach-Alaoglu, there is some subsequence, again denoted $\e_i$, and
  time varying measure $\mu:[0,T]\to M_{loc}$ such that
  \[
    (q\circ U^{\e_i})' \to \mu \in L^\infty_*(0,T;M_{loc}),
  \]
  as $i\to\infty$.  Now for space and time test function
  $\phi=\phi(x)\in C^0(K)$ and $\psi=\psi(t)\in C^1([0,T])$,
  respectively, we have
  \[
    \begin{aligned}
      \int_0^T\big\langle\mu(t),\phi\big\rangle\,\psi(t)\;dt
      &= \int_0^T\lim_i\big\langle(q\circ U^{\e_i})'(t),
        \phi\big\rangle\,\psi(t)\;dt\\
      &= \lim_i\int_0^T\big\langle(q\circ U^{\e_i})'(t),
        \phi\big\rangle\,\psi(t)\;dt\\
      &= - \lim_i\int_0^T\big\langle q\circ U^{\e_i}(t),
        \phi\big\rangle\,\psi'(t)\;dt\\
      &= - \int_0^T\lim_i\big\langle q\circ U^{\e_i}(t),
        \phi\big\rangle\,\psi'(t)\;dt\\
       &= - \int_0^T\big\langle q\circ U(t),
        \phi\big\rangle\,\psi'(t)\;dt,
    \end{aligned}
  \]
  where we have used dominated convergence and the Gelfand weak
  derivative.  This implies
  \[
    \mu' = q\circ U \in L^\infty_*(0,T;M_{loc}),
  \]
  and so in particular, by \eqref{space}
  \[
    U^\e \to U \in f^{-1}\big(L^\infty_*(0,T;BV_{loc})\big) \cap
    q^{-1}\big(W^{1,\infty}_*(0,T;BV_{loc},M_{loc})\big) = \mc Z.
  \]

  Thus the limit $U$ of mFT approximations lives in the space $\mc Z$,
  and it remains to show that $U$ is a weak* solution.
  Lemma~\ref{lem:RG} provides the estimate \eqref{residG}, so by our
  assumptions \eqref{convAss} and continuity of the residual
  \eqref{resid} we have
  \[
    \mc R(U)(t) = \lim_i\mc R(\Gamma^t) = 0,
  \]
  so $U$ is a weak* solution and the proof is complete.  Note that the
  use of preconditioners $A$ does not affect this convergence, because
  we assumed that $A$ is uniformly bounded with uniformly bounded
  inverse.
\end{proof}

Here we have chosen to use Banach-Alaoglu rather than Helly's
Selection Principle because of its wider range of applicability: in
particular, the proof applies unchanged to the alternative space
$\mc Z^{-b,\infty}$ defined in \eqref{Hbspace}, provided we choose $b$
large enough that the Dirac mass $\delta\in H^{-b}$, that is for
$b>1/2$.

\begin{cor}
  \label{cor:main}
  Under the same assumptions as Theorem~\ref{thm:conv}, and for any
  $b>1/2$, a subsequence of the mFT approximations converges to an
  $H^{-b}$ weak* solution of system \eqref{eqsys}.
\end{cor}

\subsection{General Approximation of $BV$ Solutions}

We now argue that, given a well-behaved $BV$ solution $u(x,t)$ of
system \eqref{eqsys}, we expect the approximation $U^\e(x,t)$
generated by the mFT scheme, to accurately resemble the actual
solution $u(x,t)$, for appropriately small values of $\e>0$.  Since we
are considering large amplitude, large $BV$ solutions of a general
system, for which existence and regularity theorems are not yet
available, there is little hope of proving anything rigorous in this
context, so our comments are necessarily speculative.

So suppose that we are given a solution $u(x,t)$ of
\eqref{eqsys}, such that $\|u(\cdot,t)\|_{BV}<\infty$ uniformly for
$t\in[0,T]$, and which takes on initial data $u^o(x)$ and appropriate
boundary values.  For given $\e$, we use this initial data and the
same boundary conditions to generate the mFT approximation
$U^\e(x,t)$.  Assuming the solution is unique, we then ask the
question: \emph{under what conditions will the mFT approximation
  \pmb{not} converge to the actual solution $u(x,t)$?}

Recall that in order to keep the number of waves finite, the mFT
approximation uses carrier and passenger waves, as described in
Section~\ref{sec:comp}.  In doing so, by construction, we are assured
that $U^\e$ satisfies the $\al$-ray condition for all $t\le T$.  The
assumption of uniform $BV$ bounds together with
Corollary~\ref{cor:blowup} implies that interaction times do not
accumulate for $t\in[0,T]$.  As a result, we expect an analog of the
$\al$-ray condition to continue to hold in the limit as $\e\to 0$.  We
thus introduce the following condition that we expect solutions which
are limits of mFT approximations to satisfy.

\begin{defn}
  \label{def:salray}
  We say that a function $u\in L^\infty(0,T;BV_{loc})$ satisfies the
  \emph{strong $\al$-ray condition}, if for almost every $\al\in\B R$
  and $(x_*,t_*)$, the limit
  \[
    \lim_{t\nearrow t_*}\big[u\big]\big(x_*-\al\,(t-t_*),t\big)
    \quad\text{exists},
  \]
  where $[u](x,t) := u(x+,t)-u(x-,t)$ is the jump in $u$, which exists
  for a.e.~$(x,t)$ because $u$ is uniformly $BV$.
\end{defn}

This implicitly assumes that the $\al$-ray condition applies to $U^\e$
uniformly in $\e$.  In particular, this condition holds if (the
restriction of) $u$ is $BV$ on time-like line segments.

We now assume we are given the known $BV$ solution $u(x,t)$ of
\eqref{eqsys}, and that the approximation $U^\e(x,t)$ has been
generated as described above.  We further assume that appropriate
entropy conditions hold for both the solution and approximation, and
that the corresponding entropy conditions are enough to guarantee
\emph{uniqueness} and \emph{stability} of solutions, in some norm.

Now suppose that the mFT approximations $U^\e$, which are assumed to
be uniformly in $BV$, \emph{fail to converge} to the actual solution
$u$, as $\e\to 0$.  We wish to understand what conditions prevent such
convergence.  We argue as follows.  By Helly's selection theorem or
Banach-Alaoglu, we can extract a subsequence $U^{\e_i}$ which
converges to a $BV$ limit
\[
  \wh U = \lim_{i\to\infty} U^{\e_i} \in L^\infty(0,T;BV_{loc}).
\]
By construction, for each $\e>0$, the approximation $U^\e$ satisfies
the (strong) $\al$-ray condition, and we assume that the given
solution $u$ also satisfies it.  This assumption effectively excludes
wave patterns similar to those of Figure~\ref{fig:accum}, in which the
accumulating interacting waves do not vanish in the limit, in which
case some form of blowup should occur.

According to Theorem~\ref{thm:conv}, because we are assuming \emph{a
  priori} $BV$ bounds, the only way in which the limit $\wh U$ is
\emph{not} a solution, is if \eqref{convAss} fails at some minimal
positive time $t=t_\sharp\in(0,T]$.  This means that given any
function $Q(\e)\to 0$, say including $Q(\e)=1/\log|\log\e|$, there
exists some $\delta = \delta(Q)$, and a subsequence $\e_{i_k}\to 0$,
such that
\[
  W_H(\Gamma^{t_\sharp},\e_{i_k}) \ge \delta, \qquad\text{uniformly in }k.
\]
Here the wave sequence $\Gamma^{t_\sharp}$ implicitly depends on
discretization parameter $\e_{i_k}$.

Using \eqref{hp}, we can estimate the number of heavy bases, as
\[
  \delta \le W_H(\Gamma^{t_\sharp},\e_{i_k})
  \le \#\{B_H\}\,\sqrt{N-1}\,\ee p, \qcom{so that}
  \#\{B_H\} \ge \frac \delta{\sqrt{N-1}}\,\e^{-e_p},
\]
since there is at most one passenger from each family distinct from
that of the base.  Moreover, taking the uniform bound for the total
variation to be $V_0$, we must also have
\[
  \#\{B_H\} \le \frac{V_0}{Q(\e)\,\ee p}, \qcom{since}
  |\z_b| \ge Q(\e)\,\ee p
\]
for each heavy base $\g_b$.  We therefore have both upper and lower
bounds for $\#\{B_H\}$ of almost the same order, namely
\[
  \frac{\delta\big(Q(\e)\big)}{\sqrt{N-1}}\,\e^{-e_p}
  \le \#\{B_H\} \le \frac{V_0}{Q(\e)}\,\e^{-e_p},
\]
and the function $Q(\e)\to0$ as slowly as we like.  At the same time,
each heavy base $\g_b$ has a corresponding heavy passenger, say
$\g_j$, for which
\begin{equation}
  \label{wvsz}
  Q(\e)\,|\z_b| < |\z_j| \le \min\big\{|\z_b|,\ee p\big\},
\end{equation}
so the waves have comparable strength, while all waves between $\g_j$
and $\g_b$ also have strength at most $|\z_b|$.

By the Pigeonhole Principle, there must be (at least) two distinct
families, say $b_1$ and $b_2$, and $O(\e^{-e_p})$ interactions, in
which the emerging wave strengths $\z_{b_i}$ are comparable and of
size \eqref{wvsz}, while all emerging waves between them have size at
most $\ee p$.  In this sense, the wave families $b_1$ and $b_2$ fail
to separate, in that the corresponding wave strengths are strongly
correlated rather than independent, and we say that they are
\emph{``strongly coupled by nonlinear interactions''}.  This is in
contrast to local small amplitude interactions in hyperbolic systems,
in which different families in compactly supported solutions generally
separate after a short time and propagate largely independently
afterwards, as described in ~\cite{John} and \cite{TPLdec}.

We can express this failure of separation by referring to the
classical Glimm interaction estimates~\cite{G,Yth}.  Recall that these
give accurate descriptions of the effects of nonlinear interactions of
small amplitude waves.  We can describe the estimates in the current
context as follows: suppose that (composite) waves $\al$ and $\beta$
span states $u_\ell$, $u_m$ and $u_r$, respectively.  Using the
notation of \eqref{urul}, we write
\[
  u_m - u_\ell = \sum \al_i\,\ol r^\al_i, \qcom{and}
  u_r - u_m = \sum \beta_i\,\ol r^\beta_i,
\]
in which the composite wave $\al$ has components of strength $\al_i$
in each family, and the superscripts indicate that the $\ol r^\al_i$,
as given by \eqref{ukzr}, vary with different waves.  In resolving the
interaction, we similarly write
\[
  u_r - u_\ell = \sum \z_i\,\ol r^\z_i,
\]
and resolve the $\z_i$'s in terms of the $\al_i$'s and $\beta_i$'s.
Glimm's estimate states that if the incident waves are weak, we have
\[
  \z_i = \al_i + \beta_i +
  \sum_{j>k}\al_j\,\beta_k\,\Lambda^{jk}_i + O_3,
\]
in which the $\Lambda$'s are second derivative Lie bracket
\emph{interaction coefficients}, and $O_3$ is cubic in incident wave
strength, see \cite{G, Yth, Ysurv}.

It follows that if a heavy passenger and carrier emerge from this
interaction, $|\z_{b_1}|\approx|\z_{b_2}|$, then either these were
present in the incident waves, or they arise as nonlinear effects of
the interaction of other waves of moderate strength.  If they were
previously present then $\#\{B_H\}$ does not increase across the
interaction.  On the other hand, there are only fintiely many waves of
moderate strength, so these should not interact very often.  In either
case, we do not expect $\#\{B_H\}$ to grow too large for generic
systems.  In any event, a deeper analysis of the interaction structure
will be essential in deriving $BV$ bounds, and accurate estimates for
$\#\{B_H\}$ will likely emerge from any such derivation.

\begin{conj}
  Suppose that the hyperbolic system \eqref{eqsys} has a unique $BV$
  entropy weak* solution $u$, which is stable in an appropriate
  topology.  Further, suppose that $u$ satisfies the strong $\al$-ray
  condition \ref{def:salray}, and that a ``weak separation condition''
  holds for the system.  Then the mFT approximation $U^\e$ generated
  with the same initial data will converge to the actual solution $u$
  of the nonlinear system.
\end{conj}

Note that we have not specified the ``weak separation condition''
beyond the heuristic given above.  Our conclusion is necessarily
``soft'', because little is known analytically about solutions to
nonlinear hyperbolic problems, particularly those of large amplitude.
Indeed, obtaining existence, uniqueness and/or stability in that
regime would constitute a major advance in the field.

\section{Application to Gas Dynamics}
\label{sec:gd}

We briefly recall the Euler equations for compressible gas dynamics.
We work primarily in a Lagrangian or material frame, in which the
equations and wavespeeds are simplified.  The equations are
\begin{equation}
  \label{euler}
  (-v)_t + u_x = 0, \qquad
  u_t + p_x = 0, \qquad
  \big(\textstyle{\frac12}\,u^2 + e\big)_t + (u\,p)_x = 0,
\end{equation}
expressing continuity, Newton's law, and conservation of energy,
respectively.  The equations are closed by an equation of state, say
$e=e(p,s)$, satisfying the Second Law of Thermodynamics.  Here $u$ is
the velocity, and the thermodynamics variables are specific volume
$v$, pressure $p$, specific internal energy $e$, specific enthalpy
$h$, specific entropy $s$, and temperature $\theta$.  We use $p$ and
$s$ as the primary thermodynamic variables, so that the equation of
state is given by the enthalpy $h=h(p,s)$, and the Second Law is
\[
  dh = v\,dp + \theta\,ds, \qquad h := e + p\,v,
\]
so also
\[
  v = \frac{\dd h}{\dd p} \qcom{and}
  \theta = \frac{\dd h}{\dd s}.
\]

The system is of the form \eqref{eqsys}, if we denote
\begin{equation}
  \label{gd}
  w := \(p\\u\\s\),\qquad
  q(w) := \(-v\\u\\\textstyle{\frac12}\,u^2 + h - p\,v\), \qcom{and}
  f(w) := \(u\\p\\u\,p\),
\end{equation}
and being a physical system it is symmetrizable; see below.

\subsection{Generalized Riemann Problem and mFT Scheme}

For smooth solutions the system can be written in quasilinear form,
\[
  Dq(w)\,w_t + Df(w)\,w_x = 0,
\]
where
\begin{equation}
  \label{DqDf}
  Dq(w) = \( -\frac{\dd v}{\dd p} & 0 & -\frac{\dd v}{\dd s}\\
  0 & 1 & 0 \\ -p\,\frac{\dd v}{\dd p} & u &
  \theta - p\,\frac{\dd v}{\dd s} \), \qcom{and}
  Df(w) = \( 0 & 1 & 0 \\ 1 & 0 & 0 \\ u & p & 0 \),
\end{equation}
in which we have used $e = h - p\,v$.  We can now easily find the
generalized eigenvalues \eqref{ev}: after two elementary row
operations, this is equivalent to
\[
  \( 0 & 1 & 0 \\ 1 & 0 & 0 \\ 0 & 0 & 0\)\,r_k(w)
  = \la_k(w)\,\( \frac1{c^2} & 0 & -\frac{\dd v}{\dd s}\\
  0 & 1 & 0 \\ 0 & 0 & \theta \)\,r_k(w),
\]
where we have defined the \emph{acoustic impedance} $c=c(p,s)$ by
\begin{equation}
  \label{imp}
  c^2 = -\frac{\dd p}{\dd v} = -1\Big/\frac{\dd v}{\dd p}.
\end{equation}
We can now read off the generalized eigenvalues and vectors, namely
\begin{equation}
  \label{gdev}
  \begin{aligned}
  \la_-(w):=- c(p,s),&\qquad
  \la_0(w):=0, \quad& \la_+(w):= c(p,s),\\
    r_-(w):=\(-1\\1/c\\0\),&\qquad r_0(w):=\(0\\0\\1\),
    \quad& r_+(w):=\(1\\1/c\\0\).
  \end{aligned}
\end{equation}
Thus, as is well known, the system is strictly hyperbolic as long as
$c(p,s)>0$, and the forward and backward waves are GNL while
$\frac{dc}{dp}>0$.  The eigenvectors are easily integrated to get the
simple wave curves: the forward and backward wave curves are
\begin{equation}
  \label{gdsimp}
  u_r - u_l = \pm\int_{p_l}^{p_r}\frac1{c(\fs p,s)}\;d\fs p, \qquad
  s_r = s_l,
\end{equation}
respectively, and the (stationary) entropy jumps are
\[
  p_r = p_l, \qquad u_r = u_l.
\]

To describe the shocks and contacts, we solve the jump conditions
\eqref{sc}, which are
\[
  [u] = \sigma\,[-v], \qquad [p] = \sigma\,[u], \qquad
  [u\,p] = \sigma\,\big[\textstyle{\frac12}\,u^2 + e\big].
\]
Eliminating $[u]$ in the first two equations gives
\begin{equation}
  \label{gdshock}
  \sigma = \pm\sqrt{\frac{[p]}{[-v]}} \qcom{and}
  [u] = \pm\sqrt{[p]\,[-v]},
\end{equation}
while the third equation becomes
\[
  \ol u\,[p] + [u]\,\ol p = \sigma\,\big(\ol u\,[u] + [e]\big),
  \qcom{or} \sigma\,\big([e] + \ol p\,[v] \big) = 0,
\]
where $\ol p:=\frac{p_l+p_r}2$, and we regard this as an equation for
the entropy jump across the shock.  If $\sigma=0$, then also
$[p]=[u]=0$, which is again an entropy jump, while if $\sigma\ne 0$
we rewrite this as
\[
  \begin{aligned}
    0 &= [e] + \ol p\,[v] = [h] - \ol v\,[p], \qcom{or}\\
    h(p_r,s_r)&-h(p_l,s_l) = \frac{v(p_r,s_r)+v(p_l,s_l)}2\,(p_r-p_l),
  \end{aligned}
\]
which is an equation for $s_r$, say, and coupled with \eqref{gdshock}
is the equation defining the shock curve.

The global existence and uniqueness of solutions to the Riemann
problem is well known; extension to the generalized Riemann problem
follows readily.

\begin{lemma}
  \label{lem:strgRP}
  The generalized Riemann problem for the Euler equations
  \eqref{euler} is globally uniquely solvable, although possibly with
  a vacuum.  A vacuum occurs if and only if the condition
  \begin{equation}
    \label{vac}
    u_r - u_l \ge \int_0^{p_l} \frac1{c(\fs p,s_l)}\;d\fs p
     + \int_0^{p_r} \frac1{c(\fs p,s_r)}\;d\fs p
  \end{equation}
  holds.
\end{lemma}

\begin{proof}
  The proof is a direct application of Lemma~\ref{lem:gRPsol}.  We
  wish to connect states \[ \{w_l=:w_0,w_1,w_2,w_3:=w_r\}, \]
  respectively by (arbitrarily large) backward, stationary and and
  forward waves, and we must check that the differences
  $[u]_k:=u_k-u_{k-1}$ form an independent set.

  If we interpret the integral in \eqref{gdsimp} as an average by setting
  \[
    \Big\langle\frac1c\Big\rangle_{lr} := \frac1{p_r-p_l}
    \int_{p_l}^{p_r}\frac1{c(\fs p,s)}\;d\fs p
    = \int_{-1/2}^{1/2}\frac1{c(\fs p,s)}\;d\fs \eta>0, \qquad
    \fs p := \ol p + \fs\eta\,(p_r-p_l),
  \]
  then both forward and backward simple wave curves can be written as
  \begin{equation}
    \label{gdsp}
    u_r - u_l = \Big\langle\frac1c\Big\rangle_{ab}\,(p_a - p_b),
    \qquad s_r - s_l = 0,
  \end{equation}
  where subscripts $a$ and $b$ denote the \emph{ahead} and
  \emph{behind} states, respectively.

  Similarly, recalling that the entropy condition is that the pressure
  be higher behind a shock, we write the shock curve \eqref{gdshock}
  as
  \begin{equation}
    \label{gdsh}
    u_r - u_l = \frac1{|\sigma|}\,(p_a-p_b), \qquad
    |\sigma| = \sqrt{\frac{[p]}{[-v]}}, \qquad p_b>p_a,
  \end{equation}
  in which the entropy is regarded as implicitly solved.

  Combining the waves together in order, we can then write
  \begin{equation}
    \label{wvc}
    w_1 - w_0 = (p_0-p_1)\,\ol r_-, \qquad
    w_2 - w_1 = (s_2-s_1)\,\ol r_0, \qquad
    w_3 - w_3 = (p_3-p_2)\,\ol r_+,
  \end{equation}
  where we have set
  \begin{equation}
    \label{rgd}
    \ol r_- := \(-1\\1/{\{c\}}\\{[s]}/{[p]}\), \qquad
    \ol r_0 := \(0\\0\\1\), \qquad
    \ol r_+ := \(1\\1/{\{c\}}\\{[s]}/{[p]}\),
  \end{equation}
  and where
  \[
    \frac1{\{c\}} := \Big\langle\frac1c\Big\rangle \com{or}
    \{c\} := |\sigma|,
  \]
  for simple waves or shocks, respectively, and $[s]=0$ for simple waves.

  Since these vectors are clearly independent, by
  Lemma~\ref{lem:gRPsol}, the gRP has a unique solution whenever $w_r$
  can be reached from $w_l$.  The pressure $p$ has the range $p>0$,
  with boundary $p=0$ describing the vacuum, at which the system loses
  hyperbolicity, $c(0)=0$.  Since any value of $s$ can be reached by a
  contact, we need to check the possible range of velocity differences
  $u_r-u_l$, so we consider \eqref{gdshock} and \eqref{gdsimp}.  For a
  shock wave, we have $p_b>p_a$, and we have
  \[
    u_r-u_l = \pm \sqrt{(p_b-p_a)\,(v_a-v_b)},
  \]
  which is unbounded if either $p_b\to\infty$ or $p_a\to 0$, because
  this implies $v_a\to\infty$.  Thus the shock curves are unbounded.

  For simple waves, if the wave is compressive, $p_b>p_a$,
  \eqref{gdsp} is again unbounded as $p_b\to\infty$.  However, for a
  rarefaction, $p_a>p_b$ and we must consider the limit $p_b\to 0$.
  If the integral \eqref{gdsimp} is unbounded, then again the
  rarefaction curve is unbounded and the gRP can always be solved with
  the available waves.  However, if the integral \eqref{gdsimp} is
  \emph{bounded}, then so is the velocity variation $u_r-u_l$.  In
  this case the range of allowable velocities is precisely those for
  which \eqref{vac} fails.

  If the vacuum condition \eqref{vac} holds, then we cannot solve the
  generalized Riemann problem with these waves because of the vacuum.
  However, we can still solve the gRP as a weak* solution, by
  introducing the vacuum as a new type of wave which is manifest as a
  Dirac mass in the specific volume $v$; see~\cite{MY2,YpRP2,LS} for details.
\end{proof}

Having solved the gRP, globally away from the vacuum, it follows that
the mFT scheme can be defined for gas dynamics.  However, because
there are only three wave families, the middle one of which is
linearly degenerate, wave patterns such as those in
Figure~\ref{fig:accum} cannot be drawn, and so we define the mFT
scheme \emph{without the use of composite waves}.  In this case we
define the mFT as in Section~\ref{sec:mft}, following the adjustments
in the same order but not joining any waves into composites.
According to Theorem~\ref{thm:conv}, the mFT approximations will
converge to a weak* solution as long as they are defined, uniformly in
discretization parameter $\e$, provided the total variation
$\mc V(\Gamma)$ remains uniformly bounded.

Thus, \emph{assuming} a uniform time independent variation bound, we
will obtain convergence of the scheme to a weak* solution, and thus
global existence, provided we show that there is no finite time
accumulation of interactions.  Below we show that in the $2\times2$ system
of isentropic gas dynamics, no such accumulation can occur while the
variation is uniformly bounded.  We expect that similarly there is no
such accumulation for the full $3\times3$ system, because there are no
points of self-similarity due to linear degeneracy of the (middle)
stationary contact field.

\subsection{Equivalence of Material and Spatial Frames}

System \eqref{euler}, \eqref{gd} are written in a
material or Lagrangian frame, and it is known that under weak
conditions, this is equivalent to the system written in a spatial, or
Eulerian frame.  If $y$ denotes the spatial variable, then the
equations are
\begin{equation}
  \label{euy}
  \begin{aligned}
    \rho_t &+ (\rho\,u)_y = 0, \\
    (\rho\,u)_t &+ \big(\rho\,u^2+p\big)_y = 0, \\
    \big(\textstyle{\frac12}\,\rho\,u^2 + \rho\,e\big)_t
    &+\big(\textstyle{\frac12}\,\rho\,u^3 + \rho\,u\,e + u\,p\big)_y = 0,
  \end{aligned}
\end{equation}
representing conservation of mass, momentum and energy, respectively.
Here $\rho:=1/v$ is the density of the gas, and all other variables
are the same, but regarded as functions of $(y,t)$ rather than
$(x,t)$.

The transformation $y\mapsto x$ from spatial to material coordinates
is given by
\begin{equation}
  \label{ytox}
  x := \int^y \rho(\fs y,t)\;d\fs y, \qcom{so that}
  \frac{\dd x}{\dd y} = \rho(y,t).
\end{equation}
This in turn can be inverted as
\begin{equation}
  \label{xtoy}
  \frac{\dd y}{\dd x} = \frac1{\rho(y,t)} = v(x,t), \qcom{so also}
  y = \int^x v(\fs x,t)\;d\fs x.
\end{equation}
This gives a bi-Lipschitz map as long as $v$ and $\rho$ both remain
integrable, and this is enough to show equivalence of weak solutions
of \eqref{euler} and \eqref{euy}~\cite{Wag}.  Moreover, this can be
extended up to the vacuum $\rho=0$ if we consistently allow $v$ to be
a Dirac mass and we consider weak* solutions, see~\cite{MY2}.

It follows immediately that solutions of the Riemann problem are
equivalent in the Lagrangian and Eulerian frames.  In particular, the
profiles of any centered rarefaction are consistent for positive
times.  Therefore the profiles of centered compressions must also be
consistent because compressions are backward rarefactions.  Thus the
gRP solutions are also consistent across the frames, and the dgRS and
mFT schemes must also be consistent.

Because the Lagrangian to Eulerian transformation is nonlinear, this
consistency of solutions must be reflected in the calculation of
wavespeeds in the Eulerian frame.  Here we carry out the various
calculations and check that the nonlinear wavespeeds in the two frames
are indeed consistent.

\begin{lemma}
  \label{lem:LE}
  The simple and shock wavespeeds in Eulerian and Lagrangian
  frames are related by
  \[
    \la^E_i(w) = u + v\,\la^L_i(w),\qquad
    \sigma^E_i(w) = \ol u + \ol v\,\sigma^L_i(w),
  \]
  respectively, while the states $w$ across the waves are identical,
  and these are consistently related by the maps \eqref{xtoy} and
  \eqref{ytox}.
\end{lemma}

\begin{proof}
  We use the state variable $w$ defined in \eqref{gd}, so that
  \eqref{euy} has the form \eqref{eqsys} for
  \begin{equation}
    \label{gde}
    q^E(w) := \(\rho\\\rho\,u\\\rho\,(\frac12\,u^2 + h) - p\),
    \qquad
    f^E(w) := \(\rho\,u\\\rho\,u^2 + p\\
    \rho\,u\,(\frac12\,u^2 + h)\),
  \end{equation}
  where we have used
  \[
    \rho\,e = \rho\,(h-p\,v) = \rho\,h - p, \qquad h = h(p,s).
  \]
  Differentiating, we have
  \[
    (\rho\,h)_p = \rho_p\,h + \rho\,h_p = h\,\rho_p + 1, \qquad
    (\rho\,h)_s = \rho_s\,h + \rho\,h_s = h\,\rho_s + \rho\,\theta,
  \]
  and we calculate
  \[
    Dq^E(w) = \(\rho_p & 0 & \rho_s\\[3pt]
    \rho_p\,u&\rho&\rho_s\,u\\[3pt]
    \rho_p\,(\frac12\,u^2 + h)&
    \rho\,u& \rho_s\,(\frac12\,u^2 + h) + \rho\,\theta \),
  \]
  and
  \[
    Df^E(w) = \(\rho_p\,u&\rho&\rho_s\,u\\[3pt]
    \rho_p\,u^2+1&2\,\rho\,u&\rho_s\,u^2\\[3pt]
    \rho_p\,u\,(\frac12\,u^2 + h) + u&
    \frac32\,\rho\,u^2 + \rho\,h&
    \rho_s\,u\,(\frac12\,u^2 + h) + \rho\,\theta\,u\).
  \]
  After applying elementary row operations and simplifying, these can
  be written as
  \[
    Dq^E(w) = \(1&0&0\\u&1&0\\\frac12\,u^2+h&u&1\)
    \(\rho_p&0&\rho_s\\0&\rho&0\\0&0&\rho\,\theta\),
  \]
  and
  \[
    Df^E(w) = u\,Dq^E(w) +
    \(1&0&0\\u&1&0\\\frac12\,u^2+h&u&1\)
    \(0 & \rho & 0\\1 & 0 & 0\\ 0 & 0 & 0\),
  \]
  so that the generalized eigenvalue problem \eqref{ev} becomes
  \[
    \(\rho_p & 0 & \rho_s \\ 0 & \rho & 0 \\
    0 & 0 & \rho\,\theta \)^{-1}\,
    \(0 & \rho & 0\\1 & 0 & 0\\ 0 & 0 & 0\)\,r^E =
    \(0 & \rho/\rho_p & 0 \\ 1/\rho & 0 & 0 \\ 0 & 0 & 0\)\,r^E =
    \big(\la^E-u\big)\,r^E.
  \]
  We now use \eqref{imp} to write
  \[
    \rho_p = \Big(\frac1v\Big)_p = \frac{-1}{v^2}\,v_p
    = \frac{1}{v^2\,c^2} = \frac{\rho^2}{c^2} =: \frac1{a^2},
  \]
  where we have defined the \emph{speed of sound} by
  \begin{equation}
    \label{sound}
    a(p,s) := \frac1{\sqrt{\rho_p(p,s)}}
    = \sqrt{\frac{\dd p}{\dd\rho}} = \frac c\rho = c\,v.
  \end{equation}
  It now follows easily that the eigensystem in Eulerian coordinates
  is
  \begin{equation}
    \label{gdevE}
    \begin{aligned}
      \la^E_-(w) := u - a(p,s),&\qquad
      \la^E_0(w) := u, \quad& \la^E_+(w) := u + a(p,s),\\
      r_-(w):=\(-1\\1/(\rho\,a)\\0\),&\qquad r_0(w):=\(0\\0\\1\),
      \quad& r_+(w):=\(1\\1/(\rho\,a)\\0\).
    \end{aligned}
  \end{equation}
  Comparison with \eqref{gdev} and use of \eqref{sound} shows that
  \begin{equation}
    \label{simpcomp}
    \la^E_i(w) = u + v\,\la_i^L(w), \qcom{and}
    r^E_i(w) = r_i^L(w),
  \end{equation}
  for each $i=-,0,+$ and each $w\in \mc U$, where the superscripts
  refer to the associated frame.  Equality of the
  eigenvectors is expected because eigenvectors determine
  \emph{changes of state} across the wave, while the different
  wavespeeds reflect the Lagrangian to Eulerian map.

  We similarly compute the shock curves in an Eulerian frame.
  The jump conditions \eqref{sc} are
  \[
    [\rho\,u] = \sigma^E\,[\rho], \qquad
    [\rho\,u^2+p] = \sigma^E\,[\rho\,u], \qquad
    \big[\rho\,u\,\big(\textstyle{\frac12}\,u^2+h\big)\big] = 
    \sigma^E\,[\rho\,\big(\textstyle{\frac12}\,u^2+h\big) - p\big],
  \]
  Following \cite{CF}, these simplify if we set
  \begin{equation}
    \label{bmv}
    b := u - \sigma^E, \qquad
    m := \rho\,b, \qcom{so that}
    \rho\,u = m +\sigma^E\,\rho, \qquad
    b = v\,m,
  \end{equation}
  since $v=1/\rho$.  In this notation, the jump conditions become,
  after simplifying,
  \[
    [m] = 0, \qquad
    [m\,u + p] =0, \qquad
    \big[m(\textstyle{\frac12}\,u^2+h)\big] + \sigma^E\,[p] = 0,
  \]
  which represent conservation of mass, momentum and energy across the
  shock, respectively, and by \eqref{bmv} we also have the continuity
  equation
  \[
    [u] = [b] = [m\,v] = m\,[v].
  \]
  We now write the energy conservation as
  \[
    \begin{aligned}
      m\,\ol u\,[u] + m\,[h] + \sigma^E\,[p]
      &= m\,[h] - (\ol u - \sigma^E)\,[p]\\
      &= m\,[h] - \ol b\,[p]
      = m\,\big([h] + \ol v\,[p]\big) = 0.
    \end{aligned}
  \]
  Comparing these with \eqref{gdshock}, we conclude that the jump
  conditions give equivalent conditions for the state across the jump,
  namely
  \[
    [u]^2 = [p]\,[-v] \qcom{and}
    [h] = \ol v\,[p],
  \]
  while we must have $-m = \sigma =: \sigma^L$, the Lagrangian wave
  speed, which in turn yields the relation
  \[
    \sigma^E = \ol u - m\,\ol v = \ol u + \ol v\,\sigma^L.
  \]

  Now suppose that a trajectory in material coordinates is described
  by the equation
  \[
    \frac{dx}{dt} = a^L.
  \]
  Then by \eqref{xtoy} and the chain rule, for a curve on which $w$ is
  continuous, we have
  \[
    \begin{aligned}
      \frac{dy}{dt}
      &= \int^x\frac{\dd v}{\dd t}(\fs x,t)\;d\fs x
        + v(x,t)\,\frac{dx}{dt}\\
      &= u(x,t) + v(x,t)\,a^L =: a^E.
    \end{aligned}
  \]
  On the other hand, if $w$ has a jump along the trajectory, we
  \emph{define} the Eulerian speed of the trajectory by
  \begin{equation}
    \label{LEs}
    a^E := \ol u + \ol v\,a^L,
  \end{equation}
  where $\ol\square$ is the average of left and right states.  This is
  consistent with the Heaviside function having value $\frac12$ at the
  origin.  It then follows that while $w$ remains $BV$, so that left
  and right limits exist, the relation \eqref{LEs} holds everywhere.
\end{proof}

As a corollary, we obtain equivalence of the mFT schemes defined in
either Lagrangian or Eulerian coordinates.

\begin{lemma}
  \label{lem:mFTequiv}
  The mFT schemes defined using either the Lagrangian or Eulerian
  frames are equivalent.  That is, the maps \eqref{xtoy} and
  \eqref{ytox} respect both the mFT scheme and the limits thereof.
\end{lemma}

Although we prove this lemma only for gas dynamics, evidently the same
result will hold for any more complicated system for which both
Lagrangian and Eulerian frames of reference can be consistently
defined.

\begin{proof}
  We check that the residuals of the two schemes are consistent,
  although they are not identical.  According to \eqref{pwcres}, it
  suffices to show that the residuals of a single jump in either frame
  are (uniformly) related.  We thus suppose that we have a jump in the
  state $w$, propagating with speeds $a^L$ and $a^E$ in the Lagrangian
  and Eulerian frames, respectively.  Following~\cite{Wag}, it is
  convenient to introduce a new (degenerate) equation for the identity
  $\rho\,v=1$ into the system; we do so by extending \eqref{gd} and
  \eqref{gde} by defining
  \[
    \wh q^L(w):=\(1\\-v\\u\\\frac12\,u^2+h-p\,v\), \qcom{and}
    \wh f^L(w):=\(0\\u\\p\\u\,p\),
  \]
  and by setting
  \begin{equation}
    \label{qfLE}
    \wh q^E(w) := \rho\,\wh q^L(w), \qcom{and}
    \wh f^E(w) := \wh f^L(w) + u\,\wh q^E(w).
  \end{equation}
  It follows that the systems \eqref{euler} and \eqref{euy} are
  equivalent to
  \[
    \wh q(w)_t + \wh f(w)_z = 0, \qcom{for}
    \big(\wh q,\wh f,z\big) := \big(\wh q^L,\wh f^L,x\big) \com{and}
    \big(\wh q,\wh f,z\big) := \big(\wh q^E,\wh f^E,y\big),
  \]
  respectively, where the first, respectively second, equation is
  trivial.

  Now consider the corresponding (vector) residuals
  \begin{equation}
    \label{vres}
    \wh R^L := \big[\wh f^L\big] - a^L\,\big[\wh q^L\big] \qcom{and}
    \wh R^E := \big[\wh f^E\big] - a^E\,\big[\wh q^E\big],
  \end{equation}
  which are equivalent to their $3\times3$ counterparts.  We write
  \[
    \wh q^L = v\,\wh q^E \qcom{and}
    \wh f^L = \wh f^E - u\,\wh q^E,
  \]
  and calculate
  \[
    \begin{aligned}
      \wh R^L - \wh R^E
      &= -\big[u\,\wh q^E\big] - a^L\,\big[v\,\wh q^E\big]
        + a^E\,\big[\wh q^E\big]\\
      &= \big(a^E-\ol u-a^L\,\ol v\big)\,\big[\wh q^E\big]
        - \ol{\wh q^E}\,\big([u] + a^L\,[v]\big).
    \end{aligned}
  \]
  This can be written as 
  \[
    \Big(I_4 + \ol{\wh q^E}\,{\scriptstyle{\(0\\1\\0\\0\)}}^\T\Big)\,\wh R^L
    = \wh R^E + \big(a^E-\ol u-a^L\,\ol v\big)\,\big[\wh q^E\big],
  \]
  so that the residuals are equivalent whenever
  \[
    \big(a^E-\ol u-a^L\,\ol v\big)\,\big[\wh q^E\big] = O(1)\,\wh R^E,
  \]
  and in particular, when using \eqref{LEs} to relate the wavespeeds.
\end{proof}

\subsection{Preconditioning and Residual}

We now introduce a preconditioner and calculate the corresponding
residual for a simple wave.  Thus suppose we consider a discontinuity
approximating a simple wave in the Lagrangian frame.  We use
\eqref{gd} and denote the wavespeed of the discontinuity by $a$.  Then
the unconditioned residual is
\begin{equation}
  \label{ucR}
  [f] - a\,[q]
  = \([u]-a\,[-v]\\{}[p]-a\,[u]\\
  -a\,\big[\frac12\,u^2+h-p\,v\big]\)
  = \(1&0&0\\0&1&0\\\ol p&\ol u&1\)
  \([u]-a\,[-v]\\{}[p]-a\,[u]\\
  a\,\big(\ol v\,[p]-[h]\big)\),
\end{equation}
where, as usual, $[\square]$ and $\ol \square$ are the (exact) jump
and average values of a state, respectively.  Next, we have
\[
  \frac{\dd h}{\dd p} = v, \qquad
  \frac{\dd v}{\dd p} = \frac{-1}{c^2},
\]
so integrating, we get
\begin{equation}
  \label{hvj}
  [h] = \big\langle v\big\rangle\,[p] \qcom{and}
  [-v] = \Big\langle \frac1{c^2}\Big\rangle\,[p],
\end{equation}
where we have defined
\begin{equation}
  \label{gbrak}
  \big\langle g\big\rangle := \frac1{p_r-p_l}
  \int_{p_l}^{p_r}g(\fs p,s)\;d\fs p
  = \int_{-1/2}^{1/2}g(\fs p,s)\;d\fs \eta, \qquad
  \fs p := \ol p + \fs\eta\,(p_r-p_l),
\end{equation}
so that $\big\langle g\big\rangle$ is an integral average of $g$
across the wave.  By \eqref{gdsp} we have
$[u]=\pm\big\langle\frac1c\big\rangle\,[p]$, and using this together
with \eqref{hvj} in \eqref{ucR}, we get, after simplifying,
\[
  \frac1{[p]}\(1&0&0\\0&1&0\\-\ol p&-\ol u&1\)
  \big([f] - a\,[q]\big)
  = \(\pm\big\<\frac1c\big\>\\1\\0\) - a\,\Big\<\frac1c\Big\>
  \( \Big\<\frac1{c^2}\Big\>\Big/\Big\<\frac1c\Big\>\\\pm 1\\
  \big\<\ol v - v\big\>/\Big\<\frac1c\Big\>\).
\]

We now choose the preconditioner and vector residual to be
\[
  A := \(\Big\<\frac1c\Big\>\Big/\Big\<\frac1{c^2}\Big\>
  &0&0\\0&1&0\\0&0&1\)\(1&0&0\\0&1&0\\-\ol p&-\ol u&1\)
  \qcom{and}
  R := A\,\big([f] - a\,[q]\big),
\]
and we then have
\begin{equation}
  \label{AR}
  \frac1{[p]}\,R = \(\pm\beta\\1\\0\)
  -a\,\Big\<\frac1c\Big\>\(1\\\pm 1\\
  \big\<\ol v - v\big\>/\Big\<\frac1c\Big\>\),
  \qcom{where}
  \beta := \Big\<\frac1c\Big\>^2\Big/\Big\<\frac1{c^2}\Big\>.
\end{equation}

\begin{lemma}
  \label{lem:gdR}
  The (preconditioned) least squared wavespeed $a_*$ and the ratio
  $\beta$ for a simple wave satisfy
  \[
    a_* = \pm\frac{\beta+1}{2\big\<1/c\big>} + O\big(\big|[p]\big|^2\big),
    \qcom{and}
    \beta = 1 + O\big(\big|[p]\big|^2\big),
  \]
  respectively, and the corresponding residual satisfies
  \[
    R_* = [p]\,\frac{\beta-1}2\(\pm1\\-1\\0\) + O\big(\big|[p]\big|^3\big)
     = O\big(\big|[p]\big|^3\big).
  \]
\end{lemma}

\begin{proof}
  The least square problem \eqref{AR} regarding
  $a\,\Big\<\frac1c\Big\>$ as the parameter is easily solved, namely
  \[
    a_*\,\Big\<\frac1c\Big\> = \pm\frac{\beta+1}{2+\big(\<\ol v -
      v\>/\big\<\frac1c\big\>\big)^2},
  \]
  which in turn yields
  \[
    R_* = \frac{[p]}{2+\big(\<\ol v - v\>/\big\<\frac1c\big\>\big)^2}
    \left[(\beta-1)\(\pm1\\-1\\0\) +
      \frac{\<\ol v - v\>}{\<1/c\>}
      \( \pm\beta\,\frac{\<\ol v - v\>}{\<1/c\>} \\
      \frac{\<\ol v - v\>}{\<1/c\>} \\\pm(\beta+1)\)\right].
  \]
  
  We now use Taylor's expansion in \eqref{gbrak} to evaluate
  \[
    \begin{aligned}
      \<g\>
      &= \int_{-1/2}^{1/2}\Big(g(\ol p,s) +
      \fs\eta\,[p]\,\frac{\dd g}{\dd p}(\ol p,s)
      + O\big(\fs\eta^2\,[p]^2\big)\Big)\;d\fs\eta\\
      &= g(\ol p,s) + O\big([p]^2\big),
    \end{aligned}
  \]
  for any $C^2$ function $g(p,s)$, since entropy $s$ is constant on a
  simple wave.  It now follows from \eqref{AR} that
  \[
    \beta = \frac{\Big(1/{c(\ol p,s)}+O\big([p]^2\big)\Big)^2}
    {1/{c^2(\ol p,s)}+O\big([p]^2\big)}
    = 1 + O\big([p]^2\big).
  \]
  Similarly using Taylor expansions,
  \[
    \begin{aligned}
      \ol v
      &= v(\ol p,s) +
      \frac{v\big(\ol p + \frac12\,[p]\big)-v(\ol p,s)}2 -
      \frac{v(\ol p,s)-v\big(\ol p - \frac12\,[p]\big)}2\\
      &= v(\ol p,s) +
      \Big(\frac14\,[p]\,\frac{\dd v}{\dd p}(\ol p,s)
      + O\big([p]^2\big)\Big) -
      \Big(\frac14\,[p]\,\frac{\dd v}{\dd p}(\ol p,s)
      + O\big([p]^2\big)\Big)\\
      &= v(\ol p,s) + O\big([p]^2\big) = \<v\> + O\big([p]^2\big).
    \end{aligned}
  \]
  It follows that $\<\ol v - v\> = O\big([p]^2\big)$, and the claims
  follow by direct substitution.
\end{proof}

\section{Isentropic Gas Dynamics}
\label{sec:psys}

We now further restrict to the $p$-system of isentropic flow in a
Lagrangian frame.  This is the first two equations of \eqref{euler},
namely
\begin{equation}
  \label{psys}
  \big(-v(p)\big)_t + u_x = 0, \qquad
  u_t + p_x = 0,
\end{equation}
in which $v=v(p)$ is assumed given.  Dropping the energy equation
allows us to drop the entropy variable $s$, so the state is fully
described by $w=(p,u)$, while we similarly drop the last component of
$q(w)$ and $f(w)$.  The acoustic impedance is again given by
\eqref{imp}, but is a function of the pressure only, namely
\[
  c(p) = \sqrt{-1\Big/\frac{dv}{dp}}.
\]
The simple wave and shock curves are reduced versions of \eqref{wvc},
\eqref{rgd}, and we can describe both the forward and backward waves
with the same equation by writing
\begin{equation}
  \label{wvp}
  w_r-w_l = (p_a-p_b)\,\ol r_\pm, \qquad
  \ol r_\pm := \(\mp 1\\1/\{c\}\),
\end{equation}
where the $\pm$ refer to the direction of propagation and the
subscripts $a$ and $b$ refer to the states \emph{ahead of} and
\emph{behind} the wave, respectively, and again
\[
  \{c\} := 1\Big/\Big\langle\frac1c\Big\rangle \com{or}
  \{c\} := |\sigma| = 1\Big/\sqrt{\Big\<\frac1{c^2}\Big\>},
\]
for simple waves or shocks, respectively, where we have used
\eqref{gdsh}, \eqref{hvj} without $s$.

\subsection{Wave Curves, Strengths and Residuals}

We can rewrite \eqref{wvp} as the scalar equation
\[
  u_r - u_l = \frac1{\{c\}}\,(p_a-p_b).
\]
and for simple waves, we can in turn write this as
\begin{equation}
  \label{gdwv}
  u_r-u_l = \Big\<\frac1c\Big\>\,(p_a-p_b)
  =: z(p_a)-z(p_b),
\end{equation}
where $z=z(p)$ is the \emph{Riemann coordinate}, defined by
\begin{equation}
  \label{zdef}
  z(p) := \int_0^p\frac1{c(\fs p)}\;d\fs p.
\end{equation}
The well-known fact that simple waves are characterized by constant
values of the Riemann invariants $z\mp u$ follows immediately.

Because all simple waves satisfy \eqref{gdwv}, a natural choice of
(signed) wave strength is
\begin{equation}
  \label{gdstr}
  \z := \frac12\,\big(z(p_a)-z(p_b)+u_r-u_l\big),
\end{equation}
which is also the change in (opposite) Riemann invariant.  The wave is
expansive if $p_a>p_b$, or equivalently $\z > 0$, and compressive
otherwise; this is consistent with Lax's condition that the pressure
should be greater behind a shock.  Having defined the strength of an
individual wave, the variation of a wave sequence $\Gamma$ consisting
of any number $n$ of waves $\g_j$ satisfying \eqref{wave}, is given by
\eqref{Vdef}, namely
\[
  \mc V(\Gamma) := \sum_{j=1}^n \big|\zeta_j\big|,
\]
again consistent with other treatments such as \cite{GL, DiPerna76, S}.

Having described the wave curves and defined the wave strengths, we
record the following lemma, which describes the effects of pairwise
interactions of waves.  It is well known that this holds
asymptotically for small amplitude waves, by Taylor expansions.
However, because we are treating interactions exactly by using simple
waves (that is compressions) for reflections rather than
(asymptotically small) shocks, the estimates become exact and extend
directly to waves of \emph{arbitrary} amplitudes.  Because there are
only two families, there are only two types of pairwise interactions:
either two waves of different families cross, or waves of the same
family merge; when waves merge, at least one of the incident waves
must be a shock.

\begin{lemma}
  \label{lem:gdints}
  If any wave crosses an opposite simple wave, its strength is left
  unchanged.  If any wave crosses an opposite shock, its (absolute)
  strength is increased.  When two waves of the same family merge or a
  compression collapses, the wave strengths in that family combine
  linearly.  If both incident waves are compressive or a compression
  collapses, an opposite rarefaction is reflected, while if one
  incident wave is a rarefaction, an opposite compression is
  generated.  For a fixed shock, the strength of the reflected simple
  wave is a strictly monotone function of the strength of the other
  (merging) wave.  Similarly, if a compression collapses, the strength
  of the reflected rarefaction is strictly monotone in the strength of
  the collapsed wave.
\end{lemma}

\begin{proof}
  The proof follows because, apart from the crossing of two shocks or
  simultaneous collapse of compressions, \emph{all} wave interactions
  result in at least one emergent simple wave.  The corresponding
  Riemann invariant is constant along that wave, and defines the
  strength of both incident and emerging waves, which is then
  unchanged.  The fact that strengths increase following crossing of a
  shock is a consequence of monotonicity of $c(p)$, which is genuine
  nonlinearity.  Similarly the monotonicity of the reflected simple
  wave from a merge or compression collapse follows from monotonicity
  of $c(p)$.  We omit the details of the proof and exact estimates,
  which can be found in the preprint \cite{Ypsys} and our forthcoming
  paper \cite{BY2}.
\end{proof}

We note that there is an interplay between wave strengths, defined via
Riemann invariants, and residuals.  As we have seen, wave strengths
(and Riemann invariants) combine linearly for waves in the same family
and are particularly well suited to simple waves, allowing a clean
treatment of multiple interactions.  On the other hand, shocks, which
are nonlinear and strengthen crossing simple waves, are characterized by
vanishing residual because the jump conditions are satisfied exactly.

For the isentropic system, Lemma~\ref{lem:gdR} simplifies to give the
\emph{exact} wavespeed and (preconditioned) residual of a jump
approximation as
\[
  a_* = \pm\frac{\beta+1}{2\,\<1/c\>} \qcom{and}
  R_* = [p]\,\frac{\beta-1}2\(\pm1\\-1\),
\]
respectively, where $\beta$ is given by \eqref{AR}.  In particular,
the residuals are multiples of constant orthogonal vectors, so can be
regarded as scalars,
\[
  \big|R_*\big| = \frac1{\sqrt2}\,\big|[p]\,(\beta-1)\big|.
\]

It follows that when splitting up a \emph{monotone} simple wave into
smaller pieces, the split residual can be calculated as the sum of the
(scalar) residuals of the smaller waves,
\[
  \big|R_{sum}\big| = \frac1{\sqrt2}\,
  \Big|\sum_i[p]_i\,(\beta_i-1)\Big|,
\]
because the signs of each piece are the same.  

\begin{lemma}
  \label{lem:pres}
  If simple waves are approximated using least square wave speeds,
  then whenever \emph{any} monotone simple wave is split into smaller waves,
  the absolute total residual decreases.
\end{lemma}

Although we do not split compressions in the mFT scheme, we do split
them when initializing.  This result is essentially a statement of
convexity: monotonicity of the wavespeed is implicit by our treatment
of rarefactions and compressions as different simple waves which are
never merged, and integrals of monotone functions are convex.

\begin{proof}
  Denote the \emph{signed scalar residual} of a jump approximating a
  simple wave between pressures $p_a$ ahead of the wave and $p_b$
  behind the wave as
  \[
    R_*(p_a,p_b) := \frac1{\sqrt2}\,[p]\,(\beta-1), \qcom{where}
    [p] := p_b - p_a,
  \]
  and $\beta$ is given by \eqref{AR}.  Using \eqref{gbrak}, we write
  \[
    [p]\,\beta
    = [p]\,\Big\<\frac1c\Big\>^2\Big/\Big\<\frac1{c^2}\Big\>
    = \frac{\I1ab^2}{\I2ab},
  \]
  where we have defined the integral
  \[
    \I kab := \int_{p_a}^{p_b}\frac{d\fs p}{c(\fs p)^k},
    \qquad k = 1,\;2,
  \]
  so the residual is expressed as
  \[
    \sqrt2\,R_*(p_a,p_b) = \frac{\I1ab^2}{\I2ab} + p_a - p_b.
  \]
  
  Now suppose that the simple wave is approximated by two jumps, each
  given the least squared speed, and denote the newly introduced
  middle pressure by $p_m$, where $p_m$ is any pressure between $p_a$
  and $p_b$, and with corresponding velocity given by \eqref{gdsimp}.
  Then the combined signed scalar residual of the split pair is
  $R_*(p_a,p_m)+R_*(p_m,p_b)$.  The difference in residuals after
  splitting is then
  \[
    \begin{aligned}
      \sqrt2\,\Delta R_*
      &:= \sqrt2\,\big(R_*(p_a,p_m)+R_*(p_m,p_b)-R_*(p_a,p_b)\big)\\
      &= \frac{\I1am^2}{\I2am} + \frac{\I1mb^2}{\I2mb}
        - \frac{\big(\I1am + \I1mb\big)^2}{\I2am + \I2mb},
    \end{aligned}
  \]
  and after simplifying this becomes
  \[
    \sqrt2\,\Delta R_* =
    \frac{\big(\I1am\,\I2mb - \I1mb\,\I2am\big)^2}
    {\I2am\,\I2mb\,\big(\I2am + \I2mb\big)}.
  \]
    
  It now follows that $\Delta R_*$ has the sign of $p_b-p_a$, which is
  also that of $p_b-p_m$ and $p_m-p_a$.  On the other hand, the sign
  of $R_*(p_a,p_b)$ is that of $(p_b-p_a)\,(\beta-1)$, and similarly
  for $R_*(p_a,p_m)$ and $R_*(p_m,p_b)$.  Using \eqref{gbrak} again,
  we have
  \[
    \beta = \Big\<\frac1c\Big\>^2\Big/\Big\<\frac1{c^2}\Big\>
    = \Big(\int_{-1/2}^{1/2}\;\frac1{c(\fs p)}d\fs \eta\Big)^2\Big/
    \int_{-1/2}^{1/2}\Big(\frac1{c(\fs p)}\Big)^2\;d\fs \eta < 1,
  \]
  by Jensen's inequality.  Thus $\Delta R_*$ and $R_*(p_a,p_b)$ have
  opposite signs, which implies
  \[
    \big|R_*(p_a,p_m)+R_*(p_m,p_b)\big| < \big|R_*(p_a,p_b)\big|,
  \]
  whenever $p_m$ is between $p_a$ and $p_b$.  A simple induction now
  shows that any further splits will further reduce the (absolute)
  residual.  
\end{proof}

\subsection{No Finite Accumulation of Interaction Times}

Finally, we show that for the $p$-system \eqref{psys}, because there
are only two families of waves, there is no finite accumulation of
interaction times.  The essential idea is that except for a finite
number of compression collapses or pairwise interactions, which
necessarily involve at least one strong wave, \emph{all} 
interactions consist of two waves in and two waves out.  This in turn
implies that the number of waves remains finite for all intermediate
times, and the interaction times can then accumulate only at
$t=\infty$.

Proof that there is no accumulation of interactions requires that we
have some control on the total variation of approximate solutions, so
in this context we \emph{assume} such bounds, which will be derived
in~\cite{BY2}.  Specifically, we assume the existence of a (nonlocal)
Glimm potential $\mc P = \mc P(\Gamma)\ge 0$, defined for the sequence
$\Gamma = \Gamma^t$, such that the Glimm functional
$\mc G := \mc V + \mc P$ for the total variation is nonincreasing
across \emph{every} pairwise interaction.  Because it is defined on a
wave sequence, the functional $\mc P$ is defined for the mFT
approximations for each finite discretization parameter $\e<\e_0$.

\begin{assum}
  \label{ass:pot}
  There is a nonlocal, non-negative Glimm potential
  $\mc P = \mc P(\Gamma^t)$ such that both $\mc P$ and $\mc V+\mc P$
  are decreasing,
  \begin{equation}
    \label{pot}
    \Delta\,\mc P\le 0 \qcom{and} \Delta(\mc V+\mc P) \le 0,
  \end{equation}
  across \emph{any} pairwise interaction, and such that
  $\mc V(\Gamma^0)+\mc P(\Gamma^0)$ is bounded.  That is, if $\Gamma'$
  is the wave sequence that arises after interaction of any pair of
  waves in $\Gamma$, then we have both
  \[
    \mc P(\Gamma')\le \mc P(\Gamma) \qcom{and}
    \mc V(\Gamma') + \mc P(\Gamma') \le \mc V(\Gamma) + \mc P(\Gamma).
  \]
\end{assum}

Here we are treating the discretization parameter $\e>0$ as fixed, so
we do not require uniformity in $\e$ in the assumption or in
Theorem~\ref{thm:noacc} below.  In \cite{BY2}, we will obtain
estimates uniform in both $\e$ and $t$, while assuming that the
solution remains uniformly away from vacuum, for waves of large
amplitude.  We expect that extension to solutions near or including
vacuum will require a further analysis of the relation between
discretization parameter $\e$, time of existence $T$ and lower bound
on pressure $\ul p$.  Specifically, we expect that any uniform limit
\[
  \limsup_{\e\to0^+} \mc P(\Gamma^0)
\]
will depend on some lower bound $\ul p>0$ for the pressure, away from
vacuum, and such a lower bound will in general depend on the time $T$
for which the solution is evolving; see~\cite{Chen2,Gengex}.

It is immediate that the existence of a Glimm potential implies total
variation bounds, and if we further have uniformity in $\e$ for $t<T$,
Theorem~\ref{thm:conv} implies convergence and existence of solutions
up to time $T$.

As a consequence of our assumption that there is a decreasing Glimm
potential for the $p$-system, we obtain a bound on the number of waves
in the mFT approximation for each fixed $\e>0$.  This is analogous to
the statement that a $BV$ function can have only a finite number of
jumps of finite size.

\begin{lemma}
  \label{lem:numwvs}
  Suppose that Assumption \ref{ass:pot} holds.  Then there exists some
  finite
  \[
    N = N\big(\mc P(\Gamma^0),\e\big),
  \]
  such that there are at most $N$ waves in the mFT approximation for
  any time $t$ for which the mFT approximation is defined.
\end{lemma}

\begin{proof}
  Since $\e>0$, there are only finitely many waves in the initial wave
  sequence $\Gamma^{0+}$.  Without loss of generality, we assume each
  interaction in the mFT is either the collapse of a compression or a
  pairwise interaction.  That is, there are at most two incident waves
  in any interaction.

  Next, there can be more than two waves leaving an interaction only
  if one of those is a rarefaction which has been split after the
  interaction.  Recall that a rarefaction is split if its strength
  after interaction satisfies $\z > \ee r$.  If both waves leaving the
  interaction are rarefactions, then by Lemma \ref{lem:gdints}, the
  incident waves are also crossing rarefactions, and so their
  strengths are unchanged, and in particular they will not be split
  after the interaction.

  It follows that the number of waves in the mFT approximation can 
  increase only if the emerging waves consist of a shock of one family
  and a rarefaction of the other.  There are three such interactions
  at which the number of waves increases:
  \begin{enumerate}
  \item[(a)] collapse of a compression: 1 wave in, 2 or more out;
  \item[(b)] shock and compressive wave merge: 2 waves in, 3 or more out;
  \item[(c)] rarefaction crosses a (strong) shock: 2 waves in, 3 or more out.
  \end{enumerate}
  We show that each of these can occur only a finite number of times.
  
  In case (a), a compression of strength $\z$ collapses: according to
  Lemma \ref{lem:gdints}, a shock of strength $\z$ is transmitted and
  a rarefaction, say of strength $\mu(\z)>0$, is reflected.  By
  construction, all compressions have (absolute) strength at least
  $\ee w$, and so by monotonicity,
  \[
    \Delta\mc V = \mu(\z) \ge \mu\big(\ee w\big), \qcom{so also}
    \Delta\mc P \le -\mu\big(\ee w\big).
  \]
  Since $\mc P(\Gamma^0)$ is finite and $\mc P$ is decreasing, there
  can be no more than
  \[
    N_{(a)} := \frac{\mc P(\Gamma^0)}{\mu\big(\ee w\big)}
  \]
  such interactions.

  For case (b), because two waves merge to produce a reflected
  rarefaction of strength $\ee r$, at least one incident wave is a
  shock and the other is compressive, either a shock or simple wave.
  We \emph{overcount} by considering the number of such interactions
  in which a rarefaction of strength at least $\ee r/2$ is reflected.
  Again by Lemma \ref{lem:gdints}, the shock strengths add linearly,
  so as above we get
  \[
    \Delta\mc V \ge \frac12\,\ee r, \qquad
    \Delta\mc P \le -\frac12\,\ee r, \qcom{and}
    N_{(b)} := \frac{2\,\mc P(\Gamma^0)}{\ee r},
  \]
  where $N_{(b)}$ counts the number of these interactions.

  In case (c), rather than counting only the interactions across which
  the number of rarefactions increases by a single rarefaction
  attaining strength $|\zeta|>\ee r$ across a single shock, we count
  the number of rarefactions that increase their strength from less than
  $\frac45\,\ee r$, say, to more than $\ee r$ after crossing several
  shocks, and increasing its strength each time.  Since shocks have no
  change in wavestrength after crossing a rarefaction, the change in
  variation across these cumulative interactions (excluding
  interactions of other ``far away'' waves) is
  \[
    \Delta\mc V \ge \frac15\,\ee r, \qcom{and}
    \Delta\mc P \le -\frac15\,\ee r,
  \]
  so that
  \[
    N_{(c)} := \frac{5\,\mc P(\Gamma^0)}{\ee r}.
  \]
  Note that any \emph{new} rarefaction whose strength is greater than
  $\frac45\,\ee r$, so is not included in this count, has previously
  been counted in the waves of step (b) above.

  In each case above, the emerging rarefaction is split into a number
  of smaller rarefactions.  According to our construction, each such
  split rarefaction has strength \emph{at least} $\kappa\,\ee r$, for
  $\kappa>0$.  On the other hand, because the total variation is
  bounded, the total increase of strength is uniformly bounded, and
  the total number of new waves generated in each such interaction is
  bounded by some constant $M$, which may depend on $\e$ and $\kappa$.
  Since the only interactions in which the number of waves can
  increase are enumerated above, it follows that the the total number
  of new waves generated throughout the scheme is at most
  \[
    \big(N_{(a)} + N_{(b)} + N_{(c)}\big)\,M < \infty.
  \]
  It follows that the total number of waves in the mFT approximation
  is bounded, and the proof is complete.
\end{proof}

\begin{cor}
  \label{cor:2in2out}
  For fixed $\e>0$, let $\{t_k\}$ denote the sequence of interaction
  times in the mFT approximation for the $p$-system, as in
  \eqref{times}.  If Assumption \ref{ass:pot} holds, then there is
  some
  \[
    K = K\big(\mc P(\Gamma^0),\e\big)
  \]
  such that all interactions beyond time $t_K$ consist of two waves in
  and two waves out.  In particular, the two waves emerging from any
  interaction are opposite; that is, one backward and one
  forward  wave emerge from each interaction for $k>K$.
\end{cor}

\begin{proof}
  Because there are only finitely many interactions in which the
  number of waves increases, there is some $K$ beyond which no new
  waves are generated.  In each subsequent interaction, there can be
  at most two waves out, and because no waves are split beyond $t_K$,
  the outgoing waves must be from different families.
\end{proof}

According to Theorem \ref{thm:hatU}, the mFT scheme can be defined
for all interactions times $t_k$, $k\in\B N$.  Under the further
Assumption \ref{ass:pot}, we have shown that the number of waves in
the scheme at any time is bounded.  Our final task is to show that
under this assumption, the scheme can be defined for all times.

Generally, we expect bounds for $\mc P$ to depend on a minimum lower
bound $\ul p$ for the pressure, to avoid the degeneracy at vacuum,
which is defined by $p = 0$.  However, Geng Chen has shown that for
general large amplitude data, any such lower bound on the pressure
must be time dependent: in particular, given any $t_*>0$ and
$\ul p>0$, there is piecewise constant $BV$ data for which
\[
  p(0,t) < \ul p \qcom{for some} t < t_*,
\]
so that time-independent lower bounds on the pressure are not
available~\cite{Gengex}.  Despite this example, we are able to prove
that the scheme can be continued beyond any finite time, assuming
appropriate bounds.

\begin{theorem}
  \label{thm:noacc}
  Fix $\e>0$ and for some $T>0$, suppose that Assumption \ref{ass:pot}
  holds for $t<T$.  Then the scheme can be defined beyond time
  $T$; that is, there is some $k\in\B N$ such that the $k$-th
  interaction time satisfies $t_k>T$.
\end{theorem}

Here we note that in principle, the approximate solution that extends
beyond $t=T$ may include the vacuum; however, our assumption that a
potential exists implicitly implies that there is no vacuum up until
time $T$.  Moreover, if a vacuum or vacuums arise at time $t=T+$,
finite speed of propagation and finiteness of the number of waves
means that the approximation is nevertheless defined on some open time
interval containing $T$.  We expect that in the limit $\e\to0^+$, any
vacuum that forms will do so either at $t=0+$ or at $t=\infty$,
consistent with \cite{YpRP2}.

\begin{proof}
  Denote the set of all interaction times in the mFT scheme by
  $\{t_k\}$, as in \eqref{times}.  According to Lemma~\ref{lem:numwvs}
  and Corollary~\ref{cor:2in2out}, there are integers $K$ and $N$ so
  that for all $k>K$, $t_k>t_K$, there are exactly $N$ waves in the
  approximation.

  For $t > t_K$, we define $N$ non-overlapping, linearly ordered
  trajectories, as follows.  As in \eqref{traj}, after each
  interaction time $t_j+$, set
  \[
    \rho_\ell = \g_\ell \in \Gamma^{t_j+},
  \]
  so we are just counting the waves in spatial order; now extend the
  trajectory up to $t_{j+1}-$ and then to the interaction times $t_j$
  by continuity.  Thus the $\ell$-th trajectory is simply the path
  traced out by the $\ell$-th wave; this clearly accounts for all $N$
  waves and their interactions.  Moreover, the trajectories are
  non-overlapping, and intersect only at pairwise wave interactions.

  Now assume the scheme cannot be defined beyond $T$; this means that
  the interaction times $t_k$ accumulate at the finite time $T$, that
  is
  \[
    t_k \to t_\infty := T \qcom{as} k\to\infty.
  \]
  By compactness, there must be some accumulation point
  $(x_\infty,t_\infty)$ of nontrivial interaction points $(x_k,t_k)$.
  Also, since each trajectory is Lipschitz, it extends up to
  $t=t_\infty$ by continuity.  Referring to the $\ell$-th trajectory
  as in Section \ref{sec:traj}, define the indices
  \[
    m_- = \max\big\{\ell\;:\: x_\ell^p(t_\infty) < x_\infty\big\},
    \qcom{and}
    m_+ = \min\big\{\ell\;:\: x_\ell^p(t_\infty) > x_\infty\big\}.
  \]
  Clearly, we must have $m_-< m_+$, and
  \[
    x_\ell^p(t_\infty) = x_\infty \qcom{for each}
    m_-<\ell<m_+.
  \]
  
  Now choose $\ol\la>\sup c(p)$ and set $\ul\la := -\ol\la$, and
  consider the space-time triangle $D_\Delta$ defined in \eqref{Ddel}.
  There must be some large $M$ such that for $j\le m_-$ or $j\ge m_+$,
  and for $t>t_M$, all trajectories $\rho_j$ are outside $D_\Delta$,

  Now, if $m_+ = m_- + 1$, then there are no waves in $D_\Delta$
  beyond $t=t_M$, so no trajectories meet $(x_\infty,t_\infty)$, and
  it cannot be an accumulation point.  On the other hand, for $t>t_M$,
  and any $\ell$ such that $m_-<\ell<m_+$, the trajectory $\rho_\ell$
  must remain inside $D_\Delta$; if not, then it eventually leaves
  $D_\Delta$, so ends away from $(x_\infty,t_\infty)$, contradicting
  maximality or minimality of $m_\mp$.

  Any two trajectories inside $D_\Delta$ and beyond $t_M$ cannot
  interact, for if they did, there would be both forward and backward
  waves emerging from the interaction, and one of these would have to
  leave $D_\Delta$, again a contradiction.

  Thus each trajectory in $D_\Delta$ lies on a ray running into
  $(x_\infty,t_\infty)$, and these cannot interact before
  $t=t_\infty$, contradicting the fact that $(x_\infty,t_\infty)$ is
  an accumulation of nontrivial interaction points.  This
  contradiction means that there can be no such finite $t_\infty$, so
  there must be some interaction time $t_k>T$, completing the proof.
\end{proof}

As a final corollary, we note that this theorem together with
Lemma~\ref{lem:RG} and Theorem~\ref{thm:conv} yield a \emph{rate of
  convergence} of the mFT scheme for the $p$-system.

\begin{cor}
  \label{cor:rate}
  If Assumption \ref{ass:pot} holds, then the mFT approximation
  converges to a solution of the $p$-system with convergence rate
  $\e^{2e_s}$ as $\e\to 0^+$, uniformly for $t\le T$.
\end{cor}

\begin{proof}
  According to Theorem~\ref{thm:noacc}, for $\e>0$, each approximation
  can be defined beyond $T$, with a uniform bound on the total
  variation for $t\le T$.  By Theorem~\ref{thm:conv}, the limit is a
  weak* solution, and because we have not used composite waves, we
  have $Q(\e)=W_H(\Gamma,\e) = 0$.  Thus \eqref{residG} reduces to
  \[
    \big\|\mc R(\Gamma)\big\|_M \le \mc V(\Gamma^t)\,K_s\,\e^{2e_s}
    \le \big(\mc V(\Gamma^0)+\mc P(\Gamma^0)\big)\,K_s\,\e^{2e_s},
  \]
  and the proof is complete.
\end{proof}

\subsection{Breakdown of Periodic Wave Patterns}

As an application of the mFT scheme, we briefly analyze the periodic
wave pattern introduced in \cite{BCZ18}.  This example was introduced
to demonstrate the difficulty of obtaining $BV$ bounds for the
$p$-system and provides an enlightening example of a piecewise
approximation in which the total variation can in principle become
unbounded in finite time.  The authors first find a closed cyclic set
of states which are connected by a series of simple and shock waves.
They then describe a pattern of interactions under which this cycle of
states is visited periodically, provided the widths and speeds of the
various waves are carefully controlled.  By then perturbing this
picture, they are able to produce an apparent piecewise constant
approximation in which the total variation blows up.

We develop the periodic cycle in state space, and show that such a
periodic cycle requires a certain minimum strength of simple waves.
Given such a cycle, we then consider the mFT approximation to data
which in principle generates the periodic pattern.  In particular, we
show that the periodicity of the characteristic pattern cannot be
sustained as the discretization parameter $\e\to0$.  The periodic wave
pattern breaks down because our requirement that the residual vanish
in the limit means that the large simple waves must be finely
discretized, which implies that interactions including simple waves
are spread out over time and space.  This in turn prevents the
accumulation of wave interactions, perturbation of which yields blowup
of the variation in \cite{BCZ18}.  Instead, in the mFT approximation,
the simple waves cause the pattern to spread out rapidly, so that
decay of waves due to cancellation between shocks and rarefactions of
the same family eventually becomes the dominant feature of the
solution.

\begin{figure}[thb]
\centering{  \hfill
  \begin{minipage}{0.6\linewidth}
    \includegraphics[height=2in]{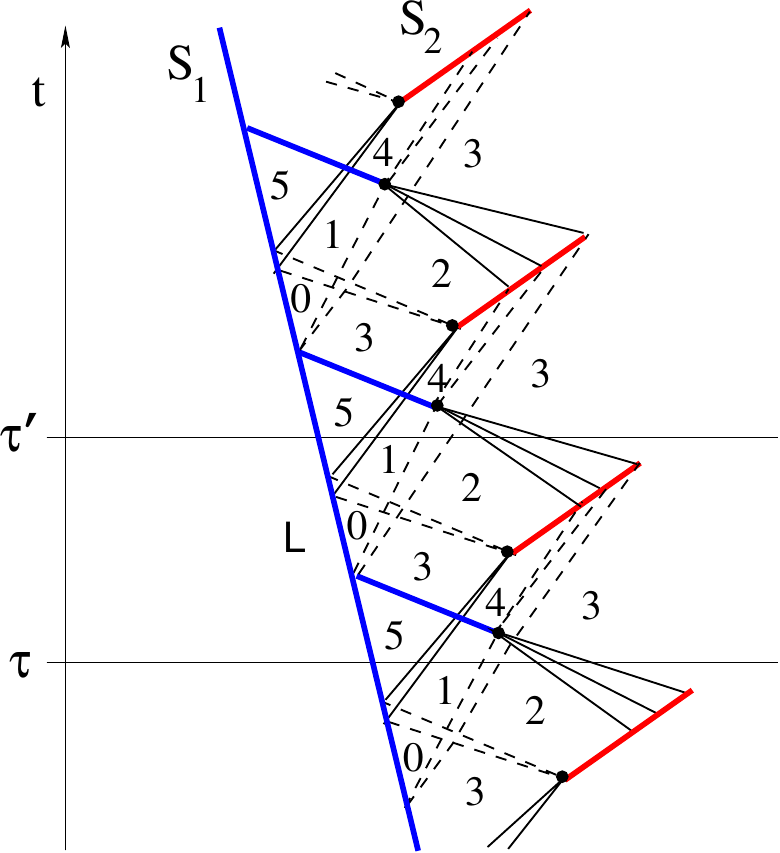}
  \end{minipage}\hspace*{-1.5in}
  \begin{minipage}{0.38\linewidth}
    \begin{tikzpicture}

\begin{axis}[
hide x axis,
hide y axis,
tick align=outside,
tick pos=left,
xlabel={h},
xmajorgrids,
xmin=-2.13968611247656, xmax=44.9334083620079,
ylabel style={rotate=-90.0},
ylabel={u},
ymajorgrids,
ymin=-54.8560814317738, ymax=112.748378098054,
ytick={0},
yticklabels={-}
]
\addplot [thick, fwdSimple]
table {%
5 -47.2376969076908
42.7937222495313 -9.44397465815947
};
\addplot [thick, bakShock]
table {%
0 0
};
\addplot [thick, bakSimple]
table {%
17.1748737956859 16.1748737956859
42.7937222495313 -9.44397465815947
};
\addplot [thick, contact]
table {%
0 0
};
\addplot [thick, fwdSimple]
table {%
17.1748737956859 16.1748737956859
23.3629173576053 22.3629173576053
};
\addplot [thick, fwdSimple]
table {%
0 0
};
\addplot [thick, fwdSimple]
table {%
1 0
17.1748737956859 16.1748737956859
};
\addplot [thick, fwdShock]
table {%
0 0
};
\addplot [thick, bakSimple]
table {%
1 0
26.6188484538454 -25.6188484538454
};
\addplot [thick, bakShock]
table {%
1 0
1.04040404040404 -0.0404991839643687
1.08080808080808 -0.0815418133431666
1.12121212121212 -0.123604571647166
1.16161616161616 -0.167106168327495
1.2020202020202 -0.212420063262126
1.24242424242424 -0.259884205420999
1.28282828282828 -0.309808558671397
1.32323232323232 -0.362480968322238
1.36363636363636 -0.418171770068413
1.4040404040404 -0.477137436367407
1.44444444444444 -0.539623479478134
1.48484848484848 -0.605866775823979
1.52525252525253 -0.676097436583212
1.56565656565657 -0.750540320107528
1.60606060606061 -0.829416259940769
1.64646464646465 -0.912943065786692
1.68686868686869 -1.00133634230577
1.72727272727273 -1.09481016107605
1.76767676767677 -1.19357761369153
1.80808080808081 -1.29785126825577
1.84848484848485 -1.40784354706301
1.88888888888889 -1.52376703975153
1.92929292929293 -1.64583476344421
1.96969696969697 -1.77426037919468
2.01010101010101 -1.90925837230671
2.05050505050505 -2.05104420269507
2.09090909090909 -2.19983443033111
2.13131313131313 -2.3558468199111
2.17171717171717 -2.51930042815197
2.21212121212121 -2.69041567652514
2.25252525252525 -2.86941441175504
2.29292929292929 -3.05651995601369
2.33333333333333 -3.2519571484195
2.37373737373737 -3.45595237918229
2.41414141414141 -3.66873361751792
2.45454545454545 -3.89053043427489
2.49494949494949 -4.12157402006561
2.53535353535354 -4.3620971995707
2.57575757575758 -4.61233444258087
2.61616161616162 -4.87252187225495
2.65656565656566 -5.14289727099974
2.6969696969697 -5.42370008431715
2.73737373737374 -5.71517142291281
2.77777777777778 -6.0175540633174
2.81818181818182 -6.33109244723566
2.85858585858586 -6.65603267980726
2.8989898989899 -6.99262252693765
2.93939393939394 -7.34111141183481
2.97979797979798 -7.70175041086892
3.02020202020202 -8.07479224885591
3.06060606060606 -8.46049129385205
3.1010101010101 -8.8591035515348
3.14141414141414 -9.27088665923523
3.18181818181818 -9.69609987967839
3.22222222222222 -10.1350040944806
3.26262626262626 -10.5878617974461
3.3030303030303 -11.0549370877003
3.34343434343434 -11.5364956626909
3.38383838383838 -12.0328048110854
3.42424242424242 -12.5441334055896
3.46464646464646 -13.0707518957063
3.50505050505051 -13.6129323004548
3.54545454545455 -14.1709482010658
3.58585858585859 -14.7450747336653
3.62626262626263 -15.3355885819602
3.66666666666667 -15.942767969936
3.70707070707071 -16.5668926545743
3.74747474747475 -17.2082439186001
3.78787878787879 -17.8671045632632
3.82828282828283 -18.5437589011609
3.86868686868687 -19.2384927491071
3.90909090909091 -19.9515934210506
3.94949494949495 -20.6833497210479
3.98989898989899 -21.4340519362919
4.03030303030303 -22.2039918302006
4.07070707070707 -22.993462635566
4.11111111111111 -23.8027590477675
4.15151515151515 -24.6321772180481
4.19191919191919 -25.4820147468569
4.23232323232323 -26.3525706772582
4.27272727272727 -27.2441454884062
4.31313131313131 -28.1570410890889
4.35353535353535 -29.0915608113374
4.39393939393939 -30.0480094041046
4.43434343434343 -31.0266930270103
4.47474747474747 -32.0279192441544
4.51515151515152 -33.0519970179969
4.55555555555556 -34.0992367033042
4.5959595959596 -35.1699500411629
4.63636363636364 -36.2644501530584
4.67676767676768 -37.3830515350189
4.71717171717172 -38.5260700518249
4.75757575757576 -39.6938229312813
4.7979797979798 -40.8866287585544
4.83838383838384 -42.1048074705702
4.87878787878788 -43.3486803504755
4.91919191919192 -44.6185700221598
4.95959595959596 -45.9148004448378
5 -47.2376969076908
};
\addplot [thick, fwdShock]
table {%
5 -47.2376969076908
5.12297852318875 -47.1146094494268
5.24595704637749 -46.9908883479069
5.36893556956624 -46.8659508494445
5.49191409275499 -46.7392623756642
5.61489261594373 -46.6103296266413
5.73787113913248 -46.4786947615836
5.86084966232123 -46.3439304687301
5.98382818550997 -46.2056357725849
6.10680670869872 -46.0634324553636
6.22978523188747 -45.9169619923444
6.35276375507621 -45.7658829190077
6.47574227826496 -45.6098685624119
6.59872080145371 -45.448605080989
6.72169932464245 -45.2817897664244
6.8446778478312 -45.1091295689984
6.96765637101995 -44.930339814061
7.09063489420869 -44.7451430824748
7.21361341739744 -44.5532682321146
7.33659194058619 -44.354449541033
7.45957046377493 -44.1484259558249
7.58254898696368 -43.934940431165
7.70552751015243 -43.7137393485346
7.82850603334117 -43.4845720038719
7.95148455652992 -43.247190155329
8.07446307971867 -43.0013476235423
8.19744160290741 -42.7467999378662
8.32042012609616 -42.4833040229025
8.44339864928491 -42.2106179204154
8.56637717247365 -41.9285005423678
8.6893556956624 -41.63671145137
8.81233421885115 -41.3350106653077
8.93531274203989 -41.0231584833283
9.05829126522864 -40.7009153307175
9.18126978841738 -40.368041620507
9.30424831160613 -40.0242976299181
9.42722683479488 -39.6694433899778
9.55020535798362 -39.3032385868439
9.67318388117237 -38.9254424735505
9.79616240436112 -38.5358137910383
9.91914092754986 -38.1341106974654
10.0421194507386 -37.7200907049138
10.1650979739274 -37.2935106227036
10.2880764971161 -36.8541265066229
10.4110550203049 -36.401693613453
10.5340335434936 -35.9359663602439
10.6570120666823 -35.4566982878497
10.7799905898711 -34.9636420282934
10.9029691130598 -34.4565492755712
11.0259476362486 -33.9351707595538
11.1489261594373 -33.399256222676
11.2719046826261 -32.8485543991392
11.3948832058148 -32.2828129963801
11.5178617290036 -31.7017786785858
11.6408402521923 -31.1051970520573
11.7638187753811 -30.4928126522429
11.8867972985698 -29.8643689322829
12.0097758217586 -29.219608252923
12.1327543449473 -28.5582718736649
12.255732868136 -27.88009994504
12.3787113913248 -27.1848315019015
12.5016899145135 -26.4722044576383
12.6246684377023 -25.7419555992276
12.747646960891 -24.9938205830476
12.8706254840798 -24.2275339313811
12.9936040072685 -23.4428290295469
13.1165825304573 -22.6394381236009
13.239561053646 -21.8170923185567
13.3625395768348 -20.9755215770764
13.4855181000235 -20.1144547185907
13.6084966232123 -19.2336194188089
13.731475146401 -18.3327422095827
13.8544536695898 -17.4115484790929
13.9774321927785 -16.4697624723291
14.1004107159672 -15.5071072918356
14.223389239156 -14.5233048986997
14.3463677623447 -13.5180761137599
14.4693462855335 -12.4911406190139
14.5923248087222 -11.4422169592079
14.715303331911 -10.3710225435907
14.8382818550997 -9.27727364781597
14.9612603782885 -8.16068541598111
15.0842389014772 -7.02097186278655
15.207217424666 -5.85784587580615
15.3301959478547 -4.67101921785726
15.4531744710435 -3.46020252945989
15.5761529942322 -2.22510533137675
15.699131517421 -0.965436027225415
15.8221100406097 0.319098093844929
15.9450885637984 1.62879085441907
16.0680670869872 2.96393718603008
16.1910456101759 4.32483312632859
16.3140241333647 5.71177581621166
16.4370026565534 7.12506349691726
16.5599811797422 8.56499550708779
16.6829597029309 10.0318722798076
16.8059382261197 11.5259953396186
16.9289167493084 13.0476672995156
17.0518952724972 14.5971918579276
17.1748737956859 16.174873795686
};
\addplot [thick, extra]
table {%
0 0
};
\addplot [thick, bakShock]
table {%
23.3629173576053 22.3629173576053
23.5591881140894 22.1666259957577
23.7554588705735 21.9702123529329
23.9517296270576 21.7735577659525
24.1480003835417 21.5765473273094
24.3442711400258 21.3790696864774
24.5405418965099 21.1810168629564
24.736812652994 20.9822840702669
24.9330834094781 20.7827695501632
25.1293541659622 20.5823744163943
25.3256249224463 20.3810025073856
25.5218956789304 20.1785602472639
25.7181664354145 19.9749565146897
25.9144371918986 19.7701025189982
26.1107079483827 19.5639116831875
26.3069787048668 19.3562995333243
26.5032494613509 19.1471835939682
26.699520217835 18.936483289244
26.8957909743191 18.7241198492153
27.0920617308032 18.5100162212392
27.2883324872873 18.2940969860013
27.4846032437714 18.076288277953
27.6808740002555 17.8565177098884
27.8771447567396 17.6347143014196
28.0734155132237 17.4108084111225
28.2696862697078 17.1847316721404
28.4659570261919 16.9564169310478
28.662227782676 16.7257981897879
28.8584985391601 16.4928105505112
29.0547692956442 16.2573901631518
29.2510400521283 16.0194741755884
29.4473108086124 15.7790006862492
29.6435815650965 15.535908699024
29.8398523215806 15.2901380803602
30.0361230780647 15.0416295184228
30.2323938345488 14.790324484209
30.4286645910329 14.5361651945119
30.624935347517 14.2790945766362
30.8212061040011 14.0190562347739
31.0174768604852 13.7559944179517
31.2137476169693 13.4898539894706
31.4100183734534 13.2205803977597
31.6062891299375 12.9481196485705
31.8025598864216 12.6724182784463
31.9988306429057 12.3934233293995
32.1951013993898 12.1110823247365
32.391372155874 11.825343245975
32.587642912358 11.5361545107954
32.7839136688421 11.2434649519794
32.9801844253263 10.9472237972841
33.1764551818104 10.6473806502075
33.3727259382945 10.3438854716014
33.5689966947786 10.0366885620915
33.7652674512627 9.72574054526468
33.9615382077468 9.41099235158775
34.1578089642309 9.0923952030227
34.354079720715 8.76990059830507
34.5503504771991 8.44346029885466
34.7466212336832 8.11302631528867
34.9428919901673 7.7785508945093
35.1391627466514 7.43998650733922
35.3354335031355 7.09728583667935
35.5317042596196 6.75040176616533
35.7279750161037 6.39928736929948
35.9242457725878 6.04389589903689
36.1205165290719 5.68418077780493
36.316787285556 5.32009558793653
36.5130580420401 4.95159406249892
36.7093287985242 4.57863007649992
36.9055995550083 4.20115763845494
37.1018703114924 3.81913088229888
37.2981410679765 3.43250405962757
37.4944118244606 3.04123153225435
37.6906825809447 2.64526776506781
37.8869533374288 2.24456731917784
38.0832240939129 1.8390848453373
38.279494850397 1.42877507762729
38.4757656068811 1.01359282739492
38.6720363633652 0.593492977432593
38.8683071198493 0.168430476388469
39.0645778763334 -0.261639666601724
39.2608486328175 -0.696762387070827
39.4571193893016 -1.13698257017333
39.6533901457857 -1.58234505544409
39.8496609022698 -2.0328946413706
40.0459316587539 -2.48867608978683
40.242202415238 -2.9497341300957
40.4384731717221 -3.41611346332787
40.6347439282062 -3.88785876604308
40.8310146846903 -4.36501469408061
41.0272854411744 -4.84762588616541
41.2235561976585 -5.33573696737536
41.4198269541426 -5.82939255247561
41.6160977106267 -6.32863724912501
41.8123684671108 -6.83351566096055
42.0086392235949 -7.34407239056379
42.204909980079 -7.86035204231464
42.4011807365631 -8.38239922513693
42.5974514930472 -8.91025855513981
42.7937222495313 -9.44397465815947
};
\addplot [thick, grey127]
table {%
0 0
};
\addplot [thick, bakShock]
table {%
6.91459099223217 105.129993573971
7.11362389588482 104.930720168006
7.31265679953748 104.730054734971
7.51168970319014 104.526733606287
7.7107226068428 104.319611839217
7.90975551049546 104.107643699665
8.10878841414812 103.88986662685
8.30782131780078 103.665387987601
8.50685422145344 103.433374080035
8.70588712510609 103.19304096204
8.90492002875875 102.943646768665
9.10395293241141 102.684485250962
9.30298583606407 102.414880321937
9.50201873971673 102.134181436854
9.70105164336939 101.841759667818
9.90008454702205 101.53700435847
10.0991174506747 101.219320265238
10.2981503543274 100.888125108149
10.49718325798 100.542847467525
10.6962161616327 100.182924973684
10.8952490652853 99.8078027455764
11.094281968938 99.416932041477
11.2933148725907 99.0097690907938
11.4923477762433 98.5857740809189
11.691380679896 98.144410277121
11.8904135835486 97.6851432568386
12.0894464872013 97.2074402425523
12.288479390854 96.7107695197673
12.4875122945066 96.1945999286224
12.6865451981593 95.6584004193015
12.8855781018119 95.1016396628384
13.0846110054646 94.5237857100894
13.2836439091172 93.9243056926672
13.4826768127699 93.3026655604808
13.6817097164226 92.6583298512643
13.8807426200752 91.9907614880994
14.0797755237279 91.2994216014713
14.2788084273805 90.5837693728575
14.4778413310332 89.8432618972401
14.6768742346859 89.0773540622713
14.8759071383385 88.2854984421133
15.0749400419912 87.4671452042257
15.2739729456438 86.6217420275883
15.4730058492965 85.7487340310381
15.6720387529491 84.8475637105601
15.8710716566018 83.9176708845152
16.0701045602545 82.9584926459075
16.2691374639071 81.9694633209059
16.4681703675598 80.9500144329232
16.6672032712124 79.8995746716409
16.8662361748651 78.8175698664369
17.0652690785178 77.703422963737
17.2643019821704 76.5565540078667
17.4633348858231 75.3763801250234
17.6623677894757 74.1623155100389
17.8614006931284 72.913771415631
18.0604335967811 71.630156143882
18.2594665004337 70.3108750397075
18.4584994040864 68.9553304861052
18.657532307739 67.5629219009972
18.8565652113917 66.1330457354962
19.0555981150443 64.6650954734478
19.254631018697 63.1584616321136
19.4536639223497 61.6125317638745
19.6526968260023 60.0266904588477
19.851729729655 58.4003193483191
20.0507626333076 56.7327971089057
20.2497955369603 55.0234994673686
20.448828440613 53.271799206007
20.6478613442656 51.4770661685703
20.8468942479183 49.6386672666302
21.0459271515709 47.7559664863626
21.2449600552236 45.8283248956927
21.4439929588763 43.8551006517608
21.6430258625289 41.8356490086721
21.8420587661816 39.7693223254963
22.0410916698342 37.655470074485
22.2401245734869 35.4934388494812
22.4391574771395 33.2825723744938
22.6381903807922 31.0222115124153
22.8372232844449 28.7116942738623
23.0362561880975 26.3503558261192
23.2352890917502 23.9375285021695
23.4343219954028 21.4725418097988
23.6333548990555 18.9547224407549
23.8323878027082 16.3833942799551
24.0314207063608 13.7578784147254
24.2304536100135 11.0774931440653
24.4294865136661 8.34155398792726
24.6285194173188 5.54937369650088
24.8275523209715 2.70026225949793
25.0265852246241 -0.206473084572139
25.2256181282768 -3.17152783914044
25.4246510319294 -6.19560024033662
25.6236839355821 -9.27939124777143
25.8227168392347 -12.4236045353492
26.0217497428874 -15.6289464821129
26.2207826465401 -18.8961261631266
26.4198155501927 -22.2258553403983
26.6188484538454 -25.6188484538454
};
\addplot [thick, grey127, mark=*, mark size=2, mark options={solid}, only marks]
table {%
5 -47.2376969076908
};
\addplot [thick, grey127, mark=*, mark size=2, mark options={solid}, only marks]
table {%
26.6188484538454 -25.6188484538454
};
\addplot [thick, grey127, mark=*, mark size=2, mark options={solid}, only marks]
table {%
1 0
};
\addplot [thick, grey127, mark=*, mark size=2, mark options={solid}, only marks]
table {%
17.1748737956859 16.1748737956859
};
\addplot [thick, grey127, mark=*, mark size=2, mark options={solid}, only marks]
table {%
42.7937222495313 -9.44397465815947
};
\addplot [thick, grey127, mark=*, mark size=2, mark options={solid}, only marks]
table {%
23.3629173576053 22.3629173576053
};
\addplot [thick, grey127, mark=*, mark size=2, mark options={solid}, only marks]
table {%
6.91459099223217 105.129993573971
};
\draw (axis cs:4,-45) node[
scale=0.8,
  anchor=east,
]{4};
\draw (axis cs:27,-27.4938484538454) node[
scale=0.8,
  anchor=north,
]{0};
\draw (axis cs:1,2.5) node[
 scale=0.8,
 anchor=south,
]{1};
\draw (axis cs:17.1748737956859,18.6748737956859) node[
scale=0.8,
  anchor=south,
]{2};
\draw (axis cs:42,-11.5) node[
scale=0.8,
  anchor=north,
]{3};
\draw (axis cs:20,24) node[
scale=0.8,
  anchor=south west,
]{5};
\draw (axis cs:8.78959099223217,103.254993573971) node[
scale=0.8,
  anchor=north,
 ]{L};
\end{axis}
\end{tikzpicture}
  \end{minipage} \hfill}
  \caption{Periodic cycle of states with repeating wave pattern}
  \label{fig:geng}
\end{figure}

We begin by describing the simplest, unperturbed example in
\cite{BCZ18}, as shown in Figure \ref{fig:geng}; the left panel shows
the proposed wave pattern, while the right panel shows the cycle in
state space.  The solution consists of a strong backward shock moving
into ahead state $w_L$, with a periodic wave interaction pattern
behind the shock.  For this description, we denote the wave spanning
left and right states by $(w_\ell|w_r)$.  The backward (blue) waves
behind the strong shock alternate between a rarefaction $(w_0|w_1)$
and and a shock $(w_5|w_3)$; these interact with the leftmost backward
shock, reflecting forward compressions $(w_5|w_1)$ and rarefactions
$(w_0|w_3)$, respectively.  After crossing other backward waves, these
forward compressions $(w_5|w_1)$ become $(w_3|w_4)$ and collapse into
shocks $(w_2|w_4)$, reflecting backward rarefaction $(w_3|w_2)$.  The
forward shock $(w_2|w_4)$ then interacts with forward rarefaction
$(w_4|w_3)$, reflecting backward compression $(w_2|w_3)$.  This
backward compression collapses into backward shock $(w_1|w_4)$, and
resolving the wave crossings, this pattern can be made to repeat, as
shown in Figure~\ref{fig:geng}.  Note that the characteristic pattern
on the left is approximate, but the cycle in state space, that is in
the $(z,u)$ plane, can be constructed exactly.

In order to realize this cycle in state space, we use the following
notation: states are denoted by $w_\square := (z_\square,u_\square)$,
where $z$ is the Riemann coordinate defined by \eqref{zdef}, and the
wavestrength is given by \eqref{gdstr}.  Forward and backward simple
waves and shocks are described by
\begin{equation}
  \label{gdwvs}
  w_b = W_\pm(w_a,\zeta), \qquad
  w_a = W_\pm^\T(w_b,\zeta), \qquad
  w_b = S_\pm(w_a,\zeta), \qquad
  w_a = S_\pm^\T(w_b,\zeta),
\end{equation}
respectively, where the subscripts $a$ and $b$ refer to the states
ahead of and behind the waves, $\pm$ denotes the sign of the
wavespeed, so that $+$ refers to forward waves and $-$ to backward
waves, and as usual, the strength $\z$ is positive for rarefactions
and negative for compressive waves.  In the $(z,u)$ plane, the simple
wave curves lie along lines of slope $\pm1$, and wave strength is the
change in opposite Riemann invariant,
$\zeta = \big([z] \pm [u]\big)/2$, where $[z]:=z_a-z_b$.  This
notation is shown in Figure~\ref{fig:wvcrv}: the left panel shows the
wave curves in the $(z,u)$ plane with ahead state $w_a$ fixed, and the right
with behind state $w_b$ fixed.

\begin{figure}[thb]
  \centering
\begin{tikzpicture}[scale=0.8]

\begin{groupplot}[group style={group size=2 by 1}]
\pgfplotsset{every linear axis/.append style={xtick=\empty,ytick=\empty}}
\nextgroupplot[
tick pos=left,
xmin=-0.12499895, xmax=2.62499995,
ymin=-2.70367479899657, ymax=2.74779403804746
]
\addplot [very thick, bakSimple]
table {%
0 1
1 0
};
\addplot [very thick, bakSimple]
table {%
1 0
2.5 -1.5
};
\addplot [very thick, bakShock]
table {%
1 -0
1.01515151515152 -0.0151542845904287
1.03030303030303 -0.0303215076157827
1.04545454545455 -0.0455148979213915
1.06060606060606 -0.0607464643104819
1.07575757575758 -0.0760275494158125
1.09090909090909 -0.0913688788457496
1.10606060606061 -0.106780606064652
1.12121212121212 -0.122272353434357
1.13636363636364 -0.13785324979461
1.15151515151515 -0.15353196491819
1.16666666666667 -0.169316741139617
1.18181818181818 -0.185215422424007
1.1969696969697 -0.20123548111419
1.21212121212121 -0.21738404256912
1.22727272727273 -0.233667907884505
1.24242424242424 -0.250093574866991
1.25757575757576 -0.266667257415912
1.27272727272727 -0.283394903451259
1.28787878787879 -0.300282211512819
1.3030303030303 -0.317334646143315
1.31818181818182 -0.334557452157518
1.33333333333333 -0.351955667889632
1.34848484848485 -0.369534137502604
1.36363636363636 -0.387297522435278
1.37878787878788 -0.40525031205635
1.39393939393939 -0.423396833587863
1.40909090909091 -0.441741261355383
1.42424242424242 -0.460287625416944
1.43939393939394 -0.479039819618302
1.45454545454545 -0.498001609117953
1.46969696969697 -0.517176637421646
1.48484848484848 -0.536568432962769
1.5 -0.556180415261979
1.51515151515152 -0.576015900696641
1.53030303030303 -0.596078107908188
1.54545454545455 -0.616370162873217
1.56060606060606 -0.636895103662079
1.57575757575758 -0.65765588490683
1.59090909090909 -0.678655381998702
1.60606060606061 -0.699896395033681
1.62121212121212 -0.721381652523353
1.63636363636364 -0.743113814886861
1.65151515151515 -0.765095477738636
1.66666666666667 -0.787329174985447
1.68181818181818 -0.80981738174533
1.6969696969697 -0.832562517100009
1.71212121212121 -0.85556694669161
1.72727272727273 -0.878832985173645
1.74242424242424 -0.902362898525566
1.75757575757576 -0.926158906239507
1.77272727272727 -0.950223183387236
1.78787878787879 -0.974557862574762
1.8030303030303 -0.999165035791546
1.81818181818182 -1.02404675616077
1.83333333333333 -1.04920503959671
1.84848484848485 -1.07464186637475
1.86363636363636 -1.10035918261941
1.87878787878788 -1.12635890171507
1.89393939393939 -1.15264290564423
1.90909090909091 -1.17921304625725
1.92424242424242 -1.20607114647779
1.93939393939394 -1.23321900144763
1.95454545454545 -1.26065837961425
1.96969696969697 -1.28839102376454
1.98484848484848 -1.31641865200767
2 -1.34474295870995
2.01515151515152 -1.37336561538443
2.03030303030303 -1.40228827153765
2.04545454545455 -1.43151255547601
2.06060606060606 -1.46104007507396
2.07575757575758 -1.49087241850596
2.09090909090909 -1.52101115494438
2.10606060606061 -1.55145783522498
2.12121212121212 -1.58221399248182
2.13636363636364 -1.61328114275322
2.15151515151515 -1.64466078556014
2.16666666666667 -1.67635440445869
2.18181818181818 -1.70836346756787
2.1969696969697 -1.74068942807399
2.21212121212121 -1.77333372471281
2.22727272727273 -1.80629778223078
2.24242424242424 -1.83958301182611
2.25757575757576 -1.87319081157102
2.27272727272727 -1.9071225668159
2.28787878787879 -1.94137965057632
2.3030303030303 -1.97596342390377
2.31818181818182 -2.01087523624103
2.33333333333333 -2.04611642576261
2.34848484848485 -2.08168831970141
2.36363636363636 -2.11759223466188
2.37878787878788 -2.15382947692058
2.39393939393939 -2.19040134271464
2.40909090909091 -2.2273091185187
2.42424242424242 -2.2645540813109
2.43939393939394 -2.3021374988284
2.45454545454545 -2.34006062981294
2.46969696969697 -2.3783247242469
2.48484848484848 -2.4169310235803
2.5 -2.45588076094911
};
\addplot [very thick, fwdSimple]
table {%
0 -1
1 -0
};
\addplot [very thick, fwdSimple]
table {%
  1 -0
  2.5 1.5
};
\addplot [very thick, fwdShock]
table {%
1 0
1.01515151515152 0.0151542845904287
1.03030303030303 0.0303215076157827
1.04545454545455 0.0455148979213915
1.06060606060606 0.0607464643104819
1.07575757575758 0.0760275494158125
1.09090909090909 0.0913688788457496
1.10606060606061 0.106780606064652
1.12121212121212 0.122272353434357
1.13636363636364 0.13785324979461
1.15151515151515 0.15353196491819
1.16666666666667 0.169316741139617
1.18181818181818 0.185215422424007
1.1969696969697 0.20123548111419
1.21212121212121 0.21738404256912
1.22727272727273 0.233667907884505
1.24242424242424 0.250093574866991
1.25757575757576 0.266667257415912
1.27272727272727 0.283394903451259
1.28787878787879 0.300282211512819
1.3030303030303 0.317334646143315
1.31818181818182 0.334557452157518
1.33333333333333 0.351955667889632
1.34848484848485 0.369534137502604
1.36363636363636 0.387297522435278
1.37878787878788 0.40525031205635
1.39393939393939 0.423396833587863
1.40909090909091 0.441741261355383
1.42424242424242 0.460287625416944
1.43939393939394 0.479039819618302
1.45454545454545 0.498001609117953
1.46969696969697 0.517176637421646
1.48484848484848 0.536568432962769
1.5 0.556180415261979
1.51515151515152 0.576015900696641
1.53030303030303 0.596078107908188
1.54545454545455 0.616370162873217
1.56060606060606 0.636895103662079
1.57575757575758 0.65765588490683
1.59090909090909 0.678655381998702
1.60606060606061 0.699896395033681
1.62121212121212 0.721381652523353
1.63636363636364 0.743113814886861
1.65151515151515 0.765095477738636
1.66666666666667 0.787329174985447
1.68181818181818 0.80981738174533
1.6969696969697 0.832562517100009
1.71212121212121 0.85556694669161
1.72727272727273 0.878832985173645
1.74242424242424 0.902362898525566
1.75757575757576 0.926158906239507
1.77272727272727 0.950223183387236
1.78787878787879 0.974557862574762
1.8030303030303 0.999165035791546
1.81818181818182 1.02404675616077
1.83333333333333 1.04920503959671
1.84848484848485 1.07464186637475
1.86363636363636 1.10035918261941
1.87878787878788 1.12635890171507
1.89393939393939 1.15264290564423
1.90909090909091 1.17921304625725
1.92424242424242 1.20607114647779
1.93939393939394 1.23321900144763
1.95454545454545 1.26065837961425
1.96969696969697 1.28839102376454
1.98484848484848 1.31641865200767
2 1.34474295870995
2.01515151515152 1.37336561538443
2.03030303030303 1.40228827153765
2.04545454545455 1.43151255547601
2.06060606060606 1.46104007507396
2.07575757575758 1.49087241850596
2.09090909090909 1.52101115494438
2.10606060606061 1.55145783522498
2.12121212121212 1.58221399248182
2.13636363636364 1.61328114275322
2.15151515151515 1.64466078556014
2.16666666666667 1.67635440445869
2.18181818181818 1.70836346756787
2.1969696969697 1.74068942807399
2.21212121212121 1.77333372471281
2.22727272727273 1.80629778223078
2.24242424242424 1.83958301182611
2.25757575757576 1.87319081157102
2.27272727272727 1.9071225668159
2.28787878787879 1.94137965057632
2.3030303030303 1.97596342390377
2.31818181818182 2.01087523624103
2.33333333333333 2.04611642576261
2.34848484848485 2.08168831970141
2.36363636363636 2.11759223466188
2.37878787878788 2.15382947692058
2.39393939393939 2.19040134271464
2.40909090909091 2.2273091185187
2.42424242424242 2.2645540813109
2.43939393939394 2.3021374988284
2.45454545454545 2.34006062981294
2.46969696969697 2.3783247242469
2.48484848484848 2.4169310235803
2.5 2.45588076094911
};
\addplot [thick, grey127]
table {%
1 2.5
2.5 1
};
\addplot [thick, grey127, mark=*, mark size=3, mark options={solid}, only marks]
table {%
1 0
};
\draw (axis cs:1,0) node[
  scale=1.2,
  anchor=south,
]{$w_a$};
\draw (axis cs:0.8,-0.6) node[
  scale=1.2,
  anchor=north,
]{$W_+(w_a,\z)$};
\draw (axis cs:2.2,1) node[
  scale=1.2,
  anchor=north,
]{$W_+(w_a,\z)$};
\draw (axis cs:2.2,-1) node[
  scale=1.2,
  anchor=south,
]{$W_-(w_a,\z)$};
\draw (axis cs:2,2) node[
  scale=1.2,
  anchor=south,
]{$S_+(w_a,\z)$};
\draw (axis cs:2,-2) node[
  scale=1.2,
  anchor=north,
]{$S_-(w_a,\z)$};
\draw (axis cs:0.6,1.8) node[
  scale=1.2,
  anchor=south,
]{$\z<0$};
\draw [decorate,very thick,grey127,
decoration = {brace}]
(0.1,1.1) -- (1.2,2.2);
\addplot [black,->]
table {%
0 -2.5
0 -1.2
};
\addplot [black,->]
table {%
0 -2.5
0.7 -2.5
};
\draw (axis cs:0.7,-2.5) node[
scale=1.2,
  anchor=south
]{$z$};
\draw (axis cs:0,-1.2) node[
scale=1.2,
  anchor=west
]{$u$};

\nextgroupplot[
tick pos=left,
xmin=0.07, xmax=2.93,
ymin=-3.16218528868882, ymax=3.16218528868882
]
\addplot [very thick, bakSimple]
table {%
0.2 0.8
1 0
};
\addplot [very thick, bakSimple]
table {%
1 0
2.8 -1.8
};
\addplot [very thick, fwdShock]
table {%
0.2 -2.87471389880802
0.208080808080808 -2.70741786090903
0.216161616161616 -2.55551537887976
0.224242424242424 -2.41707987982738
0.232323232323232 -2.29048360002354
0.24040404040404 -2.17434281873088
0.248484848484849 -2.06747457125086
0.256565656565657 -1.96886215179676
0.264646464646465 -1.87762740923022
0.272727272727273 -1.79300833774489
0.280808080808081 -1.71434082822735
0.288888888888889 -1.64104371374672
0.296969696969697 -1.57260644163189
0.305050505050505 -1.50857885386786
0.313131313131313 -1.44856267046615
0.321212121212121 -1.39220435657641
0.329292929292929 -1.33918912027147
0.337373737373737 -1.28923583914204
0.345454545454545 -1.24209275373069
0.353535353535354 -1.19753379711798
0.361616161616162 -1.15535545465338
0.36969696969697 -1.11537406740832
0.377777777777778 -1.07742350855501
0.385858585858586 -1.04135317440942
0.393939393939394 -1.00702624198159
0.402020202020202 -0.974318153061545
0.41010101010101 -0.943115291530309
0.418181818181818 -0.913313826029242
0.426262626262626 -0.884818694589756
0.434343434343434 -0.857542711507781
0.442424242424242 -0.831405779793935
0.450505050505051 -0.806334195060131
0.458585858585859 -0.782260028811558
0.466666666666667 -0.759120580875933
0.474747474747475 -0.736857892181023
0.482828282828283 -0.715418310336437
0.490909090909091 -0.694752101526749
0.498989898989899 -0.674813103113101
0.507070707070707 -0.655558412096285
0.515151515151515 -0.636948105238019
0.523232323232323 -0.618944987186687
0.531313131313131 -0.601514363424407
0.539393939393939 -0.58462383525603
0.547474747474747 -0.56824311440812
0.555555555555556 -0.552343855105537
0.563636363636364 -0.536899501752144
0.571717171717172 -0.521885150566483
0.57979797979798 -0.507277423717936
0.587878787878788 -0.493054354678265
0.595959595959596 -0.479195283651038
0.604040404040404 -0.465680762070383
0.612121212121212 -0.452492465273321
0.62020202020202 -0.439613112548838
0.628282828282828 -0.42702639385369
0.636363636363636 -0.41471690256137
0.644444444444444 -0.402670073677981
0.652525252525253 -0.390872127018224
0.660606060606061 -0.379310014887244
0.668686868686869 -0.367971373860636
0.676767676767677 -0.356844480296194
0.684848484848485 -0.345918209247618
0.692929292929293 -0.335181996483066
0.701010101010101 -0.324625803340433
0.709090909090909 -0.314240084177208
0.717171717171717 -0.304015756195866
0.725252525252525 -0.293944171446504
0.733333333333333 -0.284017090826938
0.741414141414142 -0.274226659917126
0.749494949494949 -0.264565386499715
0.757575757575758 -0.255026119631928
0.765656565656566 -0.245602030146083
0.773737373737374 -0.236286592466918
0.781818181818182 -0.227073567643729
0.78989898989899 -0.217956987504167
0.797979797979798 -0.20893113984459
0.806060606060606 -0.19999055457906
0.814141414141414 -0.191129990775722
0.822222222222222 -0.182344424515181
0.83030303030303 -0.173629037510958
0.838383838383838 -0.164979206436986
0.846464646464646 -0.156390492911581
0.854545454545454 -0.147858634091402
0.862626262626263 -0.139379533832601
0.870707070707071 -0.130949254379744
0.878787878787879 -0.122564008546194
0.886868686868687 -0.114220152352417
0.894949494949495 -0.105914178091283
0.903030303030303 -0.0976427077917872
0.911111111111111 -0.0894024870547509
0.919191919191919 -0.0811903792360578
0.927272727272727 -0.0730033599547934
0.935353535353535 -0.0648385119053138
0.943434343434344 -0.0566930199537963
0.951515151515151 -0.048564166501238
0.95959595959596 -0.0404493270961482
0.967676767676768 -0.0323459662813726
0.975757575757576 -0.0242516336605421
0.983838383838384 -0.0161639601704847
0.991919191919192 -0.0080806545451501
1 -0
};
\addplot [very thick, fwdSimple]
table {%
0.2 -0.8
1 -0
};
\addplot [very thick, fwdSimple]
table {%
1 -0
2.8 1.8
};
\addplot [very thick, bakShock]
table {%
0.2 2.87471389880802
0.208080808080808 2.70741786090903
0.216161616161616 2.55551537887976
0.224242424242424 2.41707987982738
0.232323232323232 2.29048360002354
0.24040404040404 2.17434281873088
0.248484848484849 2.06747457125086
0.256565656565657 1.96886215179676
0.264646464646465 1.87762740923022
0.272727272727273 1.79300833774489
0.280808080808081 1.71434082822735
0.288888888888889 1.64104371374672
0.296969696969697 1.57260644163189
0.305050505050505 1.50857885386786
0.313131313131313 1.44856267046615
0.321212121212121 1.39220435657641
0.329292929292929 1.33918912027147
0.337373737373737 1.28923583914204
0.345454545454545 1.24209275373069
0.353535353535354 1.19753379711798
0.361616161616162 1.15535545465338
0.36969696969697 1.11537406740832
0.377777777777778 1.07742350855501
0.385858585858586 1.04135317440942
0.393939393939394 1.00702624198159
0.402020202020202 0.974318153061545
0.41010101010101 0.943115291530309
0.418181818181818 0.913313826029242
0.426262626262626 0.884818694589756
0.434343434343434 0.857542711507781
0.442424242424242 0.831405779793935
0.450505050505051 0.806334195060131
0.458585858585859 0.782260028811558
0.466666666666667 0.759120580875933
0.474747474747475 0.736857892181023
0.482828282828283 0.715418310336437
0.490909090909091 0.694752101526749
0.498989898989899 0.674813103113101
0.507070707070707 0.655558412096285
0.515151515151515 0.636948105238019
0.523232323232323 0.618944987186687
0.531313131313131 0.601514363424407
0.539393939393939 0.58462383525603
0.547474747474747 0.56824311440812
0.555555555555556 0.552343855105537
0.563636363636364 0.536899501752144
0.571717171717172 0.521885150566483
0.57979797979798 0.507277423717936
0.587878787878788 0.493054354678265
0.595959595959596 0.479195283651038
0.604040404040404 0.465680762070383
0.612121212121212 0.452492465273321
0.62020202020202 0.439613112548838
0.628282828282828 0.42702639385369
0.636363636363636 0.41471690256137
0.644444444444444 0.402670073677981
0.652525252525253 0.390872127018224
0.660606060606061 0.379310014887244
0.668686868686869 0.367971373860636
0.676767676767677 0.356844480296194
0.684848484848485 0.345918209247618
0.692929292929293 0.335181996483066
0.701010101010101 0.324625803340433
0.709090909090909 0.314240084177208
0.717171717171717 0.304015756195866
0.725252525252525 0.293944171446504
0.733333333333333 0.284017090826938
0.741414141414142 0.274226659917126
0.749494949494949 0.264565386499715
0.757575757575758 0.255026119631928
0.765656565656566 0.245602030146083
0.773737373737374 0.236286592466918
0.781818181818182 0.227073567643729
0.78989898989899 0.217956987504167
0.797979797979798 0.20893113984459
0.806060606060606 0.19999055457906
0.814141414141414 0.191129990775722
0.822222222222222 0.182344424515181
0.83030303030303 0.173629037510958
0.838383838383838 0.164979206436986
0.846464646464646 0.156390492911581
0.854545454545454 0.147858634091402
0.862626262626263 0.139379533832601
0.870707070707071 0.130949254379744
0.878787878787879 0.122564008546194
0.886868686868687 0.114220152352417
0.894949494949495 0.105914178091283
0.903030303030303 0.0976427077917872
0.911111111111111 0.0894024870547509
0.919191919191919 0.0811903792360578
0.927272727272727 0.0730033599547934
0.935353535353535 0.0648385119053138
0.943434343434344 0.0566930199537963
0.951515151515151 0.048564166501238
0.95959595959596 0.0404493270961482
0.967676767676768 0.0323459662813726
0.975757575757576 0.0242516336605421
0.983838383838384 0.0161639601704847
0.991919191919192 0.0080806545451501
1 0
};
\addplot [thick, grey127]
table {%
0.2 -0.7
2 -2.5
};
\addplot [thick, grey127, mark=*, mark size=3, mark options={solid}, only marks]
table {%
1 0
};
\draw (axis cs:1,0) node[
  scale=1.2,
  anchor=south,
]{$w_b$};
\draw (axis cs:2.2,1.3) node[
  scale=1.2,
  anchor=south,
]{$W_+^\T(w_b,\z)$};
\draw (axis cs:2.2,-1.1) node[
  scale=1.2,
  anchor=south,
]{$W_-^\T(w_b,\z)$};
\draw (axis cs:0.2,-2.5) node[
  scale=1.2,
  anchor=west,
]{$S_+^\T(w_b,\z)$};
\draw (axis cs:0.2,2.5) node[
  scale=1.2,
  anchor=west,
]{$S_-^\T(w_b,\z)$};
\draw (axis cs:2.1,-1.7) node[
  scale=1.2,
  anchor=north,
]{$\z<0$};
\draw [decorate,very thick,grey127,
decoration = {brace}]
(2.3,-1.4) -- (1.7,-2.0);
\end{groupplot}

\end{tikzpicture}

  \caption{Wave curves and strengths}
  \label{fig:wvcrv}
\end{figure}

With this notation in place, and referring to Figures~\ref{fig:geng}
and~\ref{fig:cycle}, we describe the cycle of states behind the
persistent shock wave, as follows.  Fix (absolute) strength $\z>0$, to
be regarded as a parameter, and begin with state $w_0$.  Then define
the remaining states in turn by
\begin{equation}
  \label{cycle}
  \begin{aligned}
    w_1 &:= W_-\big(w_0,\z\big), \\
    w_4 &:= S_-\big(w_1,-\z\big), \\
    w_2 &:= S_+\big(w_4,-\beta\big), \\
    w_3 &:= W_-^\T\big(w_2,\z\big), \\
    w_5 &:= S_-^\T\big(w_3,-\z\big),    
  \end{aligned}
\end{equation}
where $-\beta$ is a shock strength chosen so that $w_2$, $w_1$ and
$w_5$ all lie on the same forward simple wave curve.  Because we wish
to allow perturbations of this cycle, we first give conditions under
which it can be closed.

\begin{lemma}
  \label{lem:cycle}
  Given any reference state $w_0$ and absolute strength
  $\z\in(0,z_0)$, define states $w_1, \dots, w_5$ by \eqref{cycle}.
  The condition $z_5<z_0$ is necessary and sufficient for this cycle
  to be closed to yield the approximate wave pattern of
  Figure~\ref{fig:geng}.  Moreover, for each $z_0>0$ there is a
  critical value
  \begin{equation}
    \label{z*}
    \z_* := \z_*(z_0) = O(1), \qquad 0<\z_*<z_0,
  \end{equation}
  such that whenever
  \[
    \z < \z_*(z_0), \qcom{we have} z_5 > z_0,
  \]
  so that the cycle can be completed \emph{only} if the wave strength $\z$ is
  \emph{large enough}.
\end{lemma}

\begin{proof}
  According to Figure~\ref{fig:geng}, the periodic cycle can be
  completed and drawn provided we can find an ahead state $w_L$ so
  that states $w_0$ and $w_5$ can both be connected as behind states
  by a backwards shock with ahead state $w_L$, that is if both
  \begin{equation}
    \label{leq}
    w_5 = S_-\big(w_L,-\al\z\big) \qcom{and}
    w_0 = S_-\big(w_L,-(1+\al)\z\big),
  \end{equation}
  for the appropriate value of parameter $\al$.  We must have $\al>0$
  and generally we expect $\al$ to be large.  Because $w_0$ and $w_5$
  are fully determined by $\z$ and $[u]$ appears linearly in all
  descriptions of waves, we can regard \eqref{leq} as two nonlinear
  equations in two variables, namely $z_L$ and $\al$.  By
  monotonicity, \eqref{leq} will have a unique solution provided the
  targets are in the appropriate range: since $\z>0$, the wave to
  $w_0$ is stronger, and so we will be able to find state $z_L$ as
  long as $z_0>z_5>z_L$, which is the stated condition.  Notice that
  if $z_5\sim z_0$, we will get $z_L\sim 0$, which is the statement
  that the shock curves approach vertical as ahead state $w_L$ nears
  vacuum.

  The (abbreviated) state cycle is shown in Figure~\ref{fig:cycle}:
  states are labeled as in Figure~\ref{fig:geng}, but the extreme left
  state $w_L$ is not shown.  In the first panel, the cycle is shown
  for small values of $\z$, and it is clear that in this picture we
  have $z_5>z_0$.  In the middle panel, $\z$ is relatively large and
  we have $z_5<z_0$, so that $w_L$ can be found.  The third panel
  shows $z_5-z_0$ as a function of $\z$, and the bifurcation point is
  clearly seen to be far from the origin.

  \begin{figure}[hbt]
    \centering

    
    \caption{Abbreviated wave cycle}
    \label{fig:cycle}
  \end{figure}

  To find $\z_*(z_0)$, we can solve the equation $z_5(\z) = z_0$ for
  $\z$, which is the bifurcation point.  Rather than doing this
  explicitly, we argue asymptotically.  That is, we show that for
  small values of $\z$, we have $z_5(\zeta)\ll z_0$, so that there is
  no ahead state $w_L$ unless $\z$ is large enough: that is, there is
  some minimum value $\z_*(z_0) = O(z_0)$ below which $z_5>z_0$.

  It thus suffices to consider small values of $\z$.  It is well known
  that for small wave strengths, the shock and simple wave curves
  agree up to errors which are cubic in wave strength, that is,
  \[
    S_\pm(w_a,\xi) = W_\pm(w_a,\xi) + O\big(|\xi|^3\big), \qquad
    S_\pm^\T(w_b,\xi) = W_\pm^\T(w_b,\xi) + O\big(|\xi|^3\big),
  \]
  for small values of $\xi<0$.  Since wave strength is change in
  opposite Riemann invariant, it follows that for small $\z$, the
  strengths $\z_{ab}$ of the reflected simple waves in Figure
  \ref{fig:cycle} satisfy
  \[
    \z_{23} = - O\big(|\z_{42}|^3\big), \qquad
    \z_{04} =  O\big(|\z_{14}|^3\big), \qcom{and}
    \z_{25} = O\big(|\z_{23}|^3\big),
  \]
  respectively, where each of these individual interactions is shown in
  Figure \ref{fig:smints}.
  
  \begin{figure}[bht]
    \centering
\begin{tikzpicture}[scale=0.6]

\begin{groupplot}[group style={group size=3 by 1}]
\nextgroupplot[
hide x axis,
hide y axis,
tick align=outside,
tick pos=left,
xmin=-2.715, xmax=2.015,
ymin=-1.1, ymax=1.1,
]
\path [draw=fwdSimple, fill=fwdSimple, opacity=0.3]
(axis cs:0,0)
--(axis cs:0,0)
--(axis cs:-2.5,-1)
--(axis cs:-1.3,-1)
--(axis cs:-1.3,-1)
--(axis cs:0,0)
--cycle;

\path [draw=bakSimple, fill=bakSimple, opacity=0.3]
(axis cs:0,0)
--(axis cs:0,0)
--(axis cs:-1.8,1)
--(axis cs:-2.5,1)
--(axis cs:-2.5,1)
--(axis cs:0,0)
--cycle;

\addplot [very thick, fwdShock]
table {%
0 0
1.8 1
};
\addplot [thick, grey127, mark=*, mark size=3, mark options={solid}, only marks]
table {%
0 0
};
\draw (axis cs:0.85,-0.2) node[
  scale=1.5,
  anchor=base west,
  text=black,
  rotate=0.0
]{\fbox{$w_4$}};
\draw (axis cs:-1.5,0) node[
  scale=1.5,
  anchor=base west,
  text=black,
  rotate=0.0
]{\fbox{$w_3$}};
\draw (axis cs:-0.2,0.6) node[
  scale=1.5,
  anchor=base west,
  text=black,
  rotate=0.0
]{\fbox{$w_2$}};
\draw (axis cs:1,0.5) node[
  scale=1.5,
  anchor=base west,
  text=black,
  rotate=0.0
]{$-\beta$};
\draw (axis cs:-1.5,-0.55) node[
  scale=1.5,
  anchor=base east,
  text=black,
  rotate=0.0
]{$-\beta$};
\draw (axis cs:-1.5,0.45) node[
  scale=1.5,
  anchor=base east,
  text=black,
  rotate=0.0
]{$O(\beta^3)$};

\nextgroupplot[
hide x axis,
hide y axis,
tick align=outside,
tick pos=left,
xmin=-1.805, xmax=2.705,
ymin=-1.1, ymax=1.1,
]
\path [draw=bakSimple, fill=bakSimple, opacity=0.3]
(axis cs:0,0)
--(axis cs:0,0)
--(axis cs:2.5,-1)
--(axis cs:1.4,-1)
--(axis cs:0,0)
--cycle;

\path [draw=fwdSimple, fill=fwdSimple, opacity=0.3]
(axis cs:0,0)
--(axis cs:0,0)
--(axis cs:1.6,1)
--(axis cs:2.1,1)
--(axis cs:0,0)
--cycle;

\addplot [very thick, bakShock]
table {%
0 0
-1.6 1
};
\addplot [thick, grey127, mark=*, mark size=3, mark options={solid}, only marks]
table {%
0 0
};
\draw (axis cs:1.5,0) node[
  scale=1.5,
  anchor=base east,
  text=black,
  rotate=0.0
]{\fbox{$w_0$}};
\draw (axis cs:-0.7,-0.2) node[
  scale=1.5,
  anchor=base east,
  text=black,
  rotate=0.0
]{\fbox{$w_1$}};
\draw (axis cs:0,0.6) node[
  scale=1.5,
  text=black,
  rotate=0.0
]{\fbox{$w_4$}};
\draw (axis cs:-0.95,0.45) node[
  scale=1.5,
  anchor=base east,
  text=black,
  rotate=0.0
]{$-\zeta$};
\draw (axis cs:1.95,-0.5) node[
  scale=1.5,
  anchor=base east,
  text=black,
  rotate=0.0
]{$-\zeta$};
\draw (axis cs:2.25,0.35) node[
  scale=1.5,
  anchor=base east,
  text=black,
  rotate=0.0
]{$O(\zeta^3)$};
\addplot [black,->]
table {%
-1.8 -1
-1.8 -0.5
};
\addplot [black,->]
table {%
-1.8 -1
-1 -1
};
\draw (axis cs:-1.8,-0.5) node[
scale=1.2,
  anchor=west
]{$t$};
\draw (axis cs:-1,-1) node[
scale=1.2,
  anchor=south
]{$x$};

\nextgroupplot[
hide x axis,
hide y axis,
tick align=outside,
tick pos=left,
xmin=-2.74, xmax=2.54,
ymin=-1.1, ymax=1.1,
]
\path [draw=bakSimple, fill=bakSimple, opacity=0.3]
(axis cs:0,0)
--(axis cs:0,0)
--(axis cs:2.5,-1)
--(axis cs:1.4,-1)
--(axis cs:0,0)
--cycle;

\path [draw=fwdSimple, fill=fwdSimple, opacity=0.3]
(axis cs:0,0)
--(axis cs:0,0)
--(axis cs:-1.8,-1)
--(axis cs:-2.3,-1)
--(axis cs:0,0)
--cycle;

\addplot [very thick, bakShock]
table {%
0 0
-1.6 1
};
\addplot [thick, grey127, mark=*, mark size=3, mark options={solid}, only marks]
table {%
0 0
};
\draw (axis cs:1.5,0.4) node[
  scale=1.5,
  anchor=base east,
  text=black,
  rotate=0.0
]{\fbox{$w_3$}};
\draw (axis cs:-1,0) node[
  scale=1.5,
  anchor=base east,
  text=black,
  rotate=0.0
]{\fbox{$w_5$}};
\draw (axis cs:0.3,-0.6) node[
  scale=1.5,
  anchor=base east,
  text=black,
  rotate=0.0
]{\fbox{$w_2$}};
\draw (axis cs:-1.1,0.6) node[
  scale=1.5,
  anchor=base east,
  text=black,
  rotate=0.0
]{$-\zeta$};
\draw (axis cs:1.95,-0.45) node[
  scale=1.5,
  anchor=base east,
  text=black,
  rotate=0.0
]{$-\zeta$};
\draw (axis cs:-1.3,-0.62) node[
  scale=1.5,
  anchor=base east,
  text=black,
  rotate=0.0
]{$O(\zeta^3)$};
\end{groupplot}

\end{tikzpicture}
    \caption{Individual wave interactions}
    \label{fig:smints}
  \end{figure}

  We now equate these wave strengths according to the cycle: these in
  turn become
  \[
    \z_{23} = \z = O\big(|\beta|^3\big), \qcom{and}
    \z_{04} =  \z_{25} = O\big(|\z|^3\big).
  \]
  It follows that for weak waves, $\beta = O\big(\z^{1/3}\big)$.
  Finally, comparing Riemann invariants of the various states, we have
  \[
    \begin{aligned}
      u_5 - u_2 &= z_5 - z_2 = \z_{25} > 0,\\
      u_3 - u_2 &= z_2 - z_3 = \z_{23} < 0,\\
      u_3 - u_4 &= z_3 - z_4 = \beta > 0,\\
      u_0 - u_4 &= z_0 - z_4 = \z_{04} > 0,
    \end{aligned}
  \]
  and combining these yields the identity
  \[
    z_5 - z_0 = \z_{25} + \zeta_{23} + \beta - \z_{04} > 0,
  \]
  in which $\beta = O\big(|\z|^{1/3}\big)$ is clearly the dominant
  term while $\z$ remains small.  In fact is is clear that for the
  equality $z_5=z_0$ to hold, we require $\z\sim\beta$, which holds
  only for $\z=O(1)$.
\end{proof}

We have shown that the cycle in state space is possible only if the
simple wave strength $\z$ is sufficiently large.  We now consider the
mFT approximation of a solution with data chosen from this cycle.  As
we have noted, we regard an mFT approximation as accurate only if the
residual is sufficiently small, for discretization parameter $\e>0$.
It follows that if we attempt to model the data which includes simple
waves of strength $\z=O(1)$, we must discretize each such simple wave
into many smaller waves of strength $O(\e)$.  This in turn means that
the simple wave interactions in Figure \ref{fig:geng}, specifically
the crossing of rarefactions $\z_{30}$ and $\z_{32}$ and the merge of
shock $\z_{42}$ with rarefaction $\z_{34}$ will be spread out over
many interactions with weak simple waves.  In particular, the
``cancellation'' of the shock $\z_{42}$ with the rarefaction $\z_{34}$
will take infinitely long, and the characteristic pattern as drawn in
Figure \ref{fig:geng} which shows apparent periodicity cannot be drawn
without a very large residual.  We state this as a corollary.

\begin{cor}
  \label{cor:breakdown}
  Wave patterns which are a perturbation of that drawn in Figure
  \ref{fig:geng}, and which perturb to solutions having infinite
  variation in finite time, are inconsistent with the mFT scheme.
\end{cor}

The critical strength $\z_*$ is large relative to a fixed ahead state,
say $z_0$.  For the prototypical case of a $\g$-law gas, $\z_*$ scales
directly with $z_0$, and for a physically reasonable value of $\g$,
say $\g\approx 1.4$, we have $\z_*(z_0)\gtrapprox 4\,z_0$.  This in
turn implies that the Mach number of each shock as drawn in
Figure~\ref{fig:geng} is of the order of $M\gtrapprox 30$, so these
shocks are extremely strong.  The (combined) simple waves in the
solution must have strengths of the same order, and it is then clear
that these cannot be reasonably approximated by a small number of jump
discontinuities.

\begin{figure}[thb]
  \centering
  \includegraphics[height=2.5in]{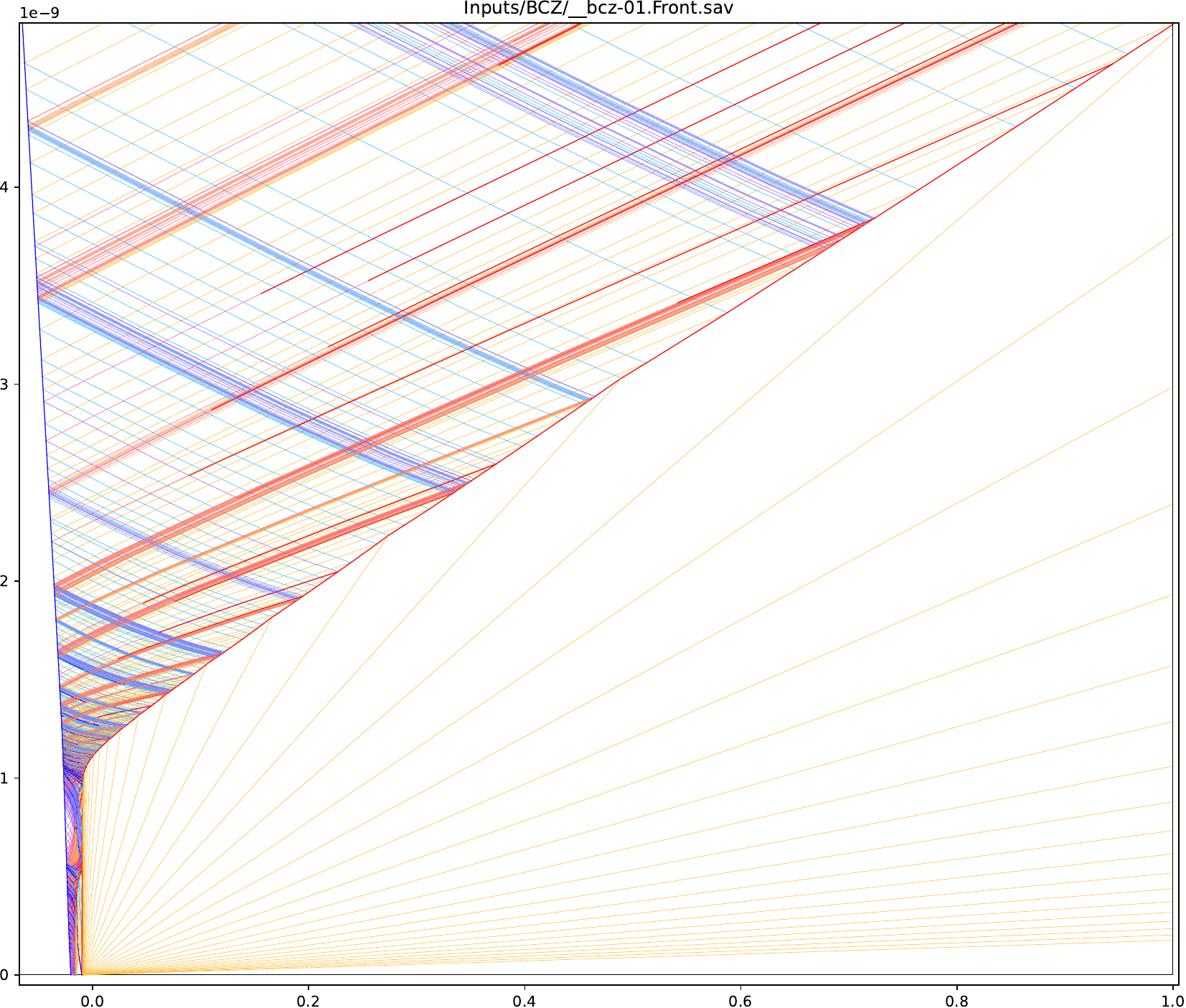}\quad
  \includegraphics[height=2.5in]{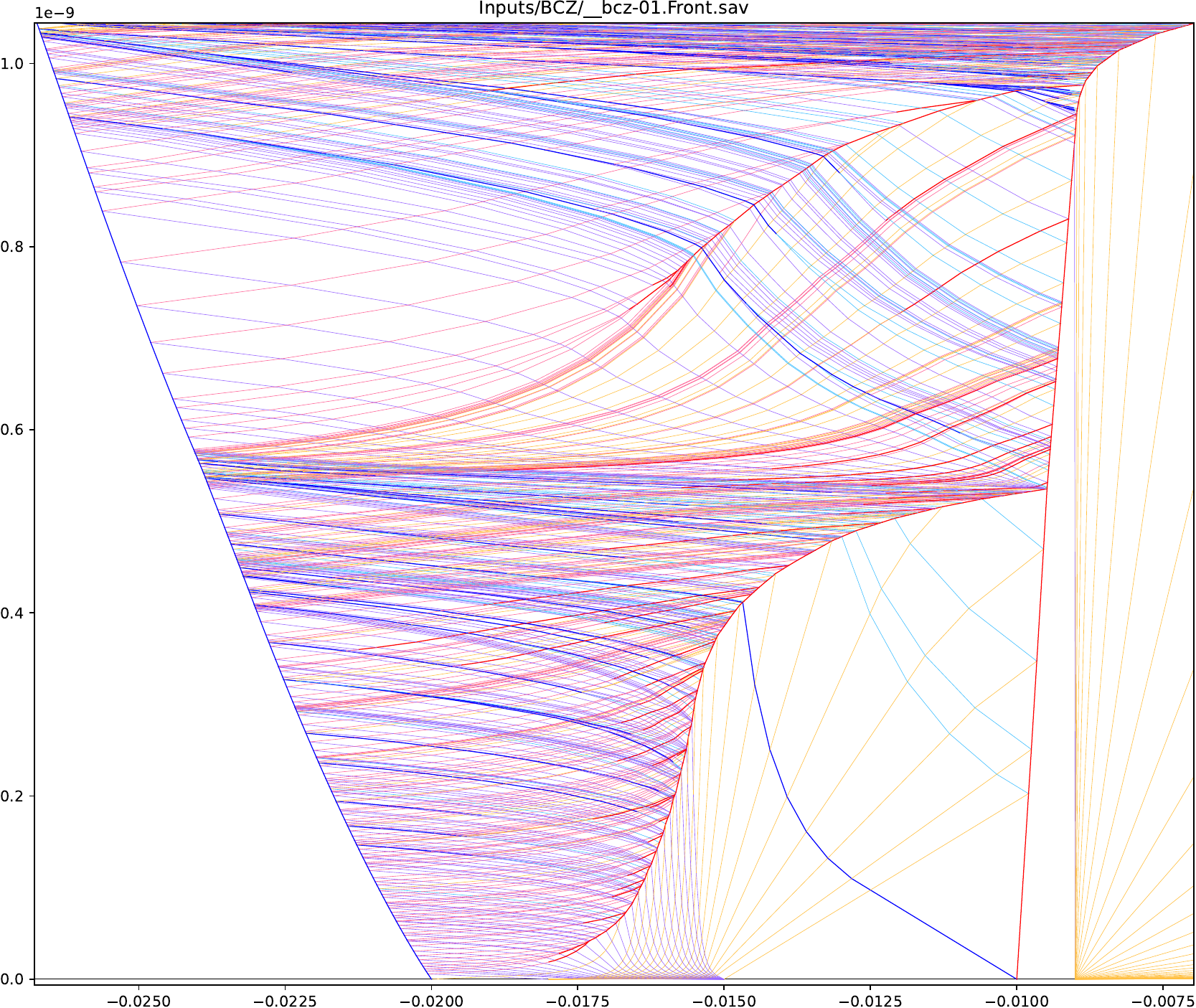}
  \caption{FT simulation of cyclic data}
  \label{fig:simul}
\end{figure}

We conclude that although the periodic cycle of states can be
generated, and initial data based on the corresponding wave pattern
will generate solutions in which the total variation grows, the actual
solution will not generate a localized wave pattern in which
interactions concentrate.  Rather, the interaction of many weak simple
waves takes up both time and space, and the corresponding solutions
will spread out rapidly, with corresponding decay due to cancellation
between shocks and rarefactions.  As an illustration, in
Figure~\ref{fig:simul}, we show a ``front plot'' of the waves
generated by our implementation of the mFT code, including
compressions, simulating the cyclic Cauchy data described above: the
right panel is a magnified piece of the left panel.  Here red and blue
waves are compressive, with the darker, thicker lines being shocks,
while the lighter thinner lines are compressions, while orange and
magenta lines are rarefactions.  The spread of the pattern begins
around $t\approx 0.002$, at which the rightmost forward shock and
rarefaction meet and interact; after this the wave pattern spreads out
rapidly in space, and we expect no accumulation of interaction times,
while there are clear indications of oscillatory behavior in both
families.  Finally we note that the (dimensionless) wave speeds in
this picture are very large: the leading forward shock has speed on
the order of $30/0.004\approx7500$, while the leading edge of the
outside rarefaction is of the order $40/0.0004\approx 10^5$; this is
consistent with the very large Mach numbers mentioned above.

\section*{Data Availability Statement}

Codes for the generation of all pictures except Figure~\ref{fig:simul}
can be accessed publicly on the web at
\url{http://gitlab.com/mbhatnagar011/weakstar3_mft/}.  The Front
Tracking and visualization code for Figure~\ref{fig:simul} will be
published as open source code when fully developed; until then, the
code is available from either author upon request.

\bibliographystyle{plain}

\end{document}